\documentclass[12pt]{article}%
\usepackage[margin=1in]{geometry}
\usepackage{amsmath,amsfonts,amsthm,amssymb}
\usepackage{color}
\usepackage{hyperref}
\usepackage{tikz}
\usepackage{amsmath}
\usepackage{amsfonts}
\usepackage{amssymb}
\usepackage{graphicx}%
\usetikzlibrary{arrows,graphs,decorations.pathmorphing,decorations.markings}
\tikzset{degil/.style={
            decoration={markings,
            mark= at position 0.5 with {
                  \node[transform shape,scale=0.5] (tempnode) {$|$};
                  }
              },
              postaction={decorate}
}
}
\theoremstyle{plain}
\newtheorem{thm}{Theorem}[section]
\newtheorem{lem}[thm]{Lemma}
\newtheorem{prop}[thm]{Proposition}
\newtheorem{cor}[thm]{Corollary}
\theoremstyle{definition}
\newtheorem{defn}[thm]{Definition}

\newtheorem{rmk}[thm]{Remark}
\newtheorem{question}[thm]{Question}

\newcommand{\concat}{
  \mathord{
    \mathchoice
    {\raisebox{1ex}{\scalebox{.7}{$\frown$}}}
    {\raisebox{1ex}{\scalebox{.7}{$\frown$}}}
    {\raisebox{.7ex}{\scalebox{.5}{$\frown$}}}
    {\raisebox{.7ex}{\scalebox{.5}{$\frown$}}}
  }
}

\newcommand{\RCA}{\mathsf{RCA}}
\newcommand{\ACA}{\mathsf{ACA}}

\newcommand{\ATR}{\mathsf{ATR}}
\newcommand{\CA}{\mathsf{CA}}
\newcommand{\AC}{\mathsf{AC}}

\newcommand{\HYP}{\mathrm{HYP}}
\newcommand{\IRT}[1]{\mathsf{IRT}_{\mathrm{#1}}}

\newcommand{\WIRT}[1]{\mathsf{WIRT}_{\mathrm{#1}}}

\newcommand{\SEP}{\mathsf{SEP}}
\newcommand{\INDEC}{\mathsf{INDEC}}
\newcommand{\ABW}{\mathsf{ABW}}

\newcommand{\seq}[1]{\langle #1 \rangle}

\numberwithin{equation}{section}
\begin{document}

\title{Halin's Infinite Ray Theorems: Complexity and Reverse Mathematics: Version E\thanks{All
the authors were partially supported by NSF Grant DMS-1161175.}}
\author{James S. Barnes\\Department of Mathematics\\Yale University\\New Haven CT 06520
\and Jun Le Goh\\Department of Mathematics\\National University of Singapore\\Singapore 119076
\and Richard A. Shore\\Department of Mathematics\\Cornell University\\Ithaca NY 14853}
\date{\today}
\maketitle

\begin{abstract}
Halin [1965] proved that if a graph has $n$ many pairwise disjoint rays for
each $n$ then it has infinitely many pairwise disjoint rays. We analyze the
complexity of this and other similar results in terms of computable and proof
theoretic complexity. The statement of Halin's theorem and the construction
proving it seem very much like standard versions of compactness arguments such
as K\"{o}nig's Lemma. Those results, while not computable, are relatively
simple. They only use arithmetic procedures or, equivalently, finitely many
iterations of the Turing jump. We show that several Halin type theorems are
much more complicated. They are among the theorems of hyperarithmetic
analysis. Such theorems imply the ability to iterate the Turing jump along any
computable well ordering. Several important logical principles in this class
have been extensively studied beginning with work of Kreisel, H. Friedman,
Steel and others in the 1960s and 1970s. Until now, only one purely
mathematical example was known. Our work provides many more and so answers
Question 30 of Montalb\'{a}n's Open Questions in Reverse Mathematics [2011].
Some of these theorems including ones in Halin [1965] are also shown to have
unusual proof theoretic strength as well.

\end{abstract}



\section{Introduction}

\label{section:intro}

In this paper we analyze the complexity of several results in infinite graph
theory. These theorems are said to be ones of Halin type or, more generally,
of ubiquity theory. The classical example is a theorem of Halin \cite{halin65}%
: If a countable graph $G$ contains, for each $n$, a sequence $\left\langle
R_{0},\ldots, R_{n-1}\right\rangle $ of disjoint rays (a ray is a sequence
$\left\langle x_{i}\mid i\in N\right\rangle $ of distinct vertices such that
there is an edge between each $x_{i}$ and $x_{i+1}$) then it contains an
infinite such sequence of rays. (Note: As will be described in Definition
\ref{rays}, when we talk about disjoint rays we always mean pairwise
disjoint.) Halin actually deals with arbitrary graphs and formulates the
result differently. The uncountable cases, however, are essentially just
counting arguments. We deal only with countable structures but discuss his
formulation in \S \ref{section:maximality}. This standard formulation of his
theorem seems like a typical compactness theorem going from arbitrarily large
finite collections of objects to an infinite collection. The archetypical
example here is K\"{o}nig's Lemma: If a finitely branching tree has paths of
length $n$ for every $n$ then it has a branch, i.e.\ an infinite path. In
outline, a modern proof of Halin's theorem for countable graphs (due to
Andreae, see \cite[Theorem 8.2.5(i)]{diestel_book}) seems much like
that of K\"{o}nig's Lemma (and many others in infinite graph theory). The
construction of the desired sequence of rays proceeds by a recursion through
the natural numbers in which each step is a simple procedure. While the
procedure is much more delicate than for K\"{o}nig's Lemma, it is basically of
the same complexity. It uses Menger's theorem for finite graphs at each step
but this represents a computable procedure (for finite graphs). The other
parts of the step depend on the same type of information as in K\"{o}nig's
Lemma. They ask, for example, if various sets (computable in the given graph)
are nonempty or infinite. Nonetheless, we prove that the complexity of this
construction and theorem are much higher than that for K\"{o}nig's Lemma or
other applications of compactness. The concepts from graph theory,
computability theory and proof theory/reverse mathematics that we need for our
analysis are discussed in \S \ref{basic}. Basic references for terminology,
background and standard results not explicitly stated or otherwise attributed
are Diestel \cite{diestel_book} for graph theory; Rogers \cite{rogers} and
Sacks \cite{sacks} for computability theory; and for reverse mathematics
Simpson \cite{sim_book} with an approach which is primarily in terms of formal
systems and Hirschfeldt \cite{hirschfeldt} with one primarily emphasizing computability.

We follow two well established procedures for measuring the complexity of
constructions and theorems. The first is basically computability theoretic. It
has its formal beginnings in the 1950s but has much earlier roots in
constructive or computable mathematics reaching back to antiquity. (See Ershov
et al.\ \cite{handbook} for history and surveys of the approach in several
areas of combinatorics, algebra and analysis.) The measuring rod here is
relative computability. We say a set $A$ of natural numbers is (Turing)
computable from a set $B$, $A\leq_{T}B$, if there is an algorithm (say on a
Turing machine or any other reasonable model of general computation) that,
when given access to all membership facts about $B$ (an oracle for $B$)
computes membership in $A$. The standard hierarchies of complexity here are
based on iterations of the Turing jump. This operator takes $B$ to $B^{\prime
}$, the halting problem relativized to $B$, i.e.\ the set of programs with
oracle for $B$, $\Phi_{e}^{B}$, such that $\Phi_{e}^{B}$ halts on input $e$.
For example, if the tree of K\"{o}nig's Lemma is computable in $B$ then there
is a branch computable in the double jump $B^{\prime\prime}$ of $B$.

The second approach is proof theoretic. It measures the complexity of a
theorem by the logical strength of the axioms needed to prove it. This
approach also has a long history but the formal subject, now called reverse
mathematics, starts with H. Friedman's work in the 1970s (e.g.\ \cite{friedman_LC,friedman_icm}). One compares axiomatic systems $S$ and $T$ by saying that $T$ is
stronger than $S$, $T\vdash S$ ($T$ proves $S$) if one can prove every
sentence $\Theta\in S$ from the axioms of $T$. Of course, we know what it
means for $\Theta$ to be provable in $S$. The goal here is to characterize to
the extent possible the axioms needed to prove a given mathematical theorem
$\Theta$. To this end, one begins with a weak base theory. Then one wants to
find a system $S$ such that not only does $S\vdash\Theta$ but also $\Theta$
(with the weak base theory) proves all the axioms of $S$. Hence the name
reverse mathematics as we seek to prove the \textquotedblleft
axioms\textquotedblright\ of $S$ from the theorem $\Theta$. Typically, the
systems here are formalized in arithmetic with quantification over sets as
well as numbers. The standard base theory ($\mathsf{RCA}_{0}$) corresponds to
the axioms needed to do computable constructions. Stronger systems are then
usually generated by adding comprehension axioms which assert the existence of
specific families of sets. For example, a very important system is
$\mathsf{ACA}_{0}$. It is equivalent in the sense of reverse mathematics just
described to K\"{o}nig's Lemma. Formally, it asserts that every subset of the
natural numbers defined by a formula that quantifies only over numbers (and
not sets) exists. This is also equivalent to asserting that for every set $B$,
the set $B^{\prime}$ exists.

The early decades of reverse mathematics were marked by a large variety of
results characterizing a wide array of theorems and constructions as being one
of five specific levels of complexity including $\mathsf{RCA}_{0}$ and
$\mathsf{ACA}_{0}$. Each of these systems (Simpson's \textquotedblleft big
five\textquotedblright) have corresponding specific recursion theoretic
construction principles. In more recent decades, there has been a
proliferation of results placing theorems and constructions outside the big
five. Sometimes these are inserted linearly and sometimes with
incomparabilities. They are now collectively often called the
\textquotedblleft zoo\textquotedblright\ of reverse mathematics. (See
\href{https://rmzoo.math.uconn.edu/diagrams/}{https://rmzoo.math.uconn.edu/diagrams/}
for pictures.)

Theorems and constructions in combinatorics in general, and graph theory in
particular, have been a rich source of such denizens of this zoo. Almost all
of them have fallen below $\mathsf{ACA}_{0}$ (K\"{o}nig's Lemma) and so have
the objects they seek constructible computably in finitely many iterations of
the Turing jump. Ramsey theory, in particular, has provided a very large class
of constructions and theorems of distinct complexity. One example of the
infinite version of a classical theorem of finite graph theory that is
computationally and reverse mathematically strictly stronger than
$\mathsf{ACA}_{0}$ is K\"{o}nig's Duality Theorem ($\mathsf{KDT}$) for
countable graphs. (Every bipartite graph has a matching and a cover consisting
of one vertex from each edge of the matching.) The proofs of this theorem for
infinite graphs (Podewski and Steffens \cite{ps76} for countable and Aharoni
\cite{a84} for arbitrary ones) are not just technically difficult but
explicitly used both transfinite recursions and well orderings of all subsets
of the given graph. These techniques lie far beyond $\mathsf{ACA}_{0}$.
Aharoni, Magidor and Shore \cite{ams92} proved that this theorem is of great
computational strength in that there are computable graphs for which the
required matching and cover compute all the iterations of the Turing jump
through all computable well-orderings. They also showed that it was strong
reverse mathematically as it implied $\mathsf{ATR}_{0}$, the standard system
above $\mathsf{ACA}_{0}$ used to deal with such transfinite recursions. Some
of the lemmas used in each of the then known proofs were shown to be
equivalent to the next and final of the big five systems, $\Pi_{1}^{1}%
$-$\mathsf{CA}_{0}$ and of corresponding computational strength. Simpson
\cite{sim94} later provided a new proof of the theorem using logical methods
that avoided these lemmas and showed that the theorem itself is equivalent to
$\mathsf{ATR}_{0}$ and so strictly weaker than the lemmas both computationally
and in terms of reverse mathematics.

The situation for the theorems of Halin type that we study here is quite
different. The standard proofs do not seem to use such strong methods.
Nonetheless, as we mentioned above, the theorems are much stronger than
$\mathsf{ACA}_{0}$ with some versions not even provable in $\mathsf{ATR}_{0}$.
We prove that these theorems occupy a few houses in the area of the reverse
mathematics zoo devoted to what are called theorems (or theories) of
hyperarithmetic analysis, THAs (Definition \ref{thadef}). Computationally, for
each computable well ordering $\alpha$, there is a computable instance of any
THA which has all of its required objects Turing above $0^{(\alpha)}$, the
$\alpha$th iteration of the Turing jump. On the other hand, they are
computationally and proof theoretically much weaker than $\mathsf{ATR}_{0}$
and so $\mathsf{KDT}$. The point here is that there is a single computable
graph such that the matching and cover required by $\mathsf{KDT}$ lies above
$0^{(\alpha)}$ for all the computable well-orderings $\alpha$, while for each
computable instance of a THA there is a computable well-ordering $\alpha$ such
that $0^{(\alpha)}$ computes the desired object. In our cases, the instances
are graphs with arbitrarily many disjoint rays and the desired object is an
infinite sequence of disjoint rays. (The general usage of terms
like instances and solutions of a theorem or principle is described at the end of \S 3.)	

Beginning with work of Kreisel \cite{kreisel}, H. Friedman
\cite{friedman_thesis}, Steel \cite{steel78} and others in the 1960s and 1970s
and continuing into the last decade (by Montalb\'{a}n \cite{montalban_jullien,
montalban_pi11sep}, Neeman \cite{neeman08,neeman11} and
others), several axiomatic systems and logical theorems were found to be THAs
and proven to lie in a number of distinct classes in terms of proof theoretic
complexity. Until now, however, there has been only one mathematical but not
logical example, i.e.\ one not mentioning classes of first order formulas or
their syntactic complexity. This was a result (INDEC) about indecomposability
of linear orderings in Jullien's thesis \cite{jullien_thesis} (see Rosenstein
\cite[Lemma 10.3]{rosenstein}). It was shown to be a THA by Montalb\'{a}n
\cite{montalban_jullien}.

The natural quest then became to find out if there are any other THAs in the
standard mathematical literature. The issue was raised explicitly in
Montalb\'{a}n's ``Open Questions in Reverse Mathematics'' \cite[Q30]%
{montalban_open}. As our answer, we provide many examples. Most of them are
provable in a well known system above $\mathsf{ACA}_{0}$ gotten by adding on a
weak form of the axiom of choice ($\Sigma^{1}_{1}\text{-}\mathsf{AC}_{0}$).

Several of the basic Halin type theorems (the $\mathsf{IRT}_{\mathrm{XYZ}}$
defined after Definition \ref{irt}) have versions (the $\mathsf{IRT}%
_{\mathrm{XYZ}}^{\ast}$ of Definition \ref{defn:IRTmc}) like those appearing
in the original papers that show that there are always families of disjoint
rays of maximal cardinality which are of the same computational strength as
the basic versions (Proposition \ref{prop:IRT_IRTmc_rshp} and Corollary
\ref{cor:variants_IRTmc_hyp_analysis}). On the other hand, the $\mathsf{IRT}_{\mathrm{XYZ}}^{\ast}$ are strictly stronger proof theoretically
than the $\mathsf{IRT}_{\mathrm{XYZ}}$ because they imply more induction than
is available in $\Sigma_{1}^{1}\text{-}\mathsf{AC}_{0}$ (Theorem
\ref{thm:IRTmc_ACA0_ast} and Corollary
\ref{cor:IRTmc_not_provable_in_Sigma11-AC}). Two of the variations we consider
are as yet open problems of graph theory (\cite{bcp15} and Bowler, personal
communication). We show that if we restrict the class of graphs to directed
forests the principles are not only provable but reverse mathematically
equivalent to $\Sigma_{1}^{1}\text{-}\mathsf{AC}_{0}+\mathsf{I}\Sigma_{1}^{1}$. Note
that as $\mathsf{ATR}_{0}\nvdash \mathsf{I}\Sigma_{1}^{1}$ \cite[IX.4.7]%
{sim_book}, these theorems are not provable even in $\mathsf{ATR}_{0}$ or
from $\mathsf{KDT}$ (Corollary \ref{cor:IRTmc_DVD_strictly_implies_SigmaAC}).
We do not know of other mathematical but nonlogical theorems of this strength.
Other versions that require maximal sets of rays (Definition \ref{defn:MIRT})
are much stronger and, in fact, equivalent to $\Pi_{1}^{1}$-$\mathsf{CA}_{0}$
(Theorem \ref{thm:MIRT_equiv_Pi11-CA}).

\section{Outline of Paper}

\label{section:outline}

Section \ref{basic} discusses basic concepts and background information. The
first subsection (\ref{graphth}) provides what we need from graph theory.
Almost all the definitions are standard. At times we give slight variations
that are equivalent to the standard ones but make dealing with the
computability and proof theoretic analysis easier. We also state the theorems
of Halin and some variants that are the main targets of our analysis.

The second subsection (\ref{comph}), assumes an intuitive view of
computability of functions $f: N \to N$ such as having an
algorithm given by a program in any standard computer language (possibly with
access to an \textquotedblleft oracle\textquotedblright\ providing information
about a given set or function). It then gives the standard notions and theorems
that can be found in basic texts on computability theory needed to follow our
analysis of the computational complexity of the graph theoretic theorems we
study. In particular, it notes the Turing jump operator and its iterations
along countable well orderings. These are our primary computational measuring
rods. The final subsection (\ref{axiomh}) provides the syntax and semantics
for the formal systems of arithmetic that are used to measure proof theoretic
complexity. It also describes the standard basic axiomatic systems and their
connections to the computational measures of the previous subsection. It
includes the formal definition of the class of theorems which includes most of
our graph theoretic examples, the THAs, Theorems (or Theories) of
Hyperarithmetic Analysis. These are defined in terms of the transfinite
iterations of the Turing jump and the hyperarithmetic sets of the previous
subsection. In addition it defines $\Sigma_{1}^{1}\text{-}\mathsf{AC}_{0}$ a
weak version of the axiom of choice that is an early well known example of
such theories and plays a crucial role in our analysis.

Section \ref{section:IRT_hyp_analysis} provides the proof that Halin's
original theorem $\mathsf{IRT}_{}$ (Definition \ref{irt}) is computationally
very complicated. For example, given any iteration $0^{(\alpha)}$ of the
Turing jump, there is a computable graph satisfying the hypotheses of
$\mathsf{IRT}_{}$ such that any instance of its conclusion computes
$0^{(\alpha)}$. Indeed, $\mathsf{IRT}_{}$ is a THA. At times, theorems or
lemmas are stated in terms of the formal systems of \S \ref{axiomh}, but the
proofs rely only on the computational notions of \S \ref{comph}.

Section \ref{section:variants_IRT} studies several variations $\mathsf{IRT}%
_{\mathrm{XYZ}}$ of Halin's $\mathsf{IRT}_{}$ where we consider directed as
well as undirected graphs, edge rather than vertex disjointness for the rays
and double as well as single rays. (See Definitions \ref{graphs} and
\ref{rays} and the discussion after Definition \ref{irt}.) We provide
reductions over $\mathsf{RCA}_{0}$ between many of the pairs of the eight
possible variants. The proofs of these reductions proceed purely
combinatorially by providing one computational process that takes an instance
of some $\mathsf{IRT}_{\mathrm{XYZ}}$, i.e.\ a graph satisfying its hypotheses,
and produces a graph satisfying the hypotheses of another $\mathsf{IRT}%
_{\mathrm{X^{\prime}Y^{\prime}Z^{\prime}}}$ and another computable process
that takes any solution to the $\mathsf{IRT}_{\mathrm{X^{\prime}Y^{\prime
}Z^{\prime}}}$ instance, i.e. any sequence of rays satisfying the conclusion
of $\mathsf{IRT}_{\mathrm{X^{\prime}Y^{\prime}Z^{\prime}}}$, and produces a
solution to the original instance of $\mathsf{IRT}_{\mathrm{XYZ}}$. (See
Propositions \ref{prop:IRT_DYZ_implies_IRT_UYZ},
\ref{prop:IRT_DEZ_implies_IRT_DVZ} and \ref{prop:IRT_DYD_implies_IRT_DYS} and
the associated Lemmas. {An additional reduction using a stronger base theory
is given in the next section (Theorem \ref{thm:UVDmc_implies_UVS}).) }

We then show that five of the eight possible variants of $\IRT{}$ are THAs (Theorem
\ref{thm:variants_IRT_hyp_analysis}). As mentioned in \S \ref{section:intro}, of the remaining
three, two are still open problems in graph theory. We do, however, have an
analysis of their restrictions to special classes of graphs in Theorem
\ref{thm:SigmaAC_implies_DED_directed_forests} and
\S \ref{section:max_cardinality}. The last of the variations, $\mathsf{IRT}%
_{\mathrm{UED}}$, has been proven more recently by Bowler, Carmesin, Pott
\cite{bcp15} using more sophisticated methods than the other results. We have
some lower bounds (Theorem
\ref{thm:IRT_variants_ACA_0_and_omega_model_IRT_variants_hyp_closed}) but we
have yet to fully analyze the complexity of their construction.

In the next section (\S \ref{section:maximality}) we study some variations of
$\mathsf{IRT}_{{}}$ that ask for different types of maximality for the
solutions. The first sort actually follow the original formulation of
$\mathsf{IRT}_{{}}$ in Halin \cite{halin65}: In any graph there is a set of
disjoint rays of maximum cardinality. For uncountable graphs this amounts to a
basic counting argument on uncountable cardinals as all rays are countable.
When restricted to countable graphs these variations, $\mathsf{IRT}%
_{\mathrm{XYZ}}^{\ast}$, are easily seen to be equivalent to our more modern
formulation by induction. Technically, the induction used is for $\Sigma
_{1}^{1}$ formulas ($\mathsf{I}\Sigma_{1}^{1}$) which is not available in
$\mathsf{RCA}_{0}$. More specifically we show (Proposition
\ref{prop:IRT_IRTmc_rshp}) that $\mathsf{IRT}_{\mathrm{XYZ}}+\mathsf{I}\Sigma_{1}^{1}$
and $\mathsf{IRT}_{\mathrm{XYZ}}^{\ast}+\mathsf{I}\Sigma_{1}^{1}$ are equivalent (over
$\mathsf{RCA}_{0}$). As the definition of THAs only depends on standard models
where full induction holds, if $\mathsf{IRT}_{\mathrm{XYZ}}$ is a THA then so
is $\mathsf{IRT}_{\mathrm{XYZ}}^{\ast}$.

We then prove that these maximal cardinality variants $\mathsf{IRT}^{\ast
}_{\mathrm{XYZ}}$ are strictly stronger proof theoretically than the basic
$\mathsf{IRT}_{\mathrm{XYZ}}$ (when they are known to be provable in
$\Sigma^{1}_{1}\text{-}\mathsf{AC}_{0}$). This is done by showing (Theorem
\ref{thm:IRTmc_ACA0_ast} and the Remark that follows it) that the relevant
$\mathsf{IRT}^{\ast}_{\mathrm{XYZ}}$ all imply weaker versions of $\mathsf{I}\Sigma_{1}^{1}$ that are analogous to the restrictions of $\Sigma^{1}_{1}%
\text{-}\mathsf{AC}_{0}$ embodied in weak (or unique)-$\Sigma^{1}_{1}%
\text{-}\mathsf{AC}_{0}$ and finite-$\Sigma^{1}_{1}\text{-}\mathsf{AC}_{0}$
(Definitions \ref{defn:unique_choice} and \ref{defn:finite_choice}). In all
the cases, it is enough induction to prove (with the apparatus of the basic
$\mathsf{IRT}_{\mathrm{XYZ}}$) the consistency of $\Sigma^{1}_{1}%
\text{-}\mathsf{AC}_{0}$ and so by G\"{o}del's second incompleteness theorem
they cannot be proved in $\Sigma^{1}_{1}\text{-}\mathsf{AC}_{0}$ (Corollary
\ref{cor:IRTmc_not_provable_in_Sigma11-AC}).

As for proving full $\Sigma_{1}^{1}$ induction from an $\mathsf{IRT}%
_{\mathrm{XYZ}}^{\ast}$ we are in much the same situation mentioned above for
$\Sigma_{1}^{1}\text{-}\mathsf{AC}_{0}$ and $\mathsf{IRT}_{\mathrm{XYZ}}$. In
particular, $\mathsf{IRT}_{\mathrm{DVD}}^{\ast}$ and $\mathsf{IRT}%
_{\mathrm{DED}}^{\ast}$ for directed forests each proves $\mathsf{I}\Sigma_{1}^{1}$ as
well as $\Sigma_{1}^{1}\text{-}\mathsf{AC}_{0}$ (Theorems
\ref{thm:DVD-IRTmc_implies_ISigma11} and
\ref{thm:DED_DVD_directed_forests_equivalence}) and so are equivalent to
$\mathsf{I}\Sigma_{1}^{1}+\Sigma_{1}^{1}\text{-}\mathsf{AC}_{0}$. As before, this shows
that they are strictly stronger than $\Sigma_{1}^{1}\text{-}\mathsf{AC}_{0}$
(Corollary \ref{cor:IRTmc_DVD_strictly_implies_SigmaAC}). Indeed, as mentioned
at the end of \S \ref{section:intro}, they are not even provable in
$\mathsf{ATR}_{0}$. We do not know of any other mathematical theorem with this
level of reverse mathematical strength.

The second variation of maximality, $\mathsf{MIRT}_{\mathrm{XYZ}}$, studied in
\S \ref{section:maximal} is also mentioned in the original Halin paper
\cite{halin65}. It asks for a set of disjoint rays which is maximal in the
sense of set containment. Of course, this follows immediately from Zorn's
Lemma for all graphs. For countable graphs we provide a reverse mathematical
analysis, showing that each of the $\mathsf{MIRT}_{\mathrm{XYZ}}$ is
equivalent to $\Pi^{1}_{1}$-$\mathsf{CA}_{0}$ (Theorem
\ref{thm:MIRT_equiv_Pi11-CA}).

In \S \ref{section:IRT_vs_other_theories}, we discuss the reverse mathematical
relationships between the THAs associated with variations of Halin's theorem
and previously studied THAs as well as one new logical THA (finite-$\Sigma
_{1}^{1}\text{-}\mathsf{AC}_{0}$ of Definition \ref{defn:unique_choice}).
Basically, all the $\mathsf{IRT}_{\mathrm{XYZ}}^{\ast}$ (and so $\mathsf{IRT}%
_{\mathrm{XYZ}}+\mathsf{I}\Sigma_{1}^{1}$) imply H.\ Friedman's $\mathsf{ABW}_{0}$
(Definition \ref{defn:ABW}) by Theorem \ref{thm:IRTmc_implies_ABW} and
finite-$\Sigma_{1}^{1}\text{-}\mathsf{AC}_{0}$ (Theorem
\ref{thm:XYZmc_implies_finite_choice}). On the other hand, none of them are
implied by it (Theorem \ref{thm:ABW_not_imply_IRT}) or by $\Delta_{1}^{1}%
$-$\mathsf{CA}_{0}$ (Definition \ref{defn:Delta11-CA} and Theorem
\ref{thm:Delta11-CA_not_imply_IRT}). $\mathsf{ABW}_{0}+\mathsf{I}\Sigma_{1}^{1}$ does,
however, imply finite-$\Sigma_{1}^{1}\text{-}\mathsf{AC}_{0}$ which is not
implied by weak (unique)-$\Sigma_{1}^{1}\text{-}\mathsf{AC}_{0}$ (Goh
\cite{goh_finite_choice}). Figure \ref{fig:hyp_analysis_zoo_IRT} summarizes
many of the known relations with references.

In the penultimate section (\S \ref{section:WIRT}) we study the only use of
$\Sigma^{1}_{1}\text{-}\mathsf{AC}_{0}$ in each of our proofs of
$\mathsf{IRT}_{\mathrm{XYZ}}$. It consists of $\mathsf{SCR}_{\mathrm{XYZ}}$
which says we can go from the hypothesis that there are arbitrarily many
disjoint rays to a sequence $\seq{X_{k}}_{k}$ in which each $X_k$ is a sequence of $k$
many disjoint rays. We analyze the strength of the $\mathsf{SCR}%
_{\mathrm{XYZ}}$ and the weakenings $\mathsf{WIRT}_{\mathrm{XYZ}}$ of
$\mathsf{IRT}_{\mathrm{XYZ}}$ which each take the existence of such a sequence
$\seq{X_{k}}_{k}$ as its hypothesis in place of there being arbitrarily many
disjoint rays. For example, for all the $\mathsf{IRT}_{\mathrm{XYZ}}$ which
are consequences of $\Sigma^{1}_{1}\text{-}\mathsf{AC}_{0}$ and so are THAs,
$\mathsf{IRT}_{\mathrm{XYZ}}$ is equivalent to $\mathsf{SCR}_{\mathrm{XYZ}}$
over $\mathsf{RCA}_{0}$ (Corollary \ref{cor:SCR_IRT_equiv}) and so all of them
are also THAs. For the same choices of $XYZ$, $\mathsf{ACA}_{0}$ proves
$\mathsf{WIRT}_{\mathrm{XYZ}}$ over $\mathsf{RCA}_{0}$. While a natural
strengthening of $\mathsf{WIRT}_{\mathrm{XYZ}}$ does imply $\mathsf{ACA}_{0}$
and indeed is equivalent to it (Theorem \ref{thm:nonuniform_WIRT}), we do not
know if $\mathsf{WIRT}_{\mathrm{XYZ}}$ itself implies $\mathsf{ACA}_{0}$. All
we can prove is that it is not a consequence of $\mathsf{RCA}_{0}$ (Theorem
\ref{thm:WIRT_RCA0}).

In the last section (\S \ref{section:open}), we mention some open problems.

\section{Basic Notions and Background\label{basic}}

We begin with basic notions and terminology from graph theory. At times we use
formalizations that are clearly equivalent to more standard ones but are
easier to work with computationally or proof-theoretically. The following two
subsections supply background and basic information about the standard
computational and logical/proof theoretic notions that we use here to measure
the complexity of the graph theorems and constructions that we analyze in the
rest of this paper. Note that we denote the set of natural numbers by $N$ when we may be thinking of them in a model of arithmetic as in \S \ref{axiomh} and by $\mathbb{N}$ when we emphasize that we specifically want the standard natural numbers.

\subsection{Graph Theoretic Notions\label{graphth}}

\begin{defn}
\label{graphs}A \emph{graph} $H$ is a pair $\left\langle V,E\right\rangle $
consisting of a set $V$ (of \emph{vertices}) and a set $E$ of unordered pairs
$\{u,v\}$ with $u\neq v$ from $V$ (called \emph{edges}). These structures are
also called \emph{undirected graphs (or here U-graphs)}. A structure $H$ of
the form $\left\langle V,E\right\rangle $ as above is a \emph{directed graph}
\emph{(or here D-graph) }if $E$ consists of ordered pairs $\left\langle
u,v\right\rangle $ of vertices with $u\neq v$. To handle both cases
simultaneously, we often use $X$ to stand for undirected (U) or directed (D).
We then use $(u, v)$ to stand for the appropriate kind of edge, i.e.\ $\{u,
v\}$ or $\langle u, v \rangle$.

An \emph{\ }$X$-\emph{subgraph }of the $X$-graph $H$ is an $X$-graph
$H^{\prime}=\left\langle V^{\prime},E^{\prime}\right\rangle $ such that
$V^{\prime}\subseteq V$ and $E^{\prime}\subseteq E$. It is an \emph{induced
}$X$\emph{-subgraph }if $E^{\prime}=\{(u,v)\mid u,v\in V^{\prime} \land (u,v)\in
E\}$.
\end{defn}

\begin{defn}
\label{rays}An\emph{\ }$X$-\emph{ray in H} is a pair consisting of an
$X$-subgraph $H^{\prime}=\left\langle V^{\prime},E^{\prime}\right\rangle $ of
$H$ and an isomorphism $f$ from $N$ with edges $(n,n+1)$ for $n\in N$ to
$H^{\prime}$. (Note that this implies that the range of $f$ is the set
$V^{\prime}$.) We say that the \emph{ray begins at} $f(0)$. We also describe
this situation by saying that $H$ contains the $X$-ray $\left\langle
H^{\prime},f\right\rangle $. We sometimes abuse notation by saying that the
sequence $\left\langle f(n)\right\rangle $ of vertices is an $X$-ray in $H$. A \emph{tail} of an $X$-ray is a final segment of said $X$-ray. Similarly we consider \emph{double X-rays }where the isomorphism $f$ is from
the integers $\mathcal{Z}=\{-n,n\mid n\in N\}$ with edges $(z,z+1)$ for
$z\in\mathcal{Z}$. A \emph{tail} of a double $X$-ray is a final segment of said $X$-ray, or an initial segment of said $X$-ray considered in reverse order.

We use $Z$-ray to stand for either a (single) ray ($Z={}$S)
or double ray ($Z={}$D) and so we have, in general, $Z$-$X$-rays or just
$Z$-rays if the type of graph (U or D) is already established. For
brevity, when we describe rays we will often only list their vertices in order
instead of defining $H^{\prime}$ and $f$ explicitly. However the reader should
be aware that we always have $H^{\prime}$ and $f$ in the background.

$H$ \emph{contains }$k$\emph{\ many }$Z$-$X$-\emph{rays} for $k\in N$ if there
is a sequence $\left\langle H_{i},f_{i}\right\rangle _{i<k}$ such that each
$\left\langle H_{i},f_{i}\right\rangle $ is a $Z$-$X$-ray in $H$ (with
$H_{i}=\left\langle V_{i},E_{i}\right\rangle $). $H$ \emph{contains }$k$\emph{\ many disjoint (or vertex-disjoint) }$Z$%
\emph{-}$X$\emph{-rays }if the $V_{i}$ are pairwise disjoint. $H$
\emph{contains }$k$\emph{\ many edge-disjoint }$Z$\emph{-}$X$\emph{-rays }if
the $E_{i}$ are pairwise disjoint. We often use $Y$ to stand for either vertex
(V) or edge (E) as in the following definitions.

An $X$-graph $H$ \emph{contains arbitrarily many Y-disjoint Z-X-rays }if it
contains $k$ many such rays for every $k\in N$.

An $X$-graph $H$ \emph{contains infinitely many Y-disjoint Z-X-rays }if there
is an $X$-subgraph $H^{\prime}=\left\langle V^{\prime},E^{\prime}\right\rangle
$ of $H$ and a sequence $\left\langle H_{i},f_{i}\right\rangle _{i\in N}$ such
that each $\left\langle H_{i},f_{i}\right\rangle $ is a $Z$-$X$-ray in $H$
(with $H_{i}=\left\langle V_{i},E_{i}\right\rangle $) such that the $V_{i}$ or
$E_{i}$, respectively for $Y=V,E$, are pairwise disjoint and $V^{\prime}= \bigcup V_{i}$ and $E^{\prime}= \bigcup E_{i}$.
\end{defn}

\begin{defn}
\label{paths}An $X$\emph{-path} $P$ in an $X$-graph $H$ is defined similarly
to single rays except that the domain of $f$ is a proper initial segment of
$N$ instead of $N$ itself. Thus they are finite sequences of distinct vertices
with edges between successive vertices in the sequence. If $P=\left\langle
x_{0},\ldots,x_{n}\right\rangle $ is a path, we say it is a \emph{path of
length }$n$ between $x_{0}$ and $x_{n}$. Our \emph{notation for truncating and
combining paths} $P=\left\langle x_{0},\ldots,x_{n}\right\rangle $,
$Q=\left\langle y_{0},\ldots,y_{m}\right\rangle $ and $R=\left\langle
z_{0},\ldots,z_{l}\right\rangle $ is as follows: $x_{i}P=\left\langle
x_{i},\ldots,x_{n}\right\rangle $, $Px_{i}=\left\langle x_{0},\ldots
,x_{i}\right\rangle $, and we use concatenation in the natural way, e.g., if
the union of $Px$, $xQy$ and $yR$ is a path, we denote it by $PxQyR$. We treat
rays as we do paths in this notation, as long as it makes sense, writing, for
example, $x_{i}R$ for the ray which is gotten by starting $R$ at an element
$x_{i}$ of $R$; $Rx_{i}$ is the path which is the initial segment of $R$
ending in $x_{i}$ and we use concatenation as for paths as well.
\end{defn}

The starting point of the work in this paper is a theorem of Halin
\cite{halin65} that we call the infinite ray theorem as expressed in \cite[Theorem 8.2.5(i)]{diestel_book}.

\begin{defn}
[Halin's Theorem]\label{irt}\emph{$\mathsf{IRT}_{{}}$, the infinite ray
theorem, }is the principle that every graph $H$ which contains arbitrarily
many disjoint rays contains infinitely many disjoint rays.
\end{defn}

The versions of Halin's theorem which we consider in this paper allow for $H$
to be an undirected or a directed graph and for the disjointness requirement
to be vertex or edge. We also allow the rays to be single or double. The
corresponding versions of Halin's Theorem are labeled as $\mathsf{IRT}%
_{\mathrm{XYZ}}$ for appropriate values of $X,$ $Y$ and $Z$ to indicate
whether the graphs are undirected or directed ($X={}$U or D); whether the
disjointness refers to the vertices or edges ($Y={}$V or E) and whether the
rays are single or double ($Z={}$S or D), respectively, in the obvious way. We
often state a theorem for several or all $XYZ$ and then in the proof use
\textquotedblleft graph\textquotedblright, \textquotedblleft
edge\textquotedblright\ and \textquotedblleft disjoint\textquotedblright%
\ unmodified with the intention that the proof can be read for any of the
cases. This is convenient for minimizing repetition in some of our arguments.

We will also consider restrictions of these theorems to specific families of
graphs. We need a few more notions to define them.

\begin{defn}
\label{treedef}A \emph{tree }is a graph $T$ with a designated element $r$
called its \emph{root }such that for each vertex $v\neq r$ there is a unique
path from $r$ to $v$. A \emph{branch} on (or in) $T$ is a ray that begins at its root.
We denote the set of its branches by $[T]$ and say that $T$ is
\emph{well-founded} if $[T]=\emptyset$ and otherwise it is \emph{ill-founded}.
A \emph{forest} is an \emph{effective disjoint union }of trees, or more
formally, a graph with a designated set $R$ (of vertices called roots) such
that for each vertex $v$ there is a unique $r\in R$ such that there is a path
from $r$ to $v$ and, moreover, there is only one such path. In general, the
\emph{effectiveness} we assume when we take \emph{effective disjoint unions of
graphs} means that we can effectively (i.e.\ computably) uniquely identify each
vertex in the union with the original vertex (and the graph to which it
belongs) which it represents in the disjoint union.

A \emph{directed tree }is a directed graph $T=\left\langle V,E\right\rangle $
such that its \emph{underlying graph} $\hat{T}=\langle V,\hat{E}\rangle$ where
$\hat{E}=\{\{u,v\}\mid\langle u,v\rangle\in E\,\vee\,\langle v,u\rangle\in
E\}$ is a tree. A \emph{directed forest} is a directed graph whose underlying
graph is a forest.
\end{defn}

\begin{defn}
\label{locfin}An $X$-graph $H$ is \emph{locally finite }if, for each $u\in V$,
the set $\{v\in E\mid(u,v) \in E \lor (v,u)\in E\}$ of \emph{neighbors of }$u$ is finite. A locally finite $X$-tree is also called \emph{finitely
branching}. (Note this does not mean there are finitely many branches in the tree.)
\end{defn}

Of course, there are many well known equivalent definitions of trees and
associated notions. We have given one possible set of graph-theoretic
definitions. In the case of undirected graphs our definition is equivalent to
all the standard ones. Readers are welcome to think in terms of their favorite
definition. Note, however, we are restricting ourselves to what would (in set
theory) be called countable trees with all nodes of finite rank. Thus, we
typically think of trees as \emph{subtrees of} $N^{<N}$, i.e.\ the sets of
finite strings of numbers (as vertices) with an edge between $\sigma$ and
$\tau$ if and only if they differ by one being an extension of the other by
one element, e.g. $\sigma\concat k=\tau$.

It does not seem as if there is a single standard definition for directed
graphs being directed trees. We have picked one that seems to be at least
fairly common and works for the only situations for which we consider them in
Theorems \ref{thm:SigmaAC_implies_DED_directed_forests},
\ref{thm:DVD-IRTmc_implies_ISigma11}, and
\ref{thm:DED_DVD_directed_forests_equivalence} and Corollary
\ref{cor:IRTmc_DVD_strictly_implies_SigmaAC}.

\subsection{Computability Hierarchies\label{comph}}

While we may cite results about uncountable graphs, all sets and structures
actually studied in this paper will be countable. Thus for purposes of
defining their complexity, we can think of all of them as being subsets of, or
relations or functions on, $\mathbb{N}$.

We do not give a formal stand alone definition of computability for sets or
functions but assume an at least intuitive grasp of some model of computation
such as by a Turing or Register machine that has unbounded memory and is
allowed to run for unboundedly many steps. (We do provide in \S \ref{axiomh} a
definition via definability in arithmetic that is equivalent to the formal
versions of machine model definitions.) Thus we say a function $f:\mathbb{N}%
\rightarrow\mathbb{N}$ is computable if there is a program for one of these
machines that computes $f(n)$ as output when given input $n$. A set $X$ is
computable if its characteristic function $\mathbb{N}\rightarrow\{0,1\}$ is
computable. Note that as the alphabets or our languages are finite, there are
only countably many programs and as our formation rules are effective, we have
a computable list of the programs and hence one, $\Phi_{e}$, of the partial
functions they compute. (They are only partial as, of course, some programs
fail to halt on some inputs.)

Fundamental to measuring the relative computational complexity of sets or
functions is the notion of machines with oracles and Turing reduction. Given a
set $X$ or function $f$ we consider machines augmented by the ability to
produce $X(n)$ or $f(n)$ if it has already produced $n$. We say that such a
machine is one with an oracle for $X$ or $f$. We then say that $X$ is
\emph{computable from (or Turing reducible to)} $Y$ if there is a machine with
oracle $Y$ which computes $X$ via some \emph{reduction} $\Phi_{e}^{X}$. We
write this as $X\leq_{T}Y$. We say $X$ is of the same (Turing) degree as $Y$,
$X\equiv_{T}Y$, if $X\leq_{T}Y$ and $Y\leq_{T}X$. We use all the same
terminology and notations for functions.

The first level beyond the computable in our basic hierarchy of computable
complexity is given by the halting problem $H=\{e|\Phi_{e}(e)$
\emph{converges}$\}$ that is $H$ is set of $e$ such that the computation of
the $e$th machine $\Phi_{e}$ on input $e$ eventually halts. We then define an
operator on sets $X\longmapsto X^{\prime}=\{e|\Phi_{e}^{X}(e)$ converges$\}$
that is $X^{\prime}$ is the set of $e$ such that the computation of the $e$th
machine with oracle $X$, $\Phi_{e}^{X}$ on input $e$ eventually halts. (It is
easy to see that $H\equiv_{T}\emptyset^{\prime}$.) The crucial fact here is
the undecidability of the halting problem (for every oracle), i.e. for every
$X$, $X^{\prime}$ is strictly above $X$ in terms of Turing computability. The
other basic fact that we need about $\emptyset^{\prime}$ is that it is
\emph{computably enumerable, }i.e. there is a computable function $f$ whose
range is $\emptyset^{\prime}$. If $f(s)=x$ we say that $x$ is \emph{enumerated
in}, or \emph{enters}, $0^{\prime}$ at (stage) $s$. If we view $H$ as defined
by using the empty oracle $\emptyset$, the procedure that takes us from the
halting problem to the Turing jump by replacing $\emptyset$ as oracle by $X$
is an instance of a general procedure called \emph{relativization}. It takes
any computable function or proof about computable functions or degrees (i.e.
ones with oracle $\emptyset$) to the same function, or proof about functions,
computable in $X$ (or degrees above that of $X$). Almost always this procedure
trivially transforms correct proofs with oracle $\emptyset$ to ones with
arbitrary oracle $X$. Typically, this transformation keeps the same programs
doing the required work with any oracle. For example, $X^{\prime}$ is
\emph{computably enumerable in} $X$ (or relative to $X$), i.e. there is a
function $\Phi_{e}^{X}$ whose range is $X^{\prime}$ and this can be taken to
be the same $e$ such that $\Phi_{e}^{\emptyset}$ enumerates $\emptyset
^{\prime}$. We also use $X_{s}^{\prime}$ to denote \emph{the set of numbers
enumerated in (or that have entered) }$X^{\prime}$\emph{ by stage }$s$. This
phenomena of the procedure or result not depending on the particular oracle or
depending in a fixed computable way on some other parameters is described as
its being \emph{uniform }in the oracle or other parameters. We describe an
important example of uniformity in Remark \ref{jumpstrees}.

We can now generate a hierarchy of computational complexity by iterating the
jump operator beginning with any set $X$: $X^{(0)}=X$; $X^{(n+1)}%
=(X^{n})^{\prime}$. While the finite iterations of the jump capture most
construction techniques and theorems in graph theory (and most other areas of
classical countable/separable mathematics), we will be interested in ones that
go beyond such techniques and proofs. The basic idea is that we continue the
hierarchy by iteration into the transfinite while still tying the iteration to
computable procedures.

\begin{defn}
\label{ord}We represent well-orderings or \emph{ordinals} $\alpha$ as
well-ordered relations on $\mathbb{N}$. Typically such \emph{ordinal
notations} are endowed with various additional structure such as identifying
$0$, successor and limit ordinals and specifying cofinal $\omega$-sequences
for the limit ordinals. If we have a representation of $\alpha$ then
restricting the well-ordering to numbers in its domain provides
representations of each ordinal $\beta<\alpha$. We generally simply work with
ordinals and omit concerns about translating standard relations and procedures
to the representation. An ordinal is recursive (in a set $X$) if it has a
recursive (in $X$) representation. For a set $X$ and ordinal (notation)
$\alpha$ computable from $X$, we define the transfinite iterations $X^{(\beta)}$
of the Turing jump of $X$ by transfinite induction on $\beta\leq\alpha$:
$X^{(0)}=X$; $X^{(\beta+1)}=(X^{\beta})^{\prime}$ and for a limit ordinal
$\lambda$, $X^{(\lambda)}=\oplus\{X^{(\beta)}|\beta<\lambda\}=\cup
\{\beta\times X^{(\beta)}|\beta<\lambda\}$ (or as the effective disjoint sum
over the $X^{(\beta)}$ in the specified cofinal sequence in $\lambda$).
\end{defn}

\begin{defn}
\label{hyp}$\HYP(X)$, the collection of all sets \emph{hyperarithmetic in
}$X$ consists of those sets recursive in some $X^{(\alpha)}$ for $\alpha$ an
ordinal recursive in $X$. We say that $Y$ is\emph{ hyperarithmetic in} $X$ or
\emph{hyperarithmetically reducible to }$X$ , $Y\leq_{h}X$ if $Y\in \HYP(X)$.
\end{defn}

These sets too, will be characterized by a definability class in arithmetic in
\S \ref{axiomh}. For now we just note that they clearly go far beyond the
sets computable from the finite iterations of the jump.

The computational strength of our graph theoretic theorems such as $\IRT{}$ is
measured by this hierarchy as we will show that, for every set $X$ and every
set $Y$ hyperarithmetic in $X$, there is a graph $G$ computable from $X$ which
satisfies the hypotheses of $\IRT{}$ but for which any collection of rays
satisfying its conclusion computes $Y$. On the other hand, placing an upper bound on the strength
of $\IRT{}$ requires analyzing its proof and the principles used in it. The
relevant one is a form of the axiom of choice. We define it in the next
subsection along with a general class of such principles, the
theorems/theories of hyperarithmetic analysis which are, computationally, the
primary objects of our analysis in this paper.

We note one important well known basic fact relating the jumps of $X$ to trees
computable from $X$. We will need it for our proofs that $\IRT{}$ and its variants
are computationally complex enough to compute all the sets hyperarithmetic in
any given set $X$ (as the instances of the graphs range over graphs computable
from $X$).

\begin{rmk}
\label{jumpstrees}For any set $X$ and any ordinal $\alpha$ computable from
$X$, there is a sequence $\left\langle T_{\beta}|\beta<\alpha\right\rangle $
computable from $X$ of trees (necessarily) computable from $X$ such that each
tree has exactly one branch $P_{\beta}$ and $P_{\beta}$
is of the same complexity as $X^{(\beta)}$, i.e. $P_{\beta}\equiv_{T}%
X^{(\beta)}$. The procedure for computing this sequence is uniform in $X$ and
the index for the program computing the well ordering $\alpha$ from $X$, i.e.
there is one computable function that when given an oracle for $X$, an
\emph{index for} $\alpha$ (i.e. the $i$ such that $\Phi_{i}^{X}$ is the well
ordering $\alpha$) and a $\beta$ in the ordering, computes the whole sequence
$\left\langle T_{\beta}|\beta<\alpha\right\rangle $ and the indices for the
reductions between $P_{\beta}$ and $X^{(\beta)}$. (See, e.g. \cite[Theorem
2.3]{shore93}). We may also easily assure that the $T_{\beta}$ are effectively
disjoint so that their union is a forest.
\end{rmk}

Some versions of the variations on $\IRT{}$ (see \S \ref{section:maximal}) that call for types of
maximality for the infinite set of disjoint rays are stronger both
computationally and proof theoretically than the $\IRT{XYZ}$ described above.
Their computational strength is captured by a kind of jump operator that goes
beyond all the hyperarithmetic ones. It captures the ability to tell if a
computable ordering is a well-ordering.

\begin{defn}
\label{hypjump}The \emph{hyperjump of} $X$, $\mathcal{O}^{X}$, is the set
$\{e|\Phi_{e}^{X}$ is (the characteristic function of) a subtree of
$\mathbb{N}^{<\mathbb{N}}$ which is well-founded$\}$.
\end{defn}

This operator also corresponds to a syntactically defined level of
comprehension as we note in \S \ref{axiomh}.

\subsection{Logical and Axiomatic Hierarchies\label{axiomh}}

The basic notions from logic that we need here are those of languages,
structures and axiomatic systems and proofs. As we will deal only with
countable sets and structures, we can assume that we are dealing just with the
natural numbers with a way to define and use sets and functions on them. Thus,
at the beginning, we have in mind the natural numbers $\mathbb{N}$ along with
the usual apparatus of the\emph{ language of (first order) arithmetic}, say
$+,\times,<,0$ and $1$ along with the syntax of standard first order logic
(the Boolean connectives $\vee,\wedge$ and $\lnot$; the variables such as $x$
and $y$ ranging over the numbers with the usual quantifiers $\forall x$ and
$\exists y$ as well as the standard equality relation $=$). A \emph{structure}
for this language is a set $N$ along with elements for $0$ and $1$, binary
functions for $+$ and $\times$ and a binary relation for $<$. We also need a
way of talking about subsets of (or functions on) the numbers. We follow the
standard practice in reverse mathematics of using sets and defining functions
in terms of their graphs. So we expand our language by adding on new classes
of (second order) variables such as $X$ and $Y$ and the associated quantifiers
$\forall X$ and $\exists Y$ along with a new relation symbol $\in$ between
numbers and sets.

A structure for this language is one of the form $\mathcal{N}=\left\langle
N,S,+,\times,<,0,1,\in\right\rangle $ where its restriction $\left\langle
N,+,\times,<,0,1\right\rangle $ is a structure for first order arithmetic and
$S\subseteq2^{N}$ is a specified nonempty collection of subsets of $N$
disjoint from $N$, the set of \textquotedblleft numbers\textquotedblright\ of
$\mathcal{N}$, over which the second order quantifiers and variables of our
language range. It is called the \emph{second order part of }$\mathcal{N}$.
The usual membership symbol $\in$ always denotes the standard membership
relation between elements of $N$ and subsets of $N$ that are in $S$ and the
language only allows atomic formulas using $\in$ which are of the form $t\in
X$ for $t$ a term of the first order language and $X$ is a second order
variable. So a sentence $\Theta$\emph{ is true in }$N$, $\mathcal{N}%
\vDash\Theta$, if first order quantification is interpreted as ranging over
$N$, second order quantification ranges over $S$ and the relations and
functions of the language are as described. This specifies the semantics for
second order arithmetic. Note that, following \cite{sim_book}, we do
not take equality for sets to be a primitive relation on this structure. The
notation for it is viewed as being defined by $A=B\leftrightarrow\forall
n(n\in A\leftrightarrow n\in B)$.

Proof theoretic notions deal with all possible structures for the language and
axiom systems to specify what we need in any particular argument. For most of
our purposes and all of the computational ones, one can restrict attention to
\emph{standard models of arithmetic}, i.e.\ ones $\mathcal{N}$ with
$N=\mathbb{N}$ and some $S\subseteq2^{\mathbb{N}}$ with the usual
interpretations of the functions and relations. We generally abbreviate these
structures as $\left\langle \mathbb{N},S\right\rangle$ with $S\subseteq
2^{\mathbb{N}}$, or simply $S$, as all the functions and relations are then fixed.

We view the syntax as one for a two sorted first order logic. So the (first
order) variables $x,y,\ldots$ range over the first sort ($N$) and the second
order ones $X,Y,\ldots$ over the second sort ($S$). We assume any standard
proof theoretic system with the caveat that $=$ is interpreted as true
equality with the equality axioms included only for $N$. For $S$ it is a
relation defined as above. This generates the provability notion $\vdash$ used
above to define our notion of logical strength and equivalences of theories
(sets of sentences often called axioms) as above. We now define the standard
weak base theory $\mathsf{RCA}_{0}$ used to define the logical strength of
mathematical theorems as described above. We then define a few other common
systems that will be used later. The formal details can be found in \cite{sim_book}.

Each axiomatic subsystem of second order arithmetic that we consider contains
the standard basic axioms for $+$, $\times$, and $<$ (which say that $N$ is a
discrete ordered semiring) and an Induction Axiom:

\medskip%
\begin{tabular}
[c]{rl}%
($\mathsf{I}_{0}$) & $(\forall X)((0\in X \land \forall n\,(n\in X\rightarrow
n+1\in X))\rightarrow\forall n\,(n\in X))$.
\end{tabular}
\medskip

Typically axiom systems for second order arithmetic are defined by adding
various types of set existence axioms although at times additional induction
axioms are used as well. In order to define them we need to specify various
standard syntactic classes of formulas determined by quantifier complexity. As
usual, we add to our language bounded quantifiers $\forall x<t$ and $\exists
x<t$ for first order (i.e.\ arithmetic) terms $t$ defined in the standard way.
We typically denote formulas by capital Greek letters except that the indexed
$\Phi_{e}$ and $\Phi_{e}^{X}$ refer, as above, to our fixed enumeration of the
Turing machines and associated partial functions.

\begin{defn}
The $\Sigma_{0}^{0}$ and $\Pi_{0}^{0}$ formulas of second order arithmetic are
just the ones with only bounded quantifiers but we allow parameters for
elements of either $N$ or $S(\mathcal{N})$ when working with a structure
$\mathcal{N}$. Proceeding inductively, a formula $\Phi$ is $\Sigma_{n+1}^{0}$
($\Pi_{n+1}^{0}$) if it is of the form $\exists x\Psi$ ($\forall x\Psi$) where
$\Psi$ is $\Pi_{n}^{0}$ ($\Sigma_{n}^{0}$). We assume some computable coding
of all these formulas (viewed as strings of symbols from our language) by
natural numbers. We say $\Phi$ is \emph{arithmetic }if it is $\Sigma_{n}^{0}$
or $\Pi_{n}^{0}$ for some $n\in\mathbb{N}$. It is $\Sigma_{1}^{1}$ ($\Pi
_{1}^{1}$) if it is of the form $\exists X\Psi$ ($\forall X\Psi$) where $\Psi$
is arithmetic. (One can continue to define $\Sigma_{n}^{1}$ and $\Pi_{n}^{1}$
in the natural way but we will not need to consider such formulas here.) We
say a set $X$ is in one of these classes $\Gamma$ relative to $A$ (i.e.\ with
$A$ as a parameter) if there is a formula $\Psi(n,A)\in\Gamma$ such that $n\in
X\Leftrightarrow\Psi(n,A)$. If $X$ is both $\Sigma_{n}^{i}$ in $A$ and
$\Pi_{n}^{i}$ in $A$ it is called $\Delta_{n}^{i}$ in $A$.
\end{defn}

We mention a few additional standard connections between the syntactic
complexity of the definition of a set $X$ and $X$'s properties in terms of
computability and graph theoretic notions. They can all be found in \cite{rogers}.

\begin{prop}
\label{jumpsformulas}The sets $A^{(n)}$ are $\Sigma_{n}^{0}$ in $A$. A set $X$
is computable in $A$ if and only if it is $\Delta_{1}^{0}$ in $A$. More
generally, it is computable in $A^{(n)}$ if and only if it is $\Delta
_{n+1}^{0}$ in $A$. It is hyperarithmetic in $A$ if and only if it is
$\Delta_{1}^{1}$ in $A$. There is a computable function $f(e,n)$ such that if
$X$ is $\Sigma_{1}^{1}$ in $A$ via the $\Sigma_{1}^{1}$ formula with code $e$
then for every $n$, $\Phi_{f(e,n)}^{A}$ is (the characteristic function of) a
tree $T$ such that $n\in X\Leftrightarrow T$ has a branch.
\end{prop}

The first system for analyzing the proof theoretic strength of theorems and
theories in reverse mathematics is just strong enough to prove the existence
of the computable sets and so supplies us with all the usual computable
functions such as pairing $\left\langle n,m\right\rangle $ or more generally
those coding finite sequences as numbers. In particular, it provides the
predicates defining the (codes $e$ of) the partial computable functions
$\Phi_{e}^{X}$ and the relations saying the computation $\Phi_{e}^{X}(n)$
halts in $s$ many steps with output $y$. Thus we have the basic tools to
define and discuss Turing reducibility and the Turing jump. It is our weak
base theory and is assumed to be included in every system we consider.

($\mathsf{RCA}_{0}$) Recursive Comprehension Axioms: In addition to the ones
mentioned above, its axioms include the schemes of recursive (generally called
$\Delta_{1}^{0}$) comprehension and $\Sigma_{1}^{0}$ induction:

\medskip%
\begin{tabular}
[c]{rl}%
($\Delta_{1}^{0}$-$\mathsf{CA}$) & $\forall n\,(\Phi(n)\leftrightarrow
\Psi(n))\rightarrow\exists X\,\forall n\,(n\in X\leftrightarrow\Phi(n))$ for
all\\
& $\Sigma_{1}^{0}$ formulas $\Phi$ and $\Pi_{1}^{0}$ formulas $\Psi$ in which
$X$ is not free.\\
($\mathsf{I}\Sigma_{1}^{0}$) & $(\Phi(0) \land \forall n\,(\Phi(n)\rightarrow\Phi
(n+1)))\rightarrow\forall n\,\Phi(n)$ for all $\Sigma_{1}^{0}$ formulas $\Phi$.
\end{tabular}
\medskip

Note that these formulas may have free set or number variables. As usual, the
existence assertion $\exists X....$ of the axiom is taken to mean that for
each instantiation of the free variables (by numbers or sets, as appropriate,
called \emph{parameters}) there is an $X$ as described. We take this for
granted as well as the restriction that the $X$ is not free in the rest of the
formula in all of the set existence axioms of any of our systems.

$\RCA_0$ suffices
to define and manipulate the basic notions of computability listed at the
beginning of Section \ref{comph}. The standard
models of $\RCA_0$ are just those whose second order part is closed under
Turing reduction and disjoint union ($X\oplus Y=\{\left\langle
0,x\right\rangle \mid x\in X\}\cup\{\left\langle 1,y\right\rangle \mid y\in
Y\}$). As suggested above what are now often called the computable in $A$ sets
which are, as mentioned above, the $\Delta_{1}^{0}$ in $A$ sets, were
originally called the sets recursive in $A$. Hence the terminology in
{{\textsf{RCA}}$_{0}$. }

Any axiom system we consider from now on will be assumed to include $\RCA_0$. If we have some axiom scheme or principle $\mathsf{ABC}$ we typically denote the system formed by adding it to
$\RCA_0$ by $\mathsf{ABC}_{0}$. We next move up to the
arithmetic comprehension axiom and its system. \medskip

($\mathsf{ACA}$) $\exists X\,\forall n\,(n\in X\leftrightarrow\Phi(n))$
for every arithmetic formula $\Phi$. \medskip

As mentioned above the $X^{(n)}$ are defined by a $\Sigma_{n}^{0}$ formula
with $X$ as a parameter. So one can show that this system is equivalent (over
{{\textsf{RCA}}$_{0}$)} to the totality of the Turing jump operator, i.e.\ for
every $X$, $X^{\prime}$ exists. Its standard models are those of
{\textsf{RCA}}$_{0}$ whose second order part is also closed under Turing
jump. It is also equivalent (in the sense of reverse mathematics) to
K\"{o}nig's Lemma, which asserts that every finitely branching tree with paths of arbitrarily
long length has a branch.

In general, we say one system of axioms $S$ is \emph{logically or reverse
mathematically reducible} to another $T$ over one $R$ if $R\cup T\vdash\psi$
for every sentence $\psi\in S$. Note that $S$ and/or $T$ may be a single
sentence or theorem. We say that $S$ and $T$ are \emph{equivalent} over $R$ if each
is reducible to the other. If no system $R$ is specified we assume that
{{\textsf{RCA}}$_{0}$ is intended.}

As we will not deal with it, we have omitted the formal definition of the
usual system \textsf{WKL}{$_{0}$ which falls strictly between {\textsf{RCA}%
}$_{0}$ and {\textsf{ACA}}$_{0}$. It is characterized by the restriction of
K\"{o}nig's Lemma to trees that are subsets of }$2^{<N}$, the tree of finite
binary strings under extension.

The next system of the five basic ones after {{\textsf{ACA}}$_{0}$ }is
{{\textsf{ATR}}}$_{0}$. Its defining axiom says that arithmetic
comprehension can be iterated along any countable well-order and so implies
the existence of the sets hyperarithmetic in $X$ for each $X$ but is
computationally stronger than this assumption. As usual the formal definition
can be found in \cite{sim_book}.

Instead, we formally describe the computationally defined class of
theorems/theories that are the main focus of this paper and include several
variations of $\mathsf{IRT}_{{}}$. The definition is semantic, not axiomatic
and involves only standard models. (Indeed by Van Wesep \cite[2.2.2]{vw_thesis},
there can be no axiomatic characterization of this class in second order arithmetic.)

\begin{defn}
\label{thadef}A sentence (theory) $T$ is a \emph{theorem (theory) of
hyperarithmetic analysis (THA)} if

\begin{enumerate}
\item For every $X\subseteq\mathbb{N}$, $\seq{\mathbb{N},\mathrm{HYP}(X)} \vDash T$ and

\item For every $S\subseteq2^{\mathbb{N}}$, if $\seq{\mathbb{N},S} \vDash T$ and
$X\in S$ then $\mathrm{HYP}(X)\subseteq S$.
\end{enumerate}
\end{defn}

It is worth pointing out some of the relations between THAs and $\ATR_0$. THAs
are defined by only using standard models and iterations of the jump over true
well orderings. $\ATR_0$ talks about all models of $\RCA_0$ and asserts the
existence of iterates of the jump over all orderings that appear well-founded
in the model. Thus for standard models it implies the second clause of the
definition of THAs (Definition \ref{thadef}). On the other hand, there is a
recursive linear order with no hyperarithmetic infinite descending sequence
and so it seems well-founded in $\mathrm{HYP}$ but it has a well-founded part
longer than every recursive ordinal. Thus iterating the jump along this
ordering would yield a set strictly Turing above every hyperarithmetic set. In
particular, $\mathrm{HYP}$ is not a model of $\ATR_0$ which therefore is not
a THA. (See e.g.\ \cite[V.2.6]{sim_book}.)

The last of the standard axiomatic systems, $\Pi_{1}^{1}$-{\textsf{CA}}$_{0}$,
is characterized by the comprehension axiom for $\Pi_{1}^{1}$ formulas:

\medskip($\Pi_{1}^{1}$-{\textsf{CA}}) $\exists X\,\forall k\,(k\in
X\leftrightarrow\Phi(k))$ for every $\Pi_{1}^{1}$ formula $\Phi(k)$.

\begin{rmk}
\label{pi11O}The hyperjump, $T^{X}$, is clearly a $\Pi_{1}^{1}$ set with
parameter $X$. In fact, every $\Pi_{1}^{1}$ set with parameter $X$ is
reducible to $T^{X}$. Indeed, there is a computable function $f(e,n)$ such
that for every index $e$ for a $\Pi_{1}^{1}$ formula $\Psi(n)$ with parameter
$X$ and every $n$, $\Psi(n)\Leftrightarrow f(e,n)\in T^{X}$ \cite[Corollary 16.XX(b)]{rogers}. Thus $\Pi_{1}^{1}$-$\mathsf{CA}_{0}$
corresponds to closure under the hyperjump. We will see it appear as
equivalent to a version of $\mathsf{IRT}_{{}}$ where we ask for a maximal set
of disjoint rays in Theorem \ref{thm:MIRT_equiv_Pi11-CA}.
\end{rmk}

For this paper, the most important other existence axiom is a restricted form
of the axiom of choice.

\medskip($\Sigma_{1}^{1}$-\textsf{AC}) $\forall n\exists X\Phi(n,X)\rightarrow
\exists X\forall n\Phi(n,X^{[n]})$ where $\Phi$ is arithmetic and
$X^{[n]}=\{m\mid\langle n,m\rangle\in X\}$ is the $nth$ \emph{column of} $X$.
\medskip

A more common but clearly equivalent version of this axiom allows $\Phi$ to be
$\Sigma_{1}^{1}$. A variant commonly called weak-$\Sigma_{1}^{1}$-\textsf{AC} (introduced in Definition \ref{defn:unique_choice} as unique-$\Sigma_{1}^{1}%
$-\textsf{AC}) requires the corresponding $\Phi$ to be
arithmetic. To make these and other choice axioms uniform we have adopted the
format required elsewhere and equivalent here to be used for all the
variations. The system $\Sigma_{1}^{1}$-\textsf{AC}$_{0}$ {is well known to be a THA
(essentially in Kreisel \cite{kreisel}). Thus it is strictly stronger than
$\ACA_0$. On the other hand, it is strictly weaker than $\ATR_0$. (It is known that $\ATR_0 \vdash \Sigma_{1}^{1}%
$-\textsf{AC}$_{0}$ \cite[V.8.3]{sim_book} but the converse fails as $\Sigma_{1}^{1}$-\textsf{AC} is true in HYP while $\ATR_0$, as we have pointed out, is not.) This choice axiom plays a crucial role in our analysis because we provide
the upper bound on the strength of most of our theorems by showing that they
follow from $\Sigma_{1}^{1}$-\textsf{AC}$_{0}$. This provides the
computational upper bound for being a THA as any consequence of a THA must also satisfy Definition \ref{thadef}(1). Thus the bulk of our proofs for the computational
complexity of the theorems we study consist of showing that they imply
Definition \ref{thadef}(2), i.e.\ closure under \textquotedblleft
hyperarithmetic in\textquotedblright.

Over the past fifty years, several other logical axioms have been shown to be
THA. We will discuss some of them in \S \ref{section:IRT_vs_other_theories}.
However, as we discussed in \S \ref{section:intro}, only one somewhat obscure
purely mathematical theorem was previously known to be a THA. We provide
several more in this paper (Theorem \ref{thm:variants_IRT_hyp_analysis},
Corollary \ref{cor:variants_IRTmc_hyp_analysis} and Theorem
\ref{thm:DED_DVD_directed_forests_equivalence}). We also introduce a new
logical axiom, finite-$\Sigma_{1}^{1}\text{-}\mathsf{AC}$ (Definition
\ref{defn:finite_choice}) which is a THA as well.

For those interested in the proof theory and so nonstandard models, we also at
times explicitly consider the induction axiom at the same $\Sigma_{1}^{1}$ level.

\medskip($\mathsf{I}\Sigma_{1}^{1}$) $(\Phi(0) \land \forall n(\Phi(n)\rightarrow
\Phi(n+1)))\rightarrow\forall n\Phi(n)$ for every $\Sigma_{1}^{1}$ formula
$\Phi$. \medskip

This axiom does not imply the existence of any infinite sets and is, of
course, true in every standard model. Thus the readers interested only in the
computational complexity of the Halin type theorems can safely ignore these considerations.

It is in the nature of reverse mathematics that sentences and sets of
sentences of second order arithmetic are often viewed in several different ways. In different contexts they may be seen as mathematical or logical principles, axioms, axiom schemes, theories, theorems or the like. We point out what may be a less familiar terminology that is currently popular. What might be seen as a typical axiom or theorem asserting that for every $X$ of some sort there is a $Y$ with some relation to $X$, i.e.\ a sentence of the form $\forall X(\Phi(X) \rightarrow \exists Y\Psi(X,Y))$ may be called a principle. With this terminology come the notions of an instance of the principle, i.e.\ an $X$ satisfying $\Phi$ and a solution for $X$, i.e.\ a $Y$ such that $\Psi(X,Y)$ holds.

\section{$\mathsf{IRT}_{}$ and Hyperarithmetic Analysis}

\label{section:IRT_hyp_analysis}

We devote this section to the proof of

\begin{thm}
\label{thm:IRT_hyp_analysis} $\mathsf{IRT}_{{}}$ is a theorem of
hyperarithmetic analysis.
\end{thm}

In this section we consider only vertex-disjoint single rays in undirected
graphs, as in the statement of $\mathsf{IRT}$. The proof of Theorem
\ref{thm:IRT_hyp_analysis} will be split into two parts. The first part
verifies that $\mathsf{IRT}_{{}}$ satisfies the second clause of Definition
\ref{thadef}.

\begin{thm}
\label{thm:omega_model_IRT_hyp_closed} Every standard model of $\mathsf{RCA}%
_{0} + \mathsf{IRT}_{}$ is closed under hyperarithmetic reduction.
\end{thm}

\begin{proof}
Fix a standard model $\mathcal{M}$ of $\mathsf{RCA}_{0}+\mathsf{IRT}$. First,
we show that $\mathcal{M}$ contains $\emptyset^{\prime}$. By relativizing the
proof, it follows that $\mathcal{M}$ is closed under Turing jump.

For each $n$, consider the tree $T_{n}\subseteq N^{<N}$ consisting of all
strings of the form $s^{\smallfrown}0^{t}$ such that some number below $n$ is
enumerated into $\emptyset^{\prime}$ at stage $s$, and either $t\leq s$ or
$\emptyset_{s}^{\prime}\upharpoonright n=\emptyset_{t}^{\prime}\upharpoonright
n$. Observe that $T_{n}$ has a unique computable branch $\{s^{\smallfrown
}0^{t}\mid t\in N\mathbb{\}}$, where $s$ is the smallest number such that
$\emptyset^{\prime}\upharpoonright n=\emptyset_{s}^{\prime}\upharpoonright n$.

Consider the disjoint union $\bigsqcup_{n} T_{n}$. Observe that $\bigsqcup_{n}
T_{n}$ satisfies the premise of $\mathsf{IRT}_{}$ (in $\mathcal{M}$), because
each $T_{n}$ has a computable branch. Apply $\mathsf{IRT}_{}$ to
$\bigsqcup_{n} T_{n}$ to obtain a sequence $\seq{R_{i}}_{i}$ of disjoint rays in
$\bigsqcup_{n} T_{n}$. Each $R_{i}$ is contained in some $T_{n}$. We can,
uniformly in $i$, extend or truncate $R_{i}$ to the unique branch $P_{n}$ of
$T_{n}$. Hence $\seq{R_{i}}_{i}$ computes a sequence of infinitely many distinct
branches $P_{n}$, which in turn computes longer and longer initial segments of
$\emptyset^{\prime}$. This proves that $\mathcal{M}$ contains $\emptyset
^{\prime}$.

Next we show that, for each computable limit ordinal $\lambda$, if $\mathcal{M}$ contains $\emptyset^{(\alpha)}$ for every $\alpha<\lambda$ then $\mathcal{M}$ contains
$\emptyset^{(\lambda)}$. (Again the desired result follows by relativization.)
By Proposition \ref{jumpstrees}, there is a computable sequence $\seq{T_{\beta}}_{\beta<\lambda}$ of trees such that each tree has exactly one branch
$P_{\beta}\equiv_{T}\emptyset^{(\beta)}$ with these reductions computed
uniformly. Fix an increasing computable sequence $\seq{\alpha_{n}}_{n}$ which is
cofinal in $\lambda$ and consider the disjoint union $\bigsqcup_{n}%
T_{\alpha_{n}}$. Observe that $\bigsqcup_{n}T_{\alpha_{n}}$ satisfies the
premise of $\mathsf{IRT}_{{}}$ (in $\mathcal{M}$): for each $n$,
$\emptyset^{(\alpha_{n})}$ computes the branches $P_{\alpha_{m}} $ for $m\leq
n$. Apply $\mathsf{IRT}_{{}}$ to $\bigsqcup_{n}T_{\alpha_{n}}$ to obtain a
sequence $\seq{R_{i}}_{i}$ of disjoint rays in $\bigsqcup_{n}T_{\alpha_{n}}$. As
before, $\seq{R_{i}}_{i}$ computes a sequence of infinitely many distinct branches
$P_{\alpha_{n}}$, and hence a sequence of infinitely many distinct
$\emptyset^{(\alpha_{n})}$. Each $\emptyset^{(\alpha_{n})}$ uniformly computes
$\emptyset^{(\alpha_{m})}$ for $m\leq n$, so we conclude that $\seq{R_{i}}_{i}$
computes $\bigoplus_{m}\emptyset^{(\alpha_{m})} $ as desired.
\end{proof}

It follows that $\mathsf{IRT}$ is not provable in $\mathsf{ACA}_{0}$, despite
the apparent similarity between $\mathsf{IRT}$ and a compactness result. (Indeed, $\IRT{}$ is not even provable in $\Delta_{1}^{1}$-$\CA_0$ (Theorem \ref{thm:Delta11-CA_not_imply_IRT}).)

Next, we present essentially the proof of $\mathsf{IRT}_{{}}$ attributed to
Andreae (see \cite[Theorem 8.2.5 and bottom of pg.\ 275]%
{diestel_book}) and then analyze it with an eye to the axioms which can be
used to formalize it. We then use this analysis to complete the proof of
Theorem \ref{thm:IRT_hyp_analysis}.

The key combinatorial lemma implicit in Andreae's proof is:

\begin{lem}
\label{lem:more_disjoint_rays} Given disjoint rays $\left\langle
R_{i}\right\rangle _{i<n}$ and disjoint rays $\left\langle S_{j}\right\rangle
_{j<n^{2}+1}$ there are $n+1$ disjoint rays $R_{0}^{\prime},\dots
,R_{n}^{\prime}$ such that for each $i<n$, $R_{i}$ and $R_{i}^{\prime}$ start
at the same vertex.
\end{lem}

Before proving Lemma \ref{lem:more_disjoint_rays}, let us use it to prove
$\mathsf{IRT}$.

\begin{proof}
[Proof of $\mathsf{IRT}_{{}}$ assuming Lemma \ref{lem:more_disjoint_rays}%
]Given a graph which has arbitrarily many disjoint rays, we build by recursion
on $n\geq1$ sequences $\left\langle R_{i}^{n}\right\rangle _{i<n}$ of disjoint
rays with initial segments $P_{i}^{n}$ of length $n$ such that $P_{i}^{n+1}$
is $P_{i}^{n}$ followed by one more vertex for $i<n$. The required infinite
sequence of disjoint rays will then be given by $R_{i}=\bigcup\{P_{i}^{n}\mid n>0\}$.

Suppose that we have $\left\langle R_{i}^{n}\right\rangle _{i<n}$ and
$\left\langle P_{i}^{n}\right\rangle _{i<n}$. By assumption, let $S_{0}%
,\dots,S_{2n^{2}}$ be a sequence of disjoint rays. Discard all rays $S_{j}$
which contain a vertex of some $P_{i}^{n}$. There are at most $n^{2}$ many of
them, so by discarding and renumbering if necessary we are left with
$S_{0},\dots,S_{n^{2}}$.

For each $i<n$, let $x_{i}$ denote the first vertex on $R_{i}^{n}$ after
$P_{i}^{n}$. Apply Lemma \ref{lem:more_disjoint_rays} to $\left\langle
x_{i}R_{i}^{n}\right\rangle _{i<n}$ and $S_{0},\dots,S_{n^{2}}$. We obtain
$n+1$ disjoint rays $R_{0}^{\prime},\dots,R_{n}^{\prime}$ such that for each
$i<n$, $R_{i}^{\prime}$ begins at vertex $x_{i}$. Now let $R_{i}^{n+1}%
=P_{i}^{n}x_{i}R_{i}^{\prime}$ for $i<n$ and $R_{n}^{n+1}=R_{n}^{\prime}$.
These are disjoint by construction. This completes the inductive step of the
construction of the $R_{0}^{n},\dots,R_{n-1}^{n}$ and so provides the required
witnesses for $\mathsf{IRT}$.
\end{proof}

It remains to prove Lemma \ref{lem:more_disjoint_rays}. The key ingredient is
Menger's theorem for finite graphs. If $A$ and $B$ are disjoint sets of
vertices in a graph, we say that $P$ is an \emph{$A$-$B$ path} if $P$ starts
with some vertex in $A$ and ends with some vertex in $B$. A set of vertices
$S$ \emph{separates $A$ and $B$} if any $A$-$B$ path contains at least one
vertex in $S$.

\begin{thm}
[{Menger, see \cite[Theorem 3.3.1]{diestel_book}}]\label{thm:menger_finite}
Let $G$ be a finite graph. If $A$ and $B$ are disjoint sets of vertices in
$G$, then the minimum size of a set of vertices which separate $A$ and $B$ is
equal to the maximum size of a set of disjoint $A$-$B$ paths.
\end{thm}

We now present the proof of Lemma \ref{lem:more_disjoint_rays}.

\begin{proof}
[Proof of Lemma \ref{lem:more_disjoint_rays}]Suppose we are given $n$ disjoint
rays $R_{0},\dots,R_{n-1}$ and $n^{2}+1$ disjoint rays $S_{0},\dots,S_{n^{2}}%
$. First, define the set
\[
\{\langle i,q\rangle \mid R_{i}\text{ intersects }S_{q}\}.
\]

Then we perform the following recursive procedure. At each step, check if
there is some $i < n$ such that $R^{\prime}_{i}$ has not been defined and
$R_{i} $ intersects at most $n$ many rays $S_{q}$ which have not been
discarded. If there is no such $i$, we end the procedure. Otherwise, find the
least such $i $ and do the following:

\begin{enumerate}
\item discard all rays $S_{q}$ which intersect $R_{i}$;

\item define $R^{\prime}_{i} = R_{i}$.
\end{enumerate}

After the procedure is complete, let $I$ be the set of $i<n$ for which
$R_{i}^{\prime}$ has not been defined. Let $\mathcal{S}$ be the set of rays
$S_{q}$ which have not been discarded. Let $m=|I|$. We observe that
$|\mathcal{S}|\geq m^{2}+1$, because
\[
(n^{2}+1)-(n-m)n=mn+1\geq m^{2}+1.
\]

Next, for each $i\in I$, let $z_{i}$ be the first vertex on $R_{i}$ such that
$R_{i}z_{i}$ meets exactly $m$ many rays in $\mathcal{S}$. (Each $z_{i}$
exists by construction of $I$.)

Observe that the finite set $\bigcup_{i\in I}R_{i}z_{i}$ meets at most $m^{2}
$ many rays in $\mathcal{S}$. Since $|\mathcal{S}|\geq m^{2}+1$, we may pick
some ray in $\mathcal{S}$ which does not meet $\bigcup_{i\in I}R_{i}z_{i}$. We
define $R_{n}^{\prime}$ to be said ray. Then, discard all rays in
$\mathcal{S}$ which do not meet $\bigcup_{i\in I}R_{i}z_{i}$.

Finally, we use Menger's theorem to define $R_{i}^{\prime}$ for each $i\in I
$. For each $i\in I$, let $x_{i}$ denote the first vertex of $R_{i}$. For each
$q$ such that $S_{q}$ remains in $\mathcal{S}$, let $y_{q}$ be the first
vertex on $S_{q}$ such that $y_{q}S_{q}$ and $\bigcup_{i\in I}R_{i}z_{i}$ are
disjoint. Then consider the following finite sets of vertices:
\begin{align*}
X  &  =\{x_{i} \mid i\in I\}\\
Y  &  =\{y_{q} \mid S_{q}\in\mathcal{S}\}\\
H  &  =\bigcup_{i\in I}R_{i}z_{i}\cup\bigcup_{S_{q}\in\mathcal{S}}S_{q}y_{q}.
\end{align*}

We want to apply Menger's theorem to $X,Y \subseteq H$. Towards that end, we
claim that $X$ cannot be separated from $Y$ in $H$ by fewer than $m$ vertices.

Suppose that $A\subseteq H$ and $|A|<m$. Since $|I|=m$ and $\{R_{i} \mid i\in I\}$
is disjoint, there is some $i\in I$ such that $R_{i}$ does not meet $A$. Next,
since $R_{i}z_{i}$ meets $m$ many disjoint rays in $\mathcal{S}$, there is
some $q$ such that $S_{q}\in\mathcal{S}$ and $R_{i}z_{i}$ meets $S_{q}$, but
$S_{q}$ does not meet $A$. Let $z$ be any vertex in both $R_{i}z_{i}$ and
$S_{q}$. Then $R_{i}zS_{q}y_{q}$ is a path in $H$ from $x_{i} $ to $y_{q}$
which does not meet $A$. This proves our claim.

By Menger's theorem, there are $m$ many disjoint $X$-$Y$ paths in $H$. Then,
for each $i\in I$, define $R_{i}^{\prime}$ by starting from $x_{i}$, then
following the $X$-$Y$ path given by Menger's theorem to some $y_{q}$, and
finally following $S_{q}$.

We have constructed a collection $R^{\prime}_{0},\dots,R^{\prime}_{n}$ of
rays. It is straightforward to check that they are disjoint, and that for each
$i < n$, $R_{i}$ and $R^{\prime}_{i}$ start at the same vertex.
\end{proof}

We now analyze these proofs from a reverse mathematical perspective to show
that $\mathsf{IRT}_{{}}$ follows from the THA $\Sigma_{1}^{1}$-$\mathsf{AC}%
_{0}$ and so also satisfies the first clause of Definition \ref{thadef}.

\begin{thm}
\label{thm:Sigma11-AC_implies_IRT} i) $\mathsf{IRT}_{{}}$ and $\Sigma_{1}^{1}%
$-$\mathsf{AC}_{0}$ each implies $\mathsf{ACA}_{0}$.
\end{thm}

ii) $\Sigma_{1}^{1}$-$\mathsf{AC}_{0}$ implies $\mathsf{IRT}$. Hence for
every $Y\subseteq\mathbb{N}$, $\mathrm{HYP}(Y)$ satisfies $\mathsf{IRT}$.

\begin{proof}
i) The proof of the first part of Theorem \ref{thm:omega_model_IRT_hyp_closed}
essentially shows that $\mathsf{IRT}_{{}}$ implies $\mathsf{ACA}_{0}$. One point to note is that for each $n$, $\mathsf{RCA}_{0}$ proves
there is some $s$ such that $\emptyset^{\prime}\upharpoonright n=\emptyset_{s}^{\prime}\upharpoonright n$ (using e.g.\ \cite[II.3.9]{sim_book}). Now
note that the same argument also proves the fact (essentially in
{\cite{kreisel}) that} $\Sigma_{1}^{1}$-$\mathsf{AC}_{0}$ implies
$\mathsf{ACA}_{0}$ as it directly supplies the sequence of branches $P_{n}$ in
$T_{n}$.

ii) The proof of $\mathsf{IRT}_{{}}$ presented above is easily seen to be one
in $\mathsf{ACA}_{0}$, except for two points. First, each step $n$ of the
induction assumed we had available a sequence $S_{0},\dots,S_{n^{2}}$ of
disjoint rays in our graph. However, all we know is that for each $n$, there
is some collection of disjoint rays of size $n$. To access this information
for each $n$ in a recursive construction, we require that there is a single
sequence $R_{n}$ such that the $n$th entry of the sequence is a
collection of disjoint rays of size $n$. Such a sequence can be obtained using
the axiom of choice. In this case, since the predicate \textquotedblleft there
exists $n$ many disjoint rays\textquotedblright\ is $\Sigma_{1}^{1}$, such a
sequence can be obtained using $\Sigma_{1}^{1}$-$\mathsf{AC}_{0}$. Therefore
we may assume we have a sequence $S=\seq{\seq{S_{j}^{n}} _{j<n}}_{n>0}$ with each
$\seq{S_{j}^{n}}_{j<n}$ a sequence of disjoint rays and
the graph $G$ consisting of all the vertices and edges occurring in any
$S_{j}^{n}$. We begin our construction with $R_{0}^{1}=S_{0}^{1}$.

The second point is the iterated application of Lemma
\ref{lem:more_disjoint_rays}. The proof of that Lemma can be done
in $\mathsf{ACA}_{0}$. (In particular, one can check that Menger's theorem is provable in $\mathsf{RCA}_{0}$ by following the first proof for it given in {\cite[Theorem 3.3.1]{diestel_book}}.) However, here we need a bit more information as
generally we cannot, in $\mathsf{ACA}_{0}$, carry out recursive constructions
that increase complexity (even by some fixed number of jumps) at each step. We
can only carry out recursions where each step is done computably in some set
of the model. So it suffices to show that the constructions for all the
instances of Lemma \ref{lem:more_disjoint_rays} needed in its iterations in
the proof can be done computably in $(G\oplus S)^{\prime}$.

The crucial observation is that, by an induction starting with $R_{0}%
^{1}=S_{0}^{1}$, we may take each $R_{i}^{n}$ to be of the form $PQ$ where $P$
is a finite path in $G$ and $Q$ is a tail of one of the $S_{j}$. To see this
simply note that each ray $R_{k}^{\prime}\neq R_{k}$ constructed in the Lemma
starts with an initial segment of some $R_{i}$; it continues with a finite path
in $G$ and ends with a tail of some $S_{j}^{2n^{2}+1}$. Thus computably in
$(G\oplus S)^{\prime}$ we can at every stage find one of these descriptions
that provides the disjoint rays needed for the next stage of the recursion.
Thus the sequences $\left\langle \left\langle R_{i}^{n}\right\rangle
_{i<n}\right\rangle _{n>0}$ and $\left\langle \left\langle P_{i}%
^{n}\right\rangle _{i<n}\right\rangle _{n>0}$ are computable from $(G\oplus
S)^{\prime}$. The sequence $\left\langle
\bigcup\{P_{i}^{n}\mid n>0\}\right\rangle _{i\in N}$ is then also computable from
$(G\oplus S)^{\prime}$ and computing the formal presentation of infinitely
many disjoint rays then takes only one more jump. This completes the analysis
of our proof of $\mathsf{IRT}$ in $\Sigma_{1}^{1}$-$\mathsf{AC}_{0}$.
\end{proof}


Theorem \ref{thm:IRT_hyp_analysis} follows from Theorems
\ref{thm:omega_model_IRT_hyp_closed} and \ref{thm:Sigma11-AC_implies_IRT}.

We will establish some implications and nonimplications between $\mathsf{IRT}%
_{{}}$ and other THAs in
\S \ref{section:IRT_vs_other_theories}.

\section{Variants of $\mathsf{IRT}_{}$ and Hyperarithmetic Analysis}

\label{section:variants_IRT}

In this section, we show that at least five of the eight principles
$\mathsf{IRT}_{\mathrm{XYZ}}$ are THAs:

\begin{thm}
\label{thm:variants_IRT_hyp_analysis} All single-ray variants of
$\mathsf{IRT}_{}$ (i.e., $\mathsf{IRT}_{\mathrm{XYS}}$) and $\mathsf{IRT}%
_{\mathrm{UVD}}$ are theorems of hyperarithmetic analysis.
\end{thm}

$\mathsf{IRT}_{\mathrm{UVS}}$ and $\mathsf{IRT}_{\mathrm{UVD}}$ were proved by
Halin \cite{halin65, halin70}. $\mathsf{IRT}_{\mathrm{UES}}$ is an exercise in
\cite[8.2.5(ii)]{diestel_book}. $\mathsf{IRT}_{\mathrm{DVS}}$ and
$\mathsf{IRT}_{\mathrm{DES}}$ may be folklore.

Of the other three variants, $\mathsf{IRT}_{\mathrm{DED}}$ and $\mathsf{IRT}%
_{\mathrm{DVD}}$ are open problems of graph theory (\cite{bcp15} and Bowler,
personal communication). We do, however, have interesting results about these
principles when restricted to directed forests (Theorem
\ref{thm:DED_DVD_directed_forests_equivalence}, Corollary
\ref{cor:IRTmc_DVD_strictly_implies_SigmaAC}). The other one, $\mathsf{IRT}%
_{\mathrm{UED}}$, was proved by Bowler, Carmesin, Pott \cite{bcp15} using
structural results about ends. We hope to analyze its strength in future work.

The proof of Theorem \ref{thm:variants_IRT_hyp_analysis} consists of several
variations of the proof of Theorem \ref{thm:IRT_hyp_analysis}. One of which
($\mathsf{IRT}_{\mathrm{DES}}$) requires some additional ideas.

In order to minimize repetition, we establish some implications between some
variants of $\mathsf{IRT}_{{}}$ over $\mathsf{RCA}_{0}$. The proofs of each of
these reductions follow the same basic plan. To deduce $\mathsf{IRT}%
_{\mathrm{XYZ}}$ from $\mathsf{IRT}_{\mathrm{X^{\prime}Y^{\prime}Z^{\prime}}}$
we provide computable maps $g$, $h$ and $k$ which, provably in $\mathsf{RCA}%
_{0}$, take $X$-graphs $G$ to $X^{\prime}$-graphs $G^{\prime}$, $Y$-disjoint
$Z$-rays or sets of $Y$-disjoint $Z$-rays in $G$ to $Y^{\prime}$-disjoint
$Z^{\prime}$-rays or sets of $Y^{\prime}$-disjoint $Z^{\prime}$-rays in
$G^{\prime}$, and $Y^{\prime}$-disjoint $Z^{\prime}$-rays or sets of
$Y^{\prime}$-disjoint $Z^{\prime}$-rays in $G^{\prime}$ to $Y$-disjoint $Z$-rays or sets of $Y$-disjoint $Z$-rays in $G$, respectively. These functions are
designed to take witnesses of the hypothesis of $\mathsf{IRT}_{\mathrm{XYZ}}$
in $G$ to witnesses of the hypothesis of $\mathsf{IRT}_{\mathrm{X^{\prime
}Y^{\prime}Z^{\prime}}}$ in $G^{\prime}$ and witnesses to the conclusion of
$\mathsf{IRT}_{\mathrm{X^{\prime}Y^{\prime}Z^{\prime}}}$ in $G^{\prime}$ to
witnesses to the conclusion of $\mathsf{IRT}_{\mathrm{XYZ}}$ in $G$. Clearly
it suffices to provide such computable maps to establish the desired reduction
in $\mathsf{RCA}_{0}$. (Those familiar with Weihrauch reducibility will recognize that these arguments establish Weihrauch reductions between certain problems corresponding to variants of $\IRT{}$.)

Unless otherwise noted all definitions and proofs in this section are in
$\mathsf{RCA}_{0}$.

\begin{lem}
\label{lem:DYZ_implies_UYZ} Given an undirected graph $G$, we can uniformly
compute a directed graph $G^{\prime}$ and mappings between Z-rays in $G$ and
Z-rays in $G^{\prime}$ which preserve Y-disjointness.
\end{lem}

\begin{proof}
We define a computable map $g$ from undirected graphs $G$ to directed graphs
$G^{\prime}$ as follows. The set of vertices of $G^{\prime}$ consists of the
vertices of $G$, together with two new vertices $x = x(u,v)$ and $y = y(u,v)$
for each edge $\{u, v\}$ in $G$. The set of edges of $G^{\prime}$ consists of
five edges $\langle u, x \rangle$, $\langle v, x \rangle$, $\langle x, y
\rangle$, $\langle y, u \rangle$, $\langle y, v \rangle$ for each edge $\{u,
v\}$ in $G$.

Next we define a computable map $h_{\mathrm{S}}$: given a ray $u_{0}%
,u_{1},\dots$ in $G$, $h_{\mathrm{S}}$ maps it to the ray $u_{0},x(u_{0}%
,u_{1}),y(u_{0},u_{1}),u_{1},\dots$ in $G^{\prime}$. Conversely, we define a
computable map $k_{\mathrm{S}}$ from rays $R^{\prime}$ in $G^{\prime}$ into
rays $R$ in $G$ as follows. Observe that exactly one of the first three
vertices in $R^{\prime} $ is a vertex in $G$, because the only outgoing edges
from a vertex $y(u,v) $ lead to $u$ or to $v$, and the only outgoing edge from
a vertex $x(u,v)$ leads to $y(u,v)$. We take this vertex (say $u_{0}$) to be
the first vertex of $R$. Every outgoing edge from $u_{0}$ leads to some
$x(u_{0},v)$. Combining the above observations, we deduce that the tail
$u_{0}R^{\prime}$ has the form $u_{0},x(u_{0},u_{1}),y(u_{0},u_{1}%
),u_{1},\dots$. Then $k_{\mathrm{S}}$ maps $R^{\prime} $ to the ray $R =
u_{0},u_{1},\dots$ in $G$.

Similarly, given a double ray $\dots,u_{-1},u_{0},u_{1},\dots$ in $G$,
$h_{\mathrm{D}}$ maps it to the double ray $\dots,u_{-1},x(u_{-1}%
,u_{0}),y(u_{-1},u_{0}),u_{0},x(u_{0},u_{1}),y(u_{0},u_{1}),u_{1},\dots$ in
$G^{\prime}$. We can show that every double ray in $G^{\prime}$ has this form
by considering the incoming edges to each vertex in $G^{\prime}$. Therefore we
can define a computable map from double rays in $G^{\prime}$ to double rays in
$G$ by $k_{\mathrm{D}} = h_{\mathrm{D}}^{-1}$.

It is straightforward to check that $h_{\mathrm{S}}$, $k_{\mathrm{S}}$,
$h_{\mathrm{D}}$ and $k_{\mathrm{D}}$ preserve $Y$-disjointness.
\end{proof}

Therefore we have

\begin{prop}
\label{prop:IRT_DYZ_implies_IRT_UYZ} The directed variants of $\mathsf{IRT}%
_{}$ imply their corresponding undirected variants, i.e., $\mathsf{IRT}%
_{\mathrm{DYZ}}$ implies $\mathsf{IRT}_{\mathrm{UYZ}}$ for each value of $Y$
and $Z$.
\end{prop}

\begin{lem}
\label{lem:DEZ_implies_DVZ} Given a directed graph $G$, we can uniformly
compute a directed graph $G^{\prime}$ and mappings between $Z$-rays in $G$ and
$Z$-rays in $G^{\prime}$ which satisfy the following properties: if two
$Z$-rays in $G$ are vertex-disjoint, then the corresponding $Z$-rays in
$G^{\prime}$ are edge-disjoint, and if two $Z$-rays in $G^{\prime}$ are
edge-disjoint, then the corresponding $Z$-rays in $G$ are vertex-disjoint.
\end{lem}

\begin{proof}
We define a computable map $g$ from directed graphs $G = \langle V,E \rangle$
to directed graphs $G^{\prime}$ as follows. The set of vertices of $G^{\prime
}$ is $\{x_{i}, x_{o} \mid x \in V\}$, where $i$ and $o$ stand for incoming and
outgoing respectively. The set of edges of $G^{\prime}$ consists of $\langle
u_{o}, v_{i} \rangle$ for each $\langle u, v \rangle\in E$, and $\langle
x_{i}, x_{o} \rangle$ for each $x \in V$.

Next we define a computable map $h_{\mathrm{S}}$: given a ray $x^{0}%
,x^{1},\dots$ in $G$, $h_{\mathrm{S}}$ maps it to the ray $x_{i}^{0},x_{o}%
^{0},x_{i}^{1},x_{o}^{1},\dots$ in $G^{\prime}$. Conversely, we define a
computable map $k_{\mathrm{S}}$ from rays $R^{\prime}$ in $G^{\prime}$ to rays
$R$ in $G$ as follows. Given $R^{\prime}$, the ray $R$ visits the vertex $x$
in $G$ whenever $\langle x_{i},x_{o}\rangle$ appears in $R^{\prime}$. (For
example, we map $x_{i}^{0},x_{o}^{0},x_{i}^{1},x_{o}^{1},\dots$ to
$x^{0},x^{1},\dots$ and we map $x_{o}^{0},x_{i}^{1},x_{o}^{1},\dots$ to
$x^{1},x^{2},\dots$.)

Similarly, $h_{\mathrm{D}}$ maps a given double ray $\dots,x^{-1},x^{0}%
,x^{1},\dots$ in $G$ to the double ray
\[
\dots,x^{-1}_{i},x^{-1}_{o},x^{0}_{i},x^{0}_{o},x^{1}_{i},x^{1}_{o},\dots
\]
in $G^{\prime}$. Every double ray in $G^{\prime}$ has this form, so we may
define $k_{\mathrm{D}} = h_{\mathrm{D}}^{-1}$.

It is straightforward to check that the above mappings have the desired properties.
\end{proof}

Therefore we have

\begin{prop}
\label{prop:IRT_DEZ_implies_IRT_DVZ} The directed edge-disjoint variants of
$\mathsf{IRT}_{}$ imply their corresponding directed vertex-disjoint variants,
i.e., $\mathsf{IRT}_{\mathrm{DEZ}}$ implies $\mathsf{IRT}_{\mathrm{DVZ}}$ for
each value of $Z$.
\end{prop}

\begin{lem}
\label{lem:DYD_implies_DYS} Given a directed graph $G$, we can uniformly
compute a directed graph $G^{\prime}$ and mappings between sets of
$Y$-disjoint rays in $G$ and sets of $Y$-disjoint double rays in $G^{\prime}$
which preserve cardinality.
\end{lem}

\begin{proof}
We define a computable map $g$ from directed graphs $G$ to directed graphs
$G^{\prime}$ containing $G$ as follows. For each vertex $x$ of $G$ we add new
vertices $x_{n}$ for each $n<0$ and edges $\left\langle x_{-1},x\right\rangle
$ and $\langle x_{n-1},x_{n}\rangle$ for all $n<0$.

Next we define a computable map $h$ from sets of $Y$-disjoint rays in $G$ to
sets of $Y$-disjoint double rays in $G^{\prime}$ as follows. Given a set of
$Y$-disjoint rays in $G$, we first ensure that each ray begins at a different
vertex, by replacing it with a tail if necessary. (This is only relevant if
the rays are edge-disjoint rather than vertex-disjoint.) Then for each ray
$x^{0},x^{1},\dots$, we consider the double ray $\dots,x_{-2}^{0},x_{-1}%
^{0},x^{0},x^{1},\dots$ in $G^{\prime}$. This yields a set of $Y$-disjoint
double rays in $G^{\prime}$ of the same cardinality.

Finally we define a computable map $k$ from sets of $Y$-disjoint double rays
in $G^{\prime}$ to sets of $Y$-disjoint rays in $G$. Given a double ray
$\dots,x^{-1},x^{0},x^{1},\dots$ in $G^{\prime}$, we search for the least
$n\geq0$ such that $x^{n}$ is a vertex in $G$. (If none of these vertices were
in $G$ then as there are edges between them they would all have to be of the
form $y_{-n}$ for a single $y$ in $G$. The only edges between these $y_{-n}$
make them into a copy of the reverse order on $N$. This order cannot have any
subsequence of order type $N$.) Now we map the given double ray to the ray
$x^{n},x^{n+1},\dots$ in $G$. It is straight forward to check that this map
induces a cardinality-preserving map $k$ from sets of $Y$-disjoint double rays
in $G^{\prime}$ to sets of $Y$-disjoint rays in $G$.
\end{proof}

Therefore we have

\begin{prop}
\label{prop:IRT_DYD_implies_IRT_DYS} The directed double ray variants of
$\mathsf{IRT}_{}$ imply their corresponding directed single ray variants,
i.e., $\mathsf{IRT}_{\mathrm{DYD}}$ implies $\mathsf{IRT}_{\mathrm{DYS}}$ for
each value of $Y$.
\end{prop}

\begin{figure}[ptb]
\centering
\begin{tikzpicture}[scale=1,every node/.style={inner sep=6pt}]
	\node (UVS) at (0,0) {UVS};
	\node (UVD) at (-2,2) {UVD};
	\node (UES) at (2,2) {UES};
	\node (UED) at (0,4) {UED};
	\node (DVS) at (0,2) {DVS};
	\node (DVD) at (-2,4) {DVD};
	\node (DES) at (2,4) {DES};
	\node (DED) at (0,6) {DED};
	\draw [->] (DVD) -- (DVS);
	\draw [->] (DVD) -- (UVD);
	\draw [->] (DES) -- (DVS);
	\draw [->] (DVS) -- (UVS);
	\draw [->] (DES) -- (UES);
	\draw [->] (DED) -- (DVD);
	\draw [->] (DED) -- (DES);
	\draw [->] (DED) -- (UED);
	\draw [->] (UVD) -- node[xshift=-12,yshift=-6] {$I\Sigma^1_1$} (UVS);
\end{tikzpicture}
\caption{Known implications between variants of $\mathsf{IRT}_{}$. All
implications are over $\mathsf{RCA}_{0}$, except for the implication from UVD
to UVS (Theorem \ref{thm:UVDmc_implies_UVS}).}%
\label{fig:variants_IRT}%
\end{figure}
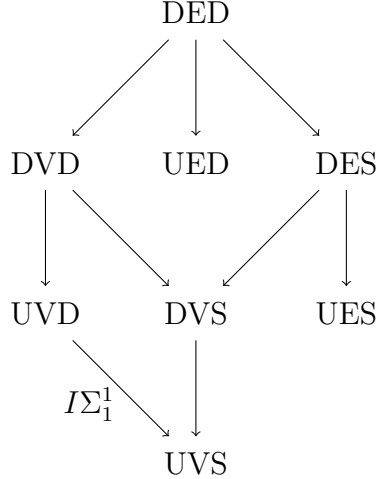

Figure \ref{fig:variants_IRT} summarizes the known implications between
variants of $\mathsf{IRT}_{}$ over $\mathsf{RCA}_{0}$. We will show that
$\mathsf{IRT}_{\mathrm{UVD}}$ implies $\mathsf{IRT}_{\mathrm{UVS}}$ over
$\mathsf{RCA}_{0} + \mathsf{I}\Sigma^{1}_{1}$ (Theorem \ref{thm:UVDmc_implies_UVS}).

\begin{rmk}
\label{rmk:BCP} Bowler, Carmesin, Pott \cite[pg.\ 2 l.\ 3--7]{bcp15} describe
an implication from $\mathsf{IRT}_{\mathrm{UVS}}$ to $\mathsf{IRT}%
_{\mathrm{UES}}$ which appears to use $\Sigma^{1}_{1}\text{-}\mathsf{AC}_{0}$.
It turns out that the graph-theoretic principle used to carry out the
implication does not imply even $\mathsf{ACA}_{0}$ over $\mathsf{RCA}_{0}$
(and is hence much weaker than $\Sigma^{1}_{1}\text{-}\mathsf{AC}_{0}$), but
when combined with $\mathsf{ACA}_{0}$, yields a THA. It and several other principles with the same property (almost
theorems/theories of hyperarithmetic analysis) are analyzed in Shore
\cite{shore_atha}.
\end{rmk}

We return to the goal of proving Theorem \ref{thm:variants_IRT_hyp_analysis}.

\begin{thm}
\label{thm:IRT_variants_ACA_0_and_omega_model_IRT_variants_hyp_closed} For
each choice of XYZ, $\mathsf{IRT}_{\mathrm{XYZ}}$ implies $\mathsf{ACA}_{0}$.
Furthermore, every standard model of $\mathsf{RCA}_{0}$ and $\mathsf{IRT}%
_{\mathrm{XYZ}}$ is closed under hyperarithmetic reduction.
\end{thm}

\begin{proof}
By Proposition \ref{prop:IRT_DYZ_implies_IRT_UYZ}, it suffices to prove the
desired result for the undirected variants of $\mathsf{IRT}$. Theorems
\ref{thm:omega_model_IRT_hyp_closed} and \ref{thm:Sigma11-AC_implies_IRT}(i) together assert the desired result for
$\mathsf{IRT}$. We describe how to modify the proofs of Theorems
\ref{thm:omega_model_IRT_hyp_closed} and \ref{thm:Sigma11-AC_implies_IRT}(i) to prove the desired result for the other variants
of $\mathsf{IRT}$.

Observe that in the aforementioned proofs, we only applied $\mathsf{IRT}_{}$
to forests such that each of the constituent trees has a unique branch. In
such graphs, none of the constituent trees can contain two rays which are
edge-disjoint. Hence the aforementioned proofs establish the desired result
for $\mathsf{IRT}_{\mathrm{UES}}$ as well.

In order to prove the desired result for $\mathsf{IRT}_{\mathrm{UVD}}$ and
$\mathsf{IRT}_{\mathrm{UED}}$, we modify the aforementioned proofs as follows.
For each tree, we relabel as needed to then add one computable branch of new
vertices (other than the root of the tree) such that the new branches are also
disjoint. The resulting trees satisfy the following properties:

\begin{itemize}
\item Each tree contains some double ray which is Turing equivalent to the
branch on the original tree.

\item No two double rays in the tree can be vertex-disjoint or edge-disjoint.

\item Given any double ray in the tree, we can uniformly compute the branch on
the original tree.
\end{itemize}

It is straightforward to check that the modified proofs establish the desired
result for $\mathsf{IRT}_{\mathrm{UVD}}$ and $\IRT{UED}$.
\end{proof}

Henceforth we will not explicitly mention uses of $\mathsf{ACA}_{0}$ whenever
we are assuming any $\mathsf{IRT}_{\mathrm{XYZ}}$ or $\Sigma_{1}^{1}%
$-$\mathsf{AC}_{0}$.

Next, we show that $\mathsf{IRT}_{\mathrm{UVD}}$ and $\mathsf{IRT}%
_{\mathrm{DES}}$ are provable in $\Sigma_{1}^{1}$-$\mathsf{AC}_{0}$ (Theorems
\ref{thm:Sigma11-AC_implies_UVD}, \ref{thm:Sigma11-AC_implies_XES}). It then
follows from Propositions \ref{prop:IRT_DYZ_implies_IRT_UYZ} and
\ref{prop:IRT_DEZ_implies_IRT_DVZ} that $\mathsf{IRT}_{\mathrm{UES}}$,
$\mathsf{IRT}_{\mathrm{DVS}}$ and $\mathsf{IRT}_{\mathrm{UVS}}$ are also
provable in $\Sigma_{1}^{1}\text{-}\mathsf{AC}_{0}$, completing the proof of
Theorem \ref{thm:variants_IRT_hyp_analysis}.

\begin{thm}
\label{thm:Sigma11-AC_implies_UVD} $\Sigma_{1}^{1}$-$\mathsf{AC}_{0}$ implies
$\mathsf{IRT}_{\mathrm{UVD}}$.
\end{thm}

\begin{proof}
The mathematical result is due to \cite{halin65} and is Exercise 42 of Chapter 8
in \cite{diestel_book}. Our proof is very similar to that of Theorem
\ref{thm:Sigma11-AC_implies_IRT}(ii) which follows \cite[Theorem 8.2.5(i)]%
{diestel_book} except we need to grow our family in \textquotedblleft two
directions\textquotedblright.

Our plan is to construct, by recursion on $n$, sequences $\left\langle
R_{i}^{n}\right\rangle _{i<n}$ and $\left\langle P_{i}^{n}\right\rangle
_{i<n}$ such that the $R_{i}^{n}$ are disjoint double rays with subpaths
$P_{i}^{n}$ of length $2n$ such that, for $i<n$, $P_{i}^{n+1}$ is $P_{i}^n
$ extended by a new vertex at each end. Our required sequence of disjoint
double rays is then $\left\langle \bigcup\left\{P_{i}^{n}\mid n\in N\right\}\right\rangle
_{i\in N}$.

As we want to reuse the proof of Lemma \ref{lem:more_disjoint_rays}, we want to
decompose double rays into single rays. To that end we introduce some
notation. If $R$ and $R_{i}$ are double rays we let $R_{f}=R$ and
$R_{i,f}=R_{i}$ while $R_{b}$ and $R_{i,b}$ are $R$ and $R_{i}$, respectively,
but with order reversed. We use $d$ or $d^{\prime}$ to stand for one of $f$ or $b$.
For single rays $R$ we use $^{\ast}\!R$ to denote the reverse sequence of
vertices. So, for example, if $(c,d)$ is an edge in a double ray $R$, we have $^{\ast
}(cR_{b})dR_{f}=R$.

We note the changes needed in the proof of Theorem
\ref{thm:Sigma11-AC_implies_IRT}. Now our sequence $S=\seq{\seq{S_{j}^{n}}_{j<n}}_{n>0}$ given by
$\Sigma_{1}^{1}$-$\mathsf{AC}_{0}$ consists of double rays and $G$ is the
corresponding graph. Our construction of the desired $\left\langle R_{i}%
^{n}\right\rangle _{i<n}$ and $\left\langle P_{i}^{n}\right\rangle _{i<n}$
again proceeds recursively in $(G\oplus S)^{\prime}$ and at the end uses one
more jump to get the required sequence of double rays. We begin our recursion
by setting $R_{0}^{1}=S_{0}^{1}$ and $P_{0}^{1}\ $as any subpath of length
$2$. Here we note that by construction every $R_{i}^{n}$ will be of the form
$QaPbR$ where $P$ is a finite path in $G$ while $Q$ and $R$ are each $S_f$ or $S_{b}$ for $S$ some $S_{j}^{n}$. Again, finding each $\seq{R_{i}^{n+1}}_{i<n}$ and $\seq{P_{i}^{n+1}}_{i<n}$ in turn is recursive in $(G\oplus S)^{\prime}$ once we prove in
$\mathsf{ACA}_{0}$ that it exists given $\seq{R_{i}^{n}}_{i<n}$ and $\seq{P_{i}^{n}}_{i<n}$.

Thus it suffices to prove an analog of Lemma \ref{lem:more_disjoint_rays} in
$\mathsf{ACA}_{0}$: Given $\seq{R_{i}^{n}}_{i<n}$ and
$\left\langle P_{i}^{n}\right\rangle _{i<n}$ and $\seq{S_{i}^{k}}_{i<k}$ for $k=11n^{2}+1$ we can construct $\seq{R_{i}^{n+1}}_{i<n+1}$ and $\seq{P_{i}^{n+1}}_{i<n+1}$ as required. To simplify the notation we omit the superscripts $n$
and $k$. For $i<n$, let $x_{i,b}$ be the first vertex in $R_{i}$ before
$P_{i}$ and $x_{i,f}$ be the first vertex in $R_{i}$ after $P_{i}$. We will
arrange for $P_{i}^{n+1}$ to be $P_{i}$ preceded by $x_{i,b}$ and followed by
$x_{i,f}$ which will then also be $x_{i,b}R_{i}^{n+1}x_{i,f}$.

We first discard from $\{S_{j}\mid j<11n^{2}+1\}$ any $S_{j}$ sharing a vertex
with any $P_{i}$. As there are $n(2n+1)\leq3n^{2}$ such vertices we have at
least $8n^{2}+1$ many $S_{j}$ remaining. We relabel these as $S_{j}$ for
$j<8n^{2}+1$ and choose an edge $(c_{j,b},c_{j,f})$ in each $S_{j}$. We now
essentially follow the proof of Lemma \ref{lem:more_disjoint_rays} but for the
sets of single rays $\mathcal{R}=\{x_{i,f}R_{i,f}, x_{i,b}R_{i,b}\mid i<n\}$
and $\mathcal{S}=\{c_{j,f}S_{j,f}, c_{j,b}S_{j,b}\mid j<8n^{2}+1\}$. Our goal
for each $x_{i,d}R_{i,d}\in\mathcal{R}$, $i<n$ and $d\in\{b,f\}$, is to find a
suitable replacement (beginning with the same vertex while maintaining the
required disjointness) so that we can assemble the $R_{i}^{n+1}$ from them for
$i<n$. We also want an $S_{h}$ disjoint from all the double rays $R^{n+1}_i$ with $i < n$ which will be $R_{n}^{n+1}$.

We follow the procedure described in the proof of Lemma \ref{lem:more_disjoint_rays} (using
$2n$ for $n$ as $\mathcal{R}$ has $2n$ elements) to define single ray
replacements $R_{\left\langle i,d\right\rangle }^{\prime}$ for the
$x_{i,d}R_{i,d}$ from which we will construct the $R_{i}^{n+1}$. When in that
procedure we would keep the old ray we do so here as well and let
$R_{\left\langle i,f\right\rangle }^{\prime}=$ $x_{i,f}R_{i,f}$. When the
procedure is completed we let $I=\{\left\langle i,d\right\rangle \mid i<n$,
$d\in\{b,f\}$ and we have not defined $R_{\left\langle i,d\right\rangle
}^{\prime}\}$ and let $m=|I|$. Clearly we have discarded at most $4n^{2}$ many
more of the rays in $\mathcal{S}$.

Now continue the construction for the $\left\langle i,d\right\rangle \in I$
and the first vertices $z_{i,d}$ on $x_{i,d}R_{i,d}$ such that $x_{i,d}%
R_{i,d}z_{i,d}$ meets exactly $m$ many rays in $\mathcal{S}$. The union $F$ of
these finite paths meets at most $m\cdot2n\leq4n^{2}$ many rays in
$\mathcal{S}$ and so at most $4n^{2}$ many double rays $S_{j}$, $j<8n^{2}+1$.
Thus there must be at least one $S_{j}$ such that it does not meet $F$ and no
$c_{j,d}S_{j,d}$ was discarded during the process. We let one such be
$R_{n}^{n+1}$ and take any subpath of length $2n+2$ as $P_{n}^{n+1}$.

Now discard all single rays $S$ in $\mathcal{S}$ not meeting $F$. For each
$c_{j,d}S_{j,d}$ remaining in $\mathcal{S}$ let $y_{j,d}$ be the first vertex
on $c_{j,d}S_{j,d}$ such that $y_{j,d}S_{j,d}$ is disjoint from $F$. As in the
proof of Lemma \ref{lem:more_disjoint_rays}, we apply Menger's theorem to
$X,Y$ and $H$ where%

\begin{align*}
X  &  =\{x_{i,d} \mid \left\langle i,d\right\rangle \in I\}\\
Y  &  =\{y_{j,d} \mid c_{j,d}S_{j,d}\in\mathcal{S}\}\\
H  &  = F\cup\bigcup_{c_{j,d}S_{j,d}\in\mathcal{S}%
}c_{j,d}S_{j,d}y_{j,d}.
\end{align*}

This produces a set of $m$ disjoint paths $P_{i,d}$ in $H$ from each $x_{i,d}
$, $\left\langle i,d\right\rangle \in I$, to a $y_{j,d^{\prime}}$ in
$c_{j,d^{\prime}}S_{j,d^{\prime}}\in\mathcal{S}$. We can now define
$R_{\left\langle i,d\right\rangle }^{\prime}$ for $\left\langle
i,d\right\rangle \in I$ as the single ray beginning with $P_{i,d}$ and then
continuing with $S_{j,d^{\prime}}$ after $y_{j,d^{\prime}}$.

We can now define the required double rays $R_{i}^{n+1}$ for $i<n$ as $^{\ast
}\!R_{\left\langle i,b\right\rangle }^{\prime}P_{i}^{n}R_{\left\langle
i,f\right\rangle }^{\prime}$ and check that these have all the desired properties.
\end{proof}

As for $\mathsf{IRT}_{\mathrm{DES}}$, instead of following the proof of
Theorem \ref{thm:Sigma11-AC_implies_IRT}, we will reduce $\mathsf{IRT}%
_{\mathrm{DES}}$ to the problem of finding an infinite sequence of
\emph{vertex-disjoint} rays in a certain locally finite graph (see
\cite[pg.\ 2 l.\ 3--7]{bcp15}). To carry out this reduction, we define the
line graph:

\begin{defn}
The \emph{line graph} $L(G)$ of an $X$-graph $G$ is the $X$-graph whose vertices
are the edges of $G$ and whose edges are the $((u,v),(v,w))$, where $(u,v)$
and $(v,w)$ are edges in $G$.
\end{defn}

\begin{lem}
\label{lem:E_ray_in_G_to_V_ray_in_L(G)} Let $G$ be an X-graph. There is a
computable mapping from rays in $G$ to rays in $L(G)$ such that if two rays in
$G$ are edge-disjoint, then their images are vertex-disjoint.
\end{lem}

\begin{proof}
Map $x_{0},x_{1},x_{2},\dots$ to $(x_{0}, x_{1}), (x_{1}, x_{2}),\dots$.
\end{proof}

However, vertex-disjoint rays in $L(G)$ do not always yield edge-disjoint rays
in $G$. An extreme counterexample is what is called the \emph{(undirected)
star graph} which consists of a single vertex with infinitely many neighbors:
It does not contain any rays yet its line graph is isomorphic to the complete graph on $N$ which contains infinitely many vertex-disjoint rays.
Nonetheless, if $G$ is locally finite, then vertex-disjoint rays in $L(G)$ do
correspond to edge-disjoint rays in $G$:

\begin{lem}
[$\mathsf{ACA}_{0}$]\label{lem:locally_finite_V_ray_in_L(G)_to_E_ray_in_G} Let
$G$ be a locally finite X-graph. There is a mapping from rays in $L(G)$ to
rays in $G$ such that if two rays in $L(G)$ are vertex-disjoint, then their
images are edge-disjoint rays in $G$.
\end{lem}

\begin{proof}
Given a ray $R=e_{0},e_{1},\dots$ in $L(G)$, we construct a ray $S=y_{0}%
,y_{1},\dots$ in $G$ by recursion. Say that $e_{i}=(u_{i},v_{i})$. Start by
defining $y_{0}$ to be $u_{0}$. Having defined $y_{n}$, we define $y_{n+1}$ as
follows. Let $k_{n}$ be the largest $k$ such that $y_{n}$ is an endpoint of
$e_{k}$. Such $k$ exists because $G$ is locally finite and $R$ is a ray. We
can find $k_{n}$ by $\mathsf{ACA}_{0}$. Then define $y_{n+1}$ to be the
endpoint of $e_{k_{n}}$ other than $y_{n}$. This completes the recursion. Note
that the $k_{n}$ are strictly increasing: $e_{k_{n}}=(y_{n},y_{n+1})$ as
$y_{n}$ is not an endpoint of $e_{k_{n}+1}$ by the maximality of $k_{n}$. The
next edge then includes $y_{n+1}$ and so $k_{n+1}\geq k_{n}+1$. Also note that
if the graph is directed the $e_{i}$ must be of the form $\left\langle
x_{i},x_{i+1}\right\rangle $ and the last occurrence of any $x$ in an $e_{i}$
must be as its first element. By construction $S$ is infinite and contains no
repeated vertices because of the maximality requirement, hence it is a ray.
Observe that every edge in $S$ is a vertex in $R$, so the above mapping maps
vertex-disjoint rays in $L(G)$ to edge-disjoint rays in $G$.
\end{proof}

It remains to show that we can restrict our attention to locally finite
graphs. We accomplish this with the help of $\Sigma_{1}^{1}$-$\mathsf{AC}_{0}%
$. Given a directed graph $G$ with arbitrarily many edge-disjoint rays, we can
use $\Sigma_{1}^{1}$-$\mathsf{AC}_{0}$ to choose a family $\seq{\seq{R_{j}^{k}}_{j<k}}_{k>0}$ where, for each $k>0$, the rays $\seq{R_{j}^{k}}_{j<k}$ are edge-disjoint. From this family, we may construct an appropriate locally
finite subgraph of $G$:

\begin{lem}
\label{lem:locally_finite_subgraph} Suppose that $G$ is an X-graph and there
is some family $\seq{\seq{R_{j}^{k}}_{j<k}}_{k>0}$ of rays in $G$ such that for each $k>0$, the rays $\seq{R_{j}^{k}}_{j<k}$ are edge-disjoint. Then there is some locally
finite $X$-subgraph $G^{\prime}$ of $G$ and some family $\seq{\seq{S_{j}^{k}}_{j<k}}_{k>0}$ of rays in $G^{\prime}$ such that for each $k>0$, the rays $\seq{S_{j}^{k}}_{j<k}$ are edge-disjoint.
\end{lem}

\begin{proof}
Define the vertices of $G^{\prime}$ to be the vertices of $G$, say
$\{v_{i} \mid i\in N\}$. We specify the set of edges $E^{\prime}$ of $G^{\prime}$
by providing a recursive construction of sets $E_{i}$ of edges putting in a
set of edges at each step. We guarantee that each $E_{i}$ is a union of
finitely many finite sets of edge-disjoint rays in $G$ and that after stage
$k$ no edge with a vertex $v_{i}$ for $i<k$ as an endpoint is ever put into
$E^{\prime}$.

Begin at stage $0$ by putting all the edges in $R_{0}^{1}$ into $E_{1}$.
Proceeding recursively at stage $k>0$ we have $E_{k}$ and consider the
edge-disjoint rays $R_{j}^{k}$, $j<k$. For each $j<k$, say $R_{j}^{k}%
=x_{j,0}^{k},x_{j,1}^{k},\dots$. Each $v_{i}$ for $i<k$ appears at most once
in $R_{j}^{k}$ as $R_{j}^{k}$ is a ray. For each $j<k$, since we have access
to the set of vertices of $R_{j}^{k}$, we can decide whether $R_{j}^{k}$
contains $v_{i}$ and, if so, find the index $n$ such that $v_{i}=x_{j,n}^{k}$.
Call it $n_{i,j}^{k}$. If there is no such $n$, set $n_{i,j}^{k}=0$. Define
$S_{j}^{k}$ to be the tail of $R_{j}^{k}$ after $x_{j,\max_{i<k}n_{i,j}^{k}%
}^{k}$. We put all the edges in $S_{j}^{k}$, $j<k$, into $E_{k+1}$. Let
$E^{\prime}=\bigcup_{k}E_{k}$.

It remains to show that $G^{\prime}$ is locally finite. Consider any vertex
$v_{k}$. No edge containing $v_{k}$ as an endpoint is put in after stage $k$.
On the other hand, $E_{k}$ is the union of finitely many finite sets of
edge-disjoint rays (all of which have been computed uniformly). Each set of
edge-disjoint rays in this union has $v_{k}$ appearing at most once in each of
its rays. Thus at most two edges containing $v_{k}$ appear in each of the
finitely many rays in this set. Therefore there are only finitely many edges
containing $v_{k}$ in each of the finite sets of edge-disjoint rays making up
$E_{k}$. All in all, only finitely many edges in $G^{\prime}$ contain $v_{k}$.
\end{proof}

We are ready to prove

\begin{thm}
\label{thm:Sigma11-AC_implies_XES} $\Sigma^{1}_{1}$-$\mathsf{AC}_{0}$ implies
$\mathsf{IRT}_{\mathrm{XES}}$ for each value of X.
\end{thm}

\begin{proof}
Given an $X$-graph $G$ with arbitrarily many edge-disjoint rays, we can use
$\Sigma_{1}^{1}$-$\mathsf{AC}_{0}$ to choose a family $\seq{\seq{Q_{j}^{k}%
}_{j<k}}_{k>0}$ such that for each $k$, the rays $Q_{j}^{k}$, $j<k $, are
edge-disjoint. By Lemma \ref{lem:locally_finite_subgraph}, there is a locally
finite subgraph $H$ of $G$ and a family $\seq{\seq{R_{j}^{k}}_{j<k}}_{k>0}$ such that
for each $k$, the $R_{j}^{k}$, $j<k$, are edge-disjoint rays in $H$. By Lemma
\ref{lem:E_ray_in_G_to_V_ray_in_L(G)}, there is a family $\seq{\seq{S_{j}^{k}}_{j<k}}_{k>0}$ such that for each $k>0$, the $S_{j}^{k}$, $j<k$, are
vertex-disjoint rays in $L(H)$. By the second part of the proof of Theorem
\ref{thm:Sigma11-AC_implies_IRT}(ii) (which can be carried out in $\mathsf{ACA}%
_{0}$), $L(H)$ has infinitely many vertex-disjoint rays. Finally by Lemma
\ref{lem:locally_finite_V_ray_in_L(G)_to_E_ray_in_G}, $H$ has infinitely many
edge-disjoint rays. Hence $G$ has infinitely many edge-disjoint rays.
\end{proof}

By the discussion before Theorem \ref{thm:Sigma11-AC_implies_UVD}, this
completes the proof of Theorem \ref{thm:variants_IRT_hyp_analysis}.

Finally, we give a proof of $\mathsf{IRT}_{\mathrm{DED}}$ for directed forests
using $\Sigma_{1}^{1}\text{-}\mathsf{AC}_{0}$ (recall that $\mathsf{IRT}%
_{\mathrm{DED}}$ remains open). We will see that $\Sigma_{1}^{1}%
\text{-}\mathsf{AC}_{0}$ and $\mathsf{IRT}_{\mathrm{DED}}$ for directed
forests are equivalent over $\mathsf{RCA}_{0}+\mathsf{I}\Sigma_{1}^{1}$ but note that
$\Sigma_{1}^{1}$-$\mathsf{AC}_{0}$ does not imply $\mathsf{I}\Sigma_{1}^{1}$. (See
Theorem \ref{thm:DED_DVD_directed_forests_equivalence} and the comment
following it).

\begin{thm}
\label{thm:SigmaAC_implies_DED_directed_forests} $\Sigma_{1}^{1}%
\text{-}\mathsf{AC}_{0}$ implies $\mathsf{IRT}_{\mathrm{DED}}$ for directed
forests.
\end{thm}

We first prove two lemmas.

\begin{lem}[$\ACA_0$]
\label{lem:DED_directed_forests_int} Let $G$ be a directed forest and let
$R_{0}=\left\langle x_{0,i}\mid i\in\mathcal{Z}\right\rangle$, $R_{1}=\left\langle x_{1,i}\mid i\in\mathcal{Z}\right\rangle $ be directed double
rays in $G$. Suppose $R_{0}$ and $R_{1}$ have an edge $\left\langle
u,v\right\rangle $ in common. Then there are vertices $t$ and $w$ such $tR_{0}%
w=tR_{1}w$ and $R_{0}$, $R_{1}$ have no vertices in common outside of those in
$tR_{0}w=tR_{1}w$. Note that we allow for the possibility that $t=-\infty$
and/or $w=+\infty$ in the sense that $(-\infty)R=R=R(+\infty)$ for any double ray
$R$. We call $tR_{0}w=tR_{1}w$ the intersection of $R_{0}$ and $R_{1}$.
\end{lem}

\begin{proof}
Suppose $R_{0}$, $R_{1}$ provide a counterexample. As they have an edge
$\left\langle u,v\right\rangle $ in common, they lie in the same directed tree
$T$ in $G$ and can be viewed as (undirected) double rays in $\hat{T}$ (the
underlying graph for $T$). As they form a counterexample to the lemma, there
must be either a first $t\in R_{0}$ (i.e.\ earliest in the double ray $R_{0}$)
such that $tR_{0}v=tR_{1}v$ or a last $w\in R_{0}$ such that $uR_{0}w=uR_{1}w$
but $R_{0}$ and $R_{1}$ have a vertex $z$ in common outside the common
interval. The situations are symmetric and we consider the second. The
immediate successors $x$ and $y$ of $w$ in $R_{0}$ and $R_{1}$, respectively,
must be different by our choice of $w$. Consider now the location of $z$ in
$R_{0}$. If it is after $x$ then the paths from $w$ to $z$ in $R_{0}$ and $w$
to $z$ in $R_{1}$ (both considered now as undirected graphs within $\hat{T}$)
are different as the immediate successor of $w$ in $R_{0}$ is $x$ while in
$R_{1}$ it is either $y$ or a vertex in $v R_1 w$. Thus there are two different paths in $\hat{T}$ from $w$ to $z$
contradicting $\hat{T}$'s being a tree. If, on the other hand $z$ is before
$u$ in $R_{0}$, it must be before $t\in R_{0}$ and a similar argument provides
different paths from $t$ to $z$ in $R_{0}$ and $R_{1}$.
\end{proof}

\begin{lem}
[$\mathsf{ACA}_{0}$]\label{lem:DED_directed_forests_ext} \label{dedext} There
is a computable function $f$ such that given any sequence $\left\langle
S_{i}\right\rangle _{i<n}$ of DED rays with subpaths $\left\langle
P_{i}\right\rangle _{i<n}$ of length $2n$ in a directed tree $T$ and sequence
$\left\langle R_{j}\right\rangle _{j<f(n)}$ of DED rays in $T$, we can
construct a sequence $\left\langle S_{i}^{\prime}\right\rangle _{i\leq n}$ of
DED rays with subpaths $\left\langle P_{i}^{\prime}\right\rangle _{i\leq n}$
of length $2n+2$ in $T$ such that $P_{i}^{\prime}$ extends $P_{i}$ at each end
for $i<n$. Indeed, we may take $f(n)=2n^{2}+2^{2n}n!+1$.
\end{lem}

\begin{proof}
First we remove all the $R_{j}$ that contain an edge in any $P_{i}\ $at cost
of at most $2n^{2}$ many $j$. Consider any remaining $R_{j}$ in the second
sequence. By Lemma \ref{lem:DED_directed_forests_int}, its intersections with
the $S_{i}$ are intervals $Q_{j,i}$ of edges in $R_{j}$ which are disjoint as
the $S_{i}$ are. By our first thinning of the $R_{j}$ list, none of the
$Q_{j,i}$ intersect any of the $P_{i}$ so each $Q_{j,i}$ must lie entirely
above or entirely below$\;P_{i}$. We associate to each $R_{j}$ a label
consisting of the set $C_{j}=\{i<n\mid Q_{j,i}\neq\emptyset\}$; the elements
$i$ of $C_{j}$ in the order in which the $Q_{j,i}$ (for $i\in C_{j}$) appear
in $R_{j}$ (in terms of the ordering of $\mathcal{Z}$) along with a $+$ or $-$
depending on which side of $P_{i}$ it falls in $S_{i}$. We write $Q_{j,i}^{s}$
for the starting vertex of $Q_{j,i}$ and $Q_{j,i}^{e}$ for the ending one. As
above, we allow the values $\pm\infty$ for these endpoints if the intervals
are infinite. Now there are, of course, at most finitely many such labels. In
particular, there are at most $2^{n}n!2^{n}$ such labels. Thus if we have
$2^{2n}n!+1$ many $R_{j}$ left at least two of them, say $R_{a}$ and $R_{e}$
have the same label say with set $C$.\smallskip\newline\textbf{Claim:}
$|C|<2$.\smallskip\newline For the sake of a contradiction, assume we have
$k\neq l$ in $C$ with $k$ preceding $l$ in the ordering of $C$ in the label.
Say $R_{a}$ is the ray such that $Q_{a,k}$ is before $Q_{e,k}$ in $S_{k}$. Note that $Q_{a,k}$ and $Q_{e,k}$ are edge disjoint as $R_a$ and $R_e$ are. We consider two cases: (1) $Q_{a,l}$ is before $Q_{e,l}$ in $S_{l}$ and (2)
$Q_{e,l}$ is before $Q_{a,l}$ in $S_{l}$. We now produce, for each case, two
vertices with two distinct sequences (i) and (ii) of adjacent edges in $T$
connecting these two vertices. These sequences are illustrated in Figures
\ref{fig:DED_directed_forests_case_1} and
\ref{fig:DED_directed_forests_case_2}. Note that by our assumptions on the
orderings of the intervals $Q_{c,d}$ (for $c\in\{a,e\}$ and $d\in\{k,l\}$) as
displayed, all of the starting or ending points of $Q_{c,d}$ that appear in
our sequences are vertices in one of the rays (i.e.\ none are $\pm\infty$):

\setlength\itemsep{5pt}

\begin{itemize}
\item[(1i)] Start at $Q_{a,k}^{e}$ in $R_{a}$ and go to $Q_{a,l}^{e}$ then in
$S_{l}$ go to $Q_{e,l}^{s}$.

\item[(1ii)] Start at $Q_{a,k}^{e}$ in $S_{k}$ and go to $Q_{e,k}^{s}$ then go
in $R_{e}$ to $Q_{e,l}^{s}$.

\item[(2i)] Start at $Q_{a,k}^{s}$ in $S_{k}$ and go to $Q_{e,k}^{e}$ then in
$R_{e}$ go to $Q_{e,l}^{e}$ then in $S_{l}$ go to $Q_{a,l}^{e}$.

\item[(2ii)] Start at $Q_{a,k}^{s}$ in $R_{a}$ and go to $Q_{a,l}^{e}$.
\end{itemize}

\begin{figure}
\centering
\begin{tikzpicture}
	\node (Sk) at (0,0) {$S_k$};
	\node (Sk-) at (0,0.65) {$\vdots$};
	\node[label=left:{$Q_{a,k}$}] (Sk_Qak) [draw, shape=rectangle, minimum width=10,minimum height=25] at (0,1.9) {};
	\node[label=left:{$Q_{e,k}$}] (Sk_Qbk) [draw, shape=rectangle, minimum width=10, minimum height=25] at (0,3.6) {};
	\node (Sk+) at (0,5) {$\vdots$};
	\filldraw (Sk_Qak.north) node [right=5] {\tiny $Q^e_{a,k}$} circle (1.5pt);
	\draw [->] (0,0.85) -- (Sk+);
	\draw [line width=1.5,->] (Sk_Qak.north) -- (Sk_Qbk.south);
	\node (Sl) at (-2.5,0) {$S_l$};
	\node (Sl-) at (-2.5,0.65) {$\vdots$};
	\node[label=left:{$Q_{a,l}$}] (Sl_Qal) [draw, shape=rectangle, minimum width=10,minimum height=25] at (-2.5,1.9) {};
	\node[label=left:{$Q_{e,l}$}] (Sl_Qbl) [draw, shape=rectangle, minimum width=10, minimum height=25] at (-2.5,3.6) {};
	\node (Sl+) at (-2.5,5) {$\vdots$};
	\filldraw (Sl_Qbl.south) node [right=5] {\tiny $Q^s_{e,l}$} circle (1.5pt);
	\draw [->] (-2.5,0.85) -- (Sl+);
	\draw [line width=1.5,->] (Sl_Qal.north) -- (Sl_Qbl.south);
	\node (Ra) at (-5,0) {$R_a$};
	\node (Ra-) at (-5,0.65) {$\vdots$};
	\node[label=left:{$Q_{a,k}$}] (Ra_Qak) [draw, shape=rectangle, minimum width=10,minimum height=25] at (-5,1.9) {};
	\node[label=left:{$Q_{a,l}$}] (Ra_Qal) [draw, shape=rectangle, minimum width=10, minimum height=25] at (-5,3.6) {};
	\node (Ra+) at (-5,5) {$\vdots$};
	\filldraw (Ra_Qak.north) node [right=5] {\tiny $Q^e_{a,k}$} circle (1.5pt);
	\draw [->] (-5,0.85) -- (Ra+);
	\draw [line width=1.5,->] (Ra_Qak.north) -- (Ra_Qal.north);
	\draw [dashed] (Ra_Qal.north) -- (Sl_Qal.north);
	\node (Rb) at (2.5,0) {$R_e$};
	\node (Rb-) at (2.5,0.65) {$\vdots$};
	\node[label=right:{$Q_{e,k}$}] (Rb_Qbk) [draw, shape=rectangle, minimum width=10,minimum height=25] at (2.5,1.9) {};
	\node[label=left:{$Q_{e,l}$}] (Rb_Qbl) [draw, shape=rectangle, minimum width=10, minimum height=25] at (2.5,3.6) {};
	\node (Rb+) at (2.5,5) {$\vdots$};
	\filldraw (Rb_Qbl.south) node [right=5] {\tiny $Q^s_{e,l}$} circle (1.5pt);
	\draw [->] (2.5,0.85) -- (Rb+);
	\draw [line width=1.5,->] (Rb_Qbk.south) -- (Rb_Qbl.south);
	\draw [dashed] (Sk_Qbk.south) -- (Rb_Qbk.south);
\end{tikzpicture}
\caption{(1i) follows the thick arrows in $R_a$ and $S_l$. (1ii) follows the thick arrows in $S_k$ and $R_e$.}
\label{fig:DED_directed_forests_case_1}
\end{figure}
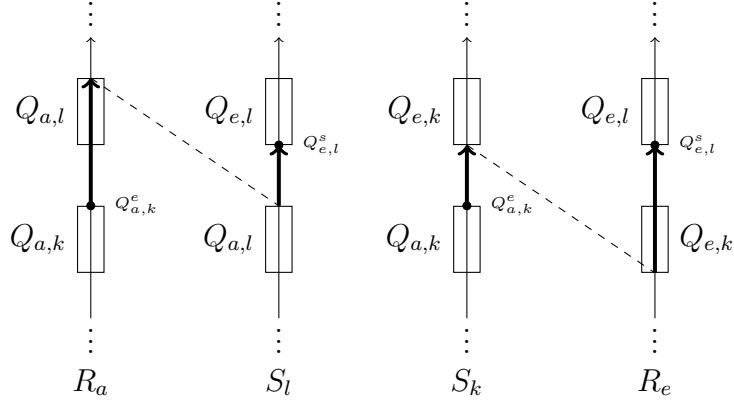

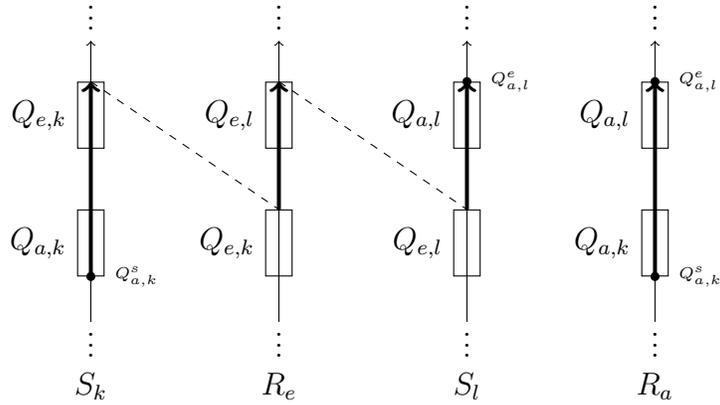
\begin{figure}
\centering
\begin{tikzpicture}
	\node (Sk) at (-5,0) {$S_k$};
	\node (Sk-) at (-5,0.65) {$\vdots$};
	\node[label=left:{$Q_{a,k}$}] (Sk_Qak) [draw, shape=rectangle, minimum width=10,minimum height=25] at (-5,1.9) {};
	\node[label=left:{$Q_{e,k}$}] (Sk_Qbk) [draw, shape=rectangle, minimum width=10, minimum height=25] at (-5,3.6) {};
	\node (Sk+) at (-5,5) {$\vdots$};
	\filldraw (Sk_Qak.south) node [right=5] {\tiny $Q^s_{a,k}$} circle (1.5pt);
	\draw [->] (-5,0.85) -- (Sk+);
	\draw [line width=1.5,->] (Sk_Qak.south) -- (Sk_Qbk.north);
	\node (Sl) at (0,0) {$S_l$};
	\node (Sl-) at (0,0.65) {$\vdots$};
	\node[label=left:{$Q_{e,l}$}] (Sl_Qbl) [draw, shape=rectangle, minimum width=10,minimum height=25] at (0,1.9) {};
	\node[label=left:{$Q_{a,l}$}] (Sl_Qal) [draw, shape=rectangle, minimum width=10, minimum height=25] at (0,3.6) {};
	\node (Sl+) at (0,5) {$\vdots$};
	\filldraw (Sl_Qal.north) node [right=5] {\tiny $Q^e_{a,l}$} circle (1.5pt);
	\draw [->] (0,0.85) -- (Sl+);
	\draw [line width=1.5,->] (Sl_Qbl.north) -- (Sl_Qal.north);
	\node (Ra) at (2.5,0) {$R_a$};
	\node (Ra-) at (2.5,0.65) {$\vdots$};
	\node[label=left:{$Q_{a,k}$}] (Ra_Qak) [draw, shape=rectangle, minimum width=10,minimum height=25] at (2.5,1.9) {};
	\node[label=left:{$Q_{a,l}$}] (Ra_Qal) [draw, shape=rectangle, minimum width=10, minimum height=25] at (2.5,3.6) {};
	\node (Ra+) at (2.5,5) {$\vdots$};
	\filldraw (Ra_Qak.south) node [right=5] {\tiny $Q^s_{a,k}$} circle (1.5pt);
	\filldraw (Ra_Qal.north) node [right=5] {\tiny $Q^e_{a,l}$} circle (1.5pt);
	\draw [->] (2.5,0.85) -- (Ra+);
	\draw [line width=1.5,->] (Ra_Qak.south) -- (Ra_Qal.north);
	\node (Rb) at (-2.5,0) {$R_e$};
	\node (Rb-) at (-2.5,0.65) {$\vdots$};
	\node[label=left:{$Q_{e,k}$}] (Rb_Qbk) [draw, shape=rectangle, minimum width=10,minimum height=25] at (-2.5,1.9) {};
	\node[label=left:{$Q_{e,l}$}] (Rb_Qbl) [draw, shape=rectangle, minimum width=10, minimum height=25] at (-2.5,3.6) {};
	\node (Rb+) at (-2.5,5) {$\vdots$};
	\draw [->] (-2.5,0.85) -- (Rb+);
	\draw [line width=1.5,->] (Rb_Qbk.north) -- (Rb_Qbl.north);
	\draw [dashed] (Sk_Qbk.north) -- (Rb_Qbk.north);
	\draw [dashed] (Rb_Qbl.north) -- (Sl_Qbl.north);
\end{tikzpicture}
\caption{(2i) follows the thick arrows in $S_k$, $R_e$ and $S_l$. (2ii) follows the thick arrow in $R_a$.}
\label{fig:DED_directed_forests_case_2}
\end{figure}

To see that the two sequences of vertices are different, note for (1) that
(1i) contains an edge in $Q_{a,l}$ but (1ii) does not. For (2) note that (2i)
contains an edge in $Q_{e,k}$ but (2ii) does not. We now, in each case, view
the associated two distinct sequences of vertices with the same endpoints in
the underlying (undirected) tree $\hat{T}$. The only way one can have such
sequences in a tree is for one of the sequences to contain some vertices $uvu$
in order. However, any three successive vertices in any of these sequences lie
within one of the $R_{j}$ or $S_{i}$ or both and so cannot have two instances
of the same vertex. The crucial point is that each $Q_{j,i}$ is in both
$R_{j}$ and $S_{i}$ and has at least two vertices. Any transition along the
sequence between an $R_{j}$ and an $S_{i}$ (in either order) goes through
$Q_{j,i}$ and so any three consecutive vertices are all contained in one
$R_{j}$ or one $S_{i}$ (or both).

Knowing now that $|C|$ is $0$ or $1$, we complete the proof of the Lemma. If
$|C|=0$, then both $R_{a}$ and $R_{e}$ are disjoint from all the $S_{i}$ and
so we may add on either one of them as $S_{n}^{\prime}$ with $P_{n}^{\prime}$
an arbitrary subpath of length $2n+2$ while keeping $S_{i}=S_{i}^{\prime}$ and
extending the $P_{i}^{\prime}$ appropriately for all $i<n$. Otherwise say
$C=\{i\}$. Let $c$ be the one of $a$ or $e$ such that $Q_{c,i}$ is closer to
$P_{i}$. (Remember that they are both on the same side of this interval in
$S_{i}$ by our fixing the label.) Now replace the tail of $S_{i}$ starting
with $Q_{c,i}$ and going away from $P_{i}$ by the tail of $R_{c}$ starting
with $Q_{c,i}$ and going in the same direction. Let this ray be $S_{i}%
^{\prime}$. Note that it is disjoint from all the $S_{j}$, $j\neq i$ as it
contains only edges that are in $S_{i}$ or $R_{c}$ neither of which share any
edges with such $S_{j}$. It is also disjoint from $R_{d}$ where $d$ is the one
of $a$, $e$ which is not $c$ since all the edges of $S_{i}^{\prime}$ are
either in $R_{c}$ or in $S_{i}$ outside of $Q_{d,i}$ by our choice of
$Q_{c,i}$ as closer to $P_{i}$. As $R_{d}$ is also disjoint from all the
$S_{j}$ for $j\neq i$ by our fixing the label, we may define $S_{j}^{^{\prime
}}=S_{j}$ and $P_{j}$ appropriately for $j<n$, $j\neq i$ and $S_{n}^{\prime
}=R_{d}$ and choosing $P_{n}^{\prime}$ of length $2n+2$ arbitrarily so as to
get the sequence required in the Lemma.
\end{proof}

Lemma \ref{lem:DED_directed_forests_ext} provides the inductive step for the
following proof:

\begin{proof}
[Proof of Theorem \ref{thm:SigmaAC_implies_DED_directed_forests}]Assume we are
given a directed forest $G$ with arbitrarily many DED rays. By $\Sigma_{1}%
^{1}\text{-}\mathsf{AC}_{0}$ we may take a sequence $\seq{\seq{R_{k,i}}_{i<k}}_{k \in N}$ such that, for each $k$, $\seq{R_{k,i}}_{i<k}$ is a sequence of $k$ many DED rays in $G$. If there are
infinitely many of the trees making up $G$ each of which contains some
$R_{k,i}$ then we are done. So we may assume that all of them are in one
directed tree $T$. We now wish to define $\seq{\seq{S_{k,s}}_{k<s}}_{s \in N}$ by recursion such that, for each $s$, $\seq{S_{k,s}}_{k<s}$ is a sequence of $s$ many DED rays with
subpaths $P_{k,s}$ of length $2s+1$ such that for each $s\in N$, $P_{k,s+1}$
extends $P_{k,s}\ $at each end so that the $\lim_{s}P_{k,s}$ form an infinite
sequence of DED rays in $T$ as required. Lemma
\ref{lem:DED_directed_forests_ext} provides precisely the required inductive
step for the construction since we have the required sequences of DED rays
$\left\langle R_{f(n),i}\mid i<f(n)\right\rangle $ at each step $n$ of the
construction. Once again we just have to note that Lemma
\ref{lem:DED_directed_forests_ext} provides witnesses for the double rays that
are composed of finite paths in $T$ and final segments in one direction or the
other of some of the $R_{k,i}$ and so they can be found recursively in $(T\oplus \seq{\seq{R_{k,i}}_{i<k}}_{k \in N})^{\prime}$.
\end{proof}

\section{Variations on Maximality}

\label{section:maximality}

In this section, we consider variants of $\mathsf{IRT}_{}$ whose solutions are
required to be maximal in terms of cardinality or maximal in terms of set inclusion.

\subsection{Maximum Cardinality Variants of $\mathsf{IRT}_{}$}

\label{section:max_cardinality}

\begin{defn}
\label{defn:IRTmc} Let $\mathsf{IRT}_{\mathrm{XYZ}}^{\ast}$ be the statement
that every $X$-graph $G$ has a set of $Y$-disjoint $Z$-rays of maximum cardinality,
or more formally, the statement that for every $X$-graph $G$:
\begin{itemize}
	\item there is no $Z$-ray in $G$, or
	\item there is some $n > 0$ and some $R$ such that $\seq{R^{[i]} \mid i<n}$ is a sequence of $Y$-disjoint $Z$-rays in $G$, and there is no $R$ such that $\seq{R^{[i]} \mid i<n+1}$ is a sequence of $Y$-disjoint $Z$-rays in $G$, or
	\item there is some $R$ such that $\seq{R^{[i]} \mid i \in N}$ is a sequence of $Y$-disjoint $Z$-rays in $G$.
\end{itemize}
When we talk about the cardinality of a
(possibly empty) finite sequence $\seq{R^{[i]}\mid i<n}$
we mean the number $n$ (which may be $0$). Of course a sequence $\seq{R^{[i]}\mid i\in N}$ is said to have infinite cardinality.
\end{defn}

$\mathsf{IRT}^{\ast}_{\mathrm{UVS}}$ was proved by Halin \cite{halin65}, who
also proved the corresponding statement for uncountable graphs.

\begin{rmk}
\label{rmk:ACA0_ast} The notation in Definition \ref{defn:IRTmc} is inspired
by the well known version $\mathsf{ACA}_{0}^{\ast}$ of $\mathsf{ACA}_{0}$
which, for every $A$, asserts the existence of $A^{(n)}$, the $n$th jump of $A$, for every $n
$:
\[
\mathsf{ACA}_{0}^{\ast}: \quad(\forall A)(\forall n)(\exists W)(W^{[0]}%
=A\wedge(\forall i<n)(W^{[i+1]}=W^{[i]^{\prime}})).
\]
This asserts (in addition to $\mathsf{ACA}_{0}$) particular instances of
$\mathsf{I}\Sigma_{1}^{1}$. So too (in addition to $\mathsf{IRT}_{\mathrm{XYZ}}$) do
the $\mathsf{IRT}^{\ast}_{\mathrm{XYZ}}$ as we are about to see.
\end{rmk}

\begin{prop}
\label{prop:IRT_IRTmc_rshp} For each choice of XYZ, $\mathsf{IRT}^{\ast
}_{\mathrm{XYZ}}$ implies $\mathsf{IRT}_{\mathrm{XYZ}}$ over $\RCA_0$ and $\mathsf{IRT}_{\mathrm{XYZ}}$ implies $\mathsf{IRT}^{\ast
}_{\mathrm{XYZ}}$ over $\RCA_0 + \mathsf{I}\Sigma^{1}_{1}$. Therefore
$\mathsf{IRT}_{\mathrm{XYZ}}$ and $\mathsf{IRT}^{\ast}_{\mathrm{XYZ}}$ are
equivalent over $\mathsf{RCA}_{0} + \mathsf{I}\Sigma^{1}_{1}$. In particular, they have
the same standard models.
\end{prop}

\begin{proof}
The first implication holds because if an $X$-graph has arbitrarily many
$Y$-disjoint $Z$-rays, then any sequence of $Y$-disjoint $Z$-rays in the graph of
maximum cardinality must be infinite. To prove the second implication, let $G$
be an $X$-graph and $\Phi(n)$ be the $\Sigma_{1}^{1}$ formula which says that there is a sequence of length $n$ of $Y$-disjoint $Z$-rays in $G$.
If $\forall n(\Phi(n)\rightarrow\Phi(n+1))$, then by $\mathsf{I}\Sigma_{1}^{1}$,
$\forall n\Phi(n)$ holds and so by $\mathsf{IRT}_{\mathrm{XYZ}}$ there is a
sequence $\seq{R_{i}}_{i\in N}$ of $Y$-disjoint $Z$-rays in
$G$ as required. On the other hand, if there is an $n$ such that $\Phi(n)$
holds but $\Phi(n+1)$ fails, then this $n$ witnesses $\mathsf{IRT}%
_{\mathrm{XYZ}}^{\ast}$.
\end{proof}

It follows from Proposition \ref{prop:IRT_IRTmc_rshp} and Theorem
\ref{thm:variants_IRT_hyp_analysis} that

\begin{cor}
\label{cor:variants_IRTmc_hyp_analysis} $\mathsf{IRT}_{\mathrm{XYS}}^{\ast}$
and $\mathsf{IRT}_{\mathrm{UVD}}^{\ast}$ are theorems of hyperarithmetic analysis.
\end{cor}

It follows from Proposition \ref{prop:IRT_IRTmc_rshp} and Theorem
\ref{thm:IRT_variants_ACA_0_and_omega_model_IRT_variants_hyp_closed} that
$\mathsf{IRT}^{\ast}_{\mathrm{XYZ}}$ implies $\mathsf{ACA}_{0}$, so we will
not explicitly mention uses of $\mathsf{ACA}_{0}$ whenever we are assuming any
$\mathsf{IRT}^{\ast}_{\mathrm{XYZ}}$.

Using Lemmas \ref{lem:DYZ_implies_UYZ}, \ref{lem:DEZ_implies_DVZ} and
\ref{lem:DYD_implies_DYS}, we can prove

\begin{prop}
\label{prop:IRTmc_implications_variants} $\mathsf{IRT}^{\ast}_{\mathrm{DYZ}}$
implies $\mathsf{IRT}^{\ast}_{\mathrm{UYZ}}$, $\mathsf{IRT}^{\ast
}_{\mathrm{DEZ}} $ implies $\mathsf{IRT}^{\ast}_{\mathrm{DVZ}}$ and
$\mathsf{IRT}^{\ast}_{\mathrm{DYD}}$ implies $\mathsf{IRT}^{\ast
}_{\mathrm{DYS}}$.
\end{prop}

Next, we show that $\mathsf{IRT}^{\ast}_{\mathrm{XYZ}}$ proves sufficient
induction in order to transcend $\Sigma^{1}_{1}$-$\mathsf{AC}_{0}$. This
implies that $\mathsf{IRT}^{\ast}_{\mathrm{XYZ}}$ is strictly stronger than
$\mathsf{IRT}_{\mathrm{XYZ}}$ for certain choices of $XYZ$ (Corollary
\ref{cor:IRTmc_stronger_IRT}). The connection between $\Sigma^{1}_{1}%
\text{-}\mathsf{AC}_{0}$ and graphs is obtained by viewing the set of
solutions of an arithmetic predicate as the set of (projections of) branches
on a subtree of $N^{<N}$. In detail:

\begin{lem}
[{\cite[V.5.4]{sim_book}}]\label{lem:arith_formula_paths_on_tree} If
$A(X)$ is an arithmetic formula, $\mathsf{ACA}_{0}$ proves that there is a
tree $T\subseteq N^{<N}$ such that
\begin{align*}
&  \forall X(A(X)\leftrightarrow\exists f(\langle X,f\rangle\in\lbrack T]))\\
\text{and}\; &  \forall X(\exists\text{ at most one }f)(\langle X,f\rangle
\in\lbrack T])\text{.}%
\end{align*}
\newline In fact, as the proof in \cite[V.5.4]{sim_book} shows, the required
functions $f$ are what are called the minimal Skolem functions and are
arithmetically defined uniformly in $X$ and the formula $A$.
\end{lem}

The following easy corollary will be useful.

\begin{lem}
\label{lem:arith_formula_paths_on_tree_seq} If $A(n,X)$ is an arithmetic
formula, $\mathsf{ACA}_{0}$ proves that there is a sequence of subtrees
$\seq{T_{n}}_{n}$ of $N^{<N}$ such that for each $n\in N$,
\begin{align*}
&  \forall X(A(n,X)\leftrightarrow\exists f(\langle X,f\rangle\in\lbrack
T_{n}]))\\
\text{and}\;  &  \forall X(\exists\text{ at most one }f)(\langle X,f\rangle
\in\lbrack T_{n}]).
\end{align*}

\end{lem}

\begin{proof}
Say that $B(Y)$ holds if and only if $A(Y(0),X)$ holds, where $X$ is such that
$Y=Y(0)\mathord{
\mathchoice
{\raisebox{1ex}{\scalebox{.7}{$\frown$}}}
{\raisebox{1ex}{\scalebox{.7}{$\frown$}}}
{\raisebox{.7ex}{\scalebox{.5}{$\frown$}}}
{\raisebox{.7ex}{\scalebox{.5}{$\frown$}}}
}X$. Apply Lemma \ref{lem:arith_formula_paths_on_tree} to the arithmetic
formula $B(Y)$ to obtain a tree $T\subseteq N^{<N}$. For each $n\in N$, define
$T_{n}$ to be the set of all $\sigma$ such that $n\mathord{
\mathchoice
{\raisebox{1ex}{\scalebox{.7}{$\frown$}}}
{\raisebox{1ex}{\scalebox{.7}{$\frown$}}}
{\raisebox{.7ex}{\scalebox{.5}{$\frown$}}}
{\raisebox{.7ex}{\scalebox{.5}{$\frown$}}}
}\sigma\in T$. It is straightforward to check that $\seq{T_{n}}_{n}$ satisfies the
desired properties.
\end{proof}

\begin{thm}
\label{thm:IRTmc_ACA0_ast} $\mathsf{IRT}_{\mathrm{XYZ}}^{\ast}$ proves
$\mathsf{ACA}_{0}^{\ast}$.
\end{thm}

Before proving the above theorem, we derive some corollaries:

\begin{cor}
\label{cor:IRTmc_not_provable_in_Sigma11-AC} $\mathsf{IRT}^{\ast
}_{\mathrm{XYZ}}$ proves the consistency of $\Sigma^{1}_{1}\text{-}%
\mathsf{AC}_{0}$. Therefore it is not provable in $\Sigma^{1}_{1}%
\text{-}\mathsf{AC}_{0}$.
\end{cor}

\begin{proof}
Simpson \cite[IX.4.6]{sim_book} proves that $\mathsf{ACA}_{0} + \mathsf{I}\Sigma
^{1}_{1}$ implies the consistency of $\Sigma^{1}_{1}\text{-}\mathsf{AC}_{0}$.
The only use of $\mathsf{I}\Sigma^{1}_{1}$ in Simpson's proof is to establish
$\mathsf{ACA}_{0}^{\ast}$, so Simpson's proof shows that $\mathsf{ACA}%
_{0}^{\ast}$ implies the consistency of $\Sigma^{1}_{1}\text{-}\mathsf{AC}%
_{0}$. The desired result then follows from Theorem \ref{thm:IRTmc_ACA0_ast}
and G\"odel's second incompleteness theorem.
\end{proof}

\begin{cor}
\label{cor:IRTmc_stronger_IRT} $\mathsf{IRT}^{\ast}_{\mathrm{XYZ}}$ is
strictly stronger than $\mathsf{IRT}_{\mathrm{XYZ}}$ for the following choices
of XYZ: XYS and UVD.
\end{cor}

\begin{proof}
We showed in \S \ref{section:variants_IRT} that the specified variants of
$\mathsf{IRT}_{{}}$ are provable in $\Sigma_{1}^{1}$-$\mathsf{AC}_{0}$. On the
other hand, none of the $\mathsf{IRT}_{{}}^{\ast}$ are provable in $\Sigma
_{1}^{1}$-$\mathsf{AC}_{0}$ (Corollary
\ref{cor:IRTmc_not_provable_in_Sigma11-AC}).
\end{proof}

We now prove Theorem \ref{thm:IRTmc_ACA0_ast}:

\begin{proof}
[Proof that $\mathsf{IRT}_{\mathrm{XYZ}}^{\ast}$ implies $\mathsf{ACA}%
_{0}^{\ast}$]By Proposition \ref{prop:IRTmc_implications_variants}, it
suffices to prove the desired result for $\mathsf{IRT}_{\mathrm{UYZ}}^{\ast}$.
To prove $\mathsf{ACA}_{0}^{\ast}$ from $\mathsf{IRT}_{\mathrm{UYS}}^{\ast}$,
begin by using Lemma \ref{lem:arith_formula_paths_on_tree_seq} to define for
each $A$ a sequence of trees $\seq{T_{n}}_{n}$ such that for each $n$ and $W$
there is at most one $f$ such that $\langle W,f\rangle\in\lbrack T_{n}]$ and%
\begin{align*}
&(\exists f)(\langle W,f\rangle \in\lbrack T_{n}]) \\
\leftrightarrow \; &W^{[0]}  =A \; \land \; (\forall i\leq n)((W^{[i]})^{\prime}=W^{[i+1]}%
)\; \land \; (\forall i>n)(W^{[i]}=\emptyset).
\end{align*}

We want to show that each $T_{n}$ is ill-founded. Note that if $m < n$ and
$T_{n} $ is ill-founded, then so is $T_{m}$. Therefore it suffices to show
that for cofinally many $n$, $T_{n}$ is ill-founded.

Apply $\mathsf{IRT}^{\ast}_{\mathrm{UYS}}$ to the disjoint union
$\bigsqcup_{n} T_{n}$ to obtain a collection $C$ of $Y$-disjoint rays of maximum
cardinality. We prove that $C$ is infinite. Suppose not. Then there is some
maximum $m$ such that $C$ contains a ray in $T_{m}$. A ray in $T_{m}$ can be
computably truncated or extended to a branch on $T_{m}$, so $T_{m}$ is
ill-founded. Hence $T_{m+1}$ is ill-founded as well (by $\mathsf{ACA}_{0}$).
But then there is a collection of $Y$-disjoint rays in $\bigsqcup_{n} T_{n}$
which has cardinality greater than that of $C$, contradiction.

We have proved that $C$ is infinite. Next we prove that each $T_{n}$ has at
most one branch. That would imply that each $T_{n}$ contains at most one ray
in $C$, so $C$ contains rays in cofinally many $T_{n}$, as desired.

If $T_{n}$ has two distinct branches $\langle W_{0},f_{0}\rangle$ and $\langle
W_{1},f_{1}\rangle$, then $W_{0}\neq W_{1}$ by the \textquotedblleft at most
one\textquotedblright\ condition in the definition of the $T_{n}$. Consider
the least $i$ such that $W_{0}^{[i]}\neq W_{1}^{[i]}$. Such $i$ exists by
$\mathsf{ACA}_{0}$. Note that $0<i\leq n$ because $W_{0}^{[0]}=A=W_{1}^{[0]}$
and $W_{0}^{[i]}=\emptyset=W_{1}^{[i]}$ for $i>n$. But then $W_{0}%
^{[i-1]}=W_{1}^{[i-1]}$ and $(W_{0}^{[i-1]})^{\prime}\neq(W_{1}^{[i-1]}%
)^{\prime}$, contradiction.

This proves that $\mathsf{IRT}_{\mathrm{UYS}}^{\ast}$ implies $\mathsf{ACA}%
_{0}^{\ast}$. In order to prove that $\mathsf{IRT}_{\mathrm{UYD}}^{\ast}$
implies $\mathsf{ACA}_{0}^{\ast}$, we modify the above proof by adding to each
$T_{n}$ a computable branch consisting of new vertices to form a tree $S_{n}$.
Apply $\mathsf{IRT}_{\mathrm{UYD}}^{\ast}$ to $\bigsqcup_{n}S_{n}$ to obtain a
collection $C$ of $Y$-disjoint double rays of maximum cardinality. Following the
above proof, we may prove that $C$ is infinite and each $S_{n}$ contains at
most one double ray in $C$. So $C$ contains double rays in cofinally many
$S_{n}$, as desired.
\end{proof}

\begin{rmk}
The same proof shows that $\mathsf{IRT}_{\mathrm{XYZ}}^{\ast}$ implies the
following induction scheme: Suppose $\seq{T_{n}}_{n}$ is a sequence of trees such that
\begin{enumerate}
	\item $T_{0}$ has a unique branch;
	\item for all $n$, the number of branches on $T_{n+1}$ is the same as the
number of branches on $T_{n}$.
\end{enumerate}
Then for all $n$, there is a sequence $\seq{P_{m}}_{m<n}$ such that for each
$m<n$, $P_{m}$ is the unique branch on $T_{m}$. It also shows that
$\mathsf{IRT}_{\mathrm{XYZ}}^{\ast}$ implies $\mathsf{ACA}_{0}^{+}$ (i.e.\ closure under the $\omega$-jump) and much more. Indeed, similar ideas prove in
Theorem \ref{thm:XYZmc_implies_finite_choice} that $\mathsf{IRT}%
_{\mathrm{XYZ}}^{\ast}$ implies unique-$\Sigma_{1}^{1}\text{-}\mathsf{AC}_{0}$
(Definition \ref{defn:unique_choice}). $\mathsf{IRT}_{\mathrm{XYZ}}^{\ast}$
also implies a similar induction scheme analogous to finite-$\Sigma_{1}%
^{1}\text{-}\mathsf{AC}_{0}$ (Definition \ref{defn:finite_choice}).
\end{rmk}

We can prove that even fragments of $\mathsf{IRT}_{\mathrm{DVD}}^{\ast}$ give
more induction than the specific instances derived in Theorem
\ref{thm:IRTmc_ACA0_ast}.

\begin{thm}
\label{thm:DVD-IRTmc_implies_ISigma11} $\mathsf{IRT}_{\mathrm{DVD}}^{\ast}$
(even for directed forests) implies $\mathsf{I}\Sigma_{1}^{1}$ over $\mathsf{RCA}_{0}$.
\end{thm}

\begin{proof}
Suppose $\Psi(n)$ is a $\Sigma_{1}^{1}$ formula such that $\Psi(0)$ and $\forall n(\Psi(n)\rightarrow\Psi(n+1))$ hold. Let $\seq{S_{i}}_{i\in N}$ be a sequence of subtrees of $N^{<N}$ such that $S_{i}$ is ill-founded if and only if $\Psi(i)$ holds. Let $\seq{T_{n}}_{n\in N}$ be a sequence of subtrees of $N^{<N}$ such that for each $n$, $T_n$ consists of all sequences $\seq{\sigma_i: i \leq n}$ where for each $i \leq n$, $\sigma_i$ is (a code for) a string in $S_i$. We order these sequences by component-wise extension. It is clear that $T_{n}$ is ill-founded if and only if $(\forall i \leq n)\Psi(i)$ holds.

For each $n$, orient each edge in $T_{n}$ towards its root and add a computable $D$-ray of new vertices which starts at its root. This forms a directed tree $G_{n}$. If $G_{n}$ contains a double ray, $T_{n}$ is ill-founded and $(\forall i \leq n)\Psi(i)$ holds. Furthermore, no two disjoint double rays can lie in the same $G_{n}$.

Let $G$ be the directed forest $\bigsqcup_{n}G_{n}$. By $\mathsf{IRT}_{\mathrm{DVD}}^{\ast}$, there is a sequence $\seq{R_{i}}_{i}$ of disjoint double rays in $G$ of maximum cardinality, so the sequence may be for $i<k$ for some $k$ or $i\in N$. Since $\Psi(0)$ holds, $\seq{R_{i}}_{i}$ is nonempty. If $\seq{R_{i}}_{i}$ is finite, let $n$ be maximal such that $G_{n}$ contains some $R_{i}$. Then $\Psi(n)$ holds, so $\Psi(n+1)$ holds as well. It follows that $G_{n+1}$ contains some double ray, which we can then add to $\seq{R_{i}}_{i}$ to obtain a larger sequence of disjoint double rays in $G$ for the desired contradiction. Therefore $\seq{R_{i}}_{i}$ is infinite. Since each $G_{n}$ contains at most one $R_{i}$, infinitely many $G_{n}$ contain some $R_{i}$. Therefore $\Psi(n)$ holds for all $n$.
\end{proof}

In fact, we have the following equivalences:

\begin{thm}
\label{thm:DED_DVD_directed_forests_equivalence} The following are equivalent
(over $\mathsf{RCA}_{0}$):

\begin{enumerate}
\item $\Sigma^{1}_{1}\text{-}\mathsf{AC}_{0} + \mathsf{I}\Sigma^{1}_{1}$;

\item $\mathsf{IRT}_{\mathrm{DED}}$ for directed forests$\;+\;\mathsf{I}\Sigma^{1}_{1}$;

\item $\mathsf{IRT}^{\ast}_{\mathrm{DED}}$ for directed forests;

\item $\mathsf{IRT}^{\ast}_{\mathrm{DVD}}$ for directed forests;

\item $\mathsf{IRT}_{\mathrm{DVD}}$ for directed forests$\;+\;\mathsf{I}\Sigma^{1}_{1}$.
\end{enumerate}
\end{thm}

\begin{proof}
(1) $\to$ (2) follows from Theorem
\ref{thm:SigmaAC_implies_DED_directed_forests}. (2) $\to$ (3) follows from the
proof of Proposition \ref{prop:IRT_IRTmc_rshp}. (3) $\to$ (4) follows from the
observation that the mapping of graphs defined in Lemma
\ref{lem:DEZ_implies_DVZ} sends a directed forest to a directed forest. (4)
$\to$ (5) follows from Theorem \ref{thm:DVD-IRTmc_implies_ISigma11} and the
proof of Proposition \ref{prop:IRT_IRTmc_rshp}.

To prove (5) $\rightarrow$ (1), suppose $A(n,X)$ is an arithmetic formula such
that $\forall n\exists XA(n,X)$. By Lemma
\ref{lem:arith_formula_paths_on_tree_seq}, there is a sequence $\seq{T_{n}}_{n}$
of subtrees of $N^{<N}$ such that
\begin{align*}
&  \forall n\forall X(A(n,X)\leftrightarrow\exists f(\langle X,f\rangle
\in\lbrack T_{n}]))\\
\text{and}\;  &  \forall X(\exists\text{ at most one }f)(\langle X,f\rangle
\in\lbrack T_{n}]).
\end{align*}
By assumption on $A(n,X)$, each $T_{n}$ is ill-founded. We use $\seq{T_{n}}_{n}$
to construct a sequence $\seq{G_{n}}_{n}$ of directed trees as we did in the proof
of Theorem \ref{thm:DVD-IRTmc_implies_ISigma11} to construct $T_{n}$ from the
$S_{i}$.

By $I\Sigma_{1}^{1}$, the directed forest $\bigsqcup_{n}G_{n}$ contains
arbitrarily many disjoint double rays. Therefore $\bigsqcup_{n}G_{n}$ contains
infinitely many disjoint double rays $\seq{R_{k}}_{k}$, by $\IRT{DVD}$. Note
that any double ray in any $G_{n}$ must contain the computable ray we added,
so any two double rays in the same $G_{n}$ must intersect. This implies that
each $R_{k}$ belongs to some distinct $G_{n}$. Therefore for every $m$, there
is some $k$ and some $n>m$ such that $R_{k}$ is a double ray in $G_{n}$. When we remove the added computable ray from $R_k$ we are left with a branch in $T_n$ which is of the form $\langle X,f \rangle$ where $X$ consists of witnesses $X_i$ for $i < n$.
\end{proof}

Since $\Sigma^{1}_{1}\text{-}\mathsf{AC}_{0}$ ($\mathsf{ATR}_{0}$, even) does
not prove $\mathsf{I}\Sigma^{1}_{1}$ \cite[IX.4.7]{sim_book}, it follows that

\begin{cor}
\label{cor:IRTmc_DVD_strictly_implies_SigmaAC} $\mathsf{IRT}^{\ast
}_{\mathrm{DYD}}$ (even for directed forests) is not provable in
$\mathsf{ATR}_{0}$, and strictly implies $\Sigma^{1}_{1}\text{-}%
\mathsf{AC}_{0}$ over $\mathsf{RCA}_{0}$.
\end{cor}

Next, we show that $\mathsf{IRT}^{\ast}_{\mathrm{UVD}}$ implies $\mathsf{IRT}%
_{\mathrm{UVS}}$ over $\mathsf{RCA}_{0}$ (see Figure \ref{fig:variants_IRT}).

\begin{thm}
\label{thm:UVDmc_implies_UVS} $\mathsf{IRT}^{\ast}_{\mathrm{UVD}}$ implies
$\mathsf{IRT}_{\mathrm{UVS}}$ over $\mathsf{RCA}_{0}$. Therefore (1)
$\mathsf{IRT}_{\mathrm{UVD}}$ implies $\mathsf{IRT}_{\mathrm{UVS}}$ over
$\mathsf{RCA}_{0} + \mathsf{I}\Sigma^{1}_{1}$; (2) if any standard model of
$\mathsf{RCA}_{0}$ satisfies $\mathsf{IRT}_{\mathrm{UVD}}$, then it satisfies
$\mathsf{IRT}_{\mathrm{UVS}}$ as well.
\end{thm}

\begin{proof}
Let $G$ be a graph which contains arbitrarily many disjoint single rays. By
$\mathsf{IRT}^{\ast}_{\mathrm{UVD}}$, there is a sequence of disjoint double
rays in $G$ of maximum cardinality. If this sequence is infinite, then there
are infinitely many disjoint single rays in $G$ as desired. Otherwise, suppose
that $\seq{R_{i}}_{i<j}$ is a sequence of disjoint double rays in $G$ of maximum
cardinality $j$. Let $\mathcal{R}$ be the subgraph of $G$ consisting of the
union of all $R_{i}$. Let $H$ be the induced subgraph of $G$ consisting of all
vertices which do not lie in $\mathcal{R}$. Note that $H$ does not contain any
double ray, otherwise $G$ would contain $j+1$ many disjoint double rays. Next,
we expand $H$ to the graph $H^{\prime}$, defined below.

Decompose $H$ into its connected components $\seq{H_{i}}_{i}$ (there may only be
finitely many). Any two single rays in the same $H_{i}$ must intersect,
because if $S_{0}$ and $S_{1}$ are disjoint single rays in the same $H_{i}$,
then we can construct a double ray in $H_{i}$ by connecting them (start with a
path between $S_{0}$ and $S_{1}$ of minimum length, then connect it to the
tails of $S_{0}$ and $S_{1}$ which begin at the endpoints of the path).

For each $i$, define $H_{i}^{\prime}$ by adding a computable ray of new
vertices to $H_{i}$, which begins at the $<_{N}$-least vertex in $H_{i}$.
Define $H^{\prime}$ to be the disjoint union $\bigsqcup_{i}H_{i}^{\prime}$.

By $\mathsf{IRT}^{\ast}_{\mathrm{UVD}}$, there is a sequence of disjoint
double rays in $H^{\prime}$ of maximum cardinality.

\underline{Case 1.} If this sequence is infinite, then $G$ contains infinitely
many disjoint single rays because each double ray in the sequence has a tail
which lies in $H$. In this case we are done.

\underline{Case 2.} Otherwise, $H^{\prime}$ does not contain arbitrarily many
disjoint double rays. Since any two single rays in the same $H_{i}$ must
intersect, we can transform any collection of disjoint single rays in $H$ into
a collection of disjoint double rays in $H^{\prime}$ of equal cardinality by
connecting each single ray to the $<_{N}$-least vertex in its connected
component $H_{i}$ and then following the computable ray we added. It follows
that $H$ does not contain arbitrarily many disjoint single rays. Fix $l$ such
that $H$ does not contain $l+1$ many disjoint single rays.

Towards a contradiction, we construct a collection of $(j+1)$-many disjoint
double rays in $G$ as follows. Fix a collection $\mathcal{S}$ of
$l+2j+4j(j+1)$ many disjoint single rays in $G$. First, at most $2j$ of these
single rays lie in $\mathcal{R}$. In fact at most $2j$ of these single rays
can have finite intersection with $H$, because given a collection of disjoint
single rays each of which has finite intersection with $H$, we can obtain a
collection of disjoint single rays in $\mathcal{R}$ of the same cardinality by
replacing each ray with an appropriate tail. Second, by reasoning analogous to
the above, at most $l$ of these single rays can have finite intersection with
$\mathcal{R}$. Therefore, there are at least $4j(j+1)$ many disjoint single
rays in $\mathcal{S}$ each of which have infinite intersection with both
$\mathcal{R}$ and $H$.

Next, choose an edge $(u_{i},v_{i})$ in each $R_{i}$ and split $R_{i}$ into
two single rays $u_{i}R_{i,b}$ and $v_{i}R_{i,f}$. By the pigeonhole
principle, there is some single ray $R$ of the form $u_{i}R_{i,b}$ or
$v_{i}R_{i,f}$, and at least $2(j+1)$ many disjoint single rays in
$\mathcal{S}$, each of which have infinite intersection with both $R$ and $H$.
Call these rays $S_{0},S_{1},\dots,S_{2(j+1)-1}$. Discard all the other rays
in $\mathcal{S}$. Below we describe how to connect pairs of single rays
$S_{k}$ using segments of $R$ to form a collection of $(j+1)$-many disjoint
double rays in $G$.

Let $x_{0},x_{1},\dots$ denote the vertices of $R$. Since each single ray
$S_{k}$ has infinite intersection with $R$, by the pigeonhole principle, there
is a pair of disjoint rays $S_{k_{0}}$ and $S_{l_{0}}$ such that for each tail
$R^{\prime}$ of $R$, there is a vertex in $S_{k_{0}}\cap R^{\prime}$ and a
vertex in $S_{l_{0}}\cap R^{\prime}$ such that no $S_{k}$ intersects $R$
between these two vertices. (Formally, we justify this by defining the
following coloring recursively. Start from the first vertex in $R$ which is
also in some $S_{k}$. Search for the next vertex on $R$ which intersects some
$S_{l}$, $l\neq k$. Then we color $0$ with the unordered pair $\{k,l\}$. Then
we search for the next vertex on $R$ which intersects some $S_{m}$, $m\neq l$
and color $1$ with $\{l,m\}$, and so on. Some color $\{k_{0},l_{0}\}$ must
appear infinitely often.) Then we commit to connecting $S_{k_{0}}$ and
$S_{l_{0}}$ (but we do not do so just yet). Applying the pigeonhole principle
again, there is a pair of disjoint rays $S_{k_{1}}$ and $S_{l_{1}}$ (with
$k_{1},l_{1},k_{0},l_{0}$ all distinct) such that for each tail $R^{\prime}$
of $R$, there is a vertex $x$ in $S_{k_{1}}\cap R^{\prime}$ and a first vertex
$y$ in $S_{l_{1}}\cap R^{\prime}$ (after $x$ in $R$) such that no $S_{k}$,
except perhaps $S_{k_{1}}$, $S_{k_{0}}$ or $S_{l_{0}}$, intersects $R$ between
these two vertices. We may eliminate any elements of $S_{k_{1}}$ by changing
$x$ (if necessary) to the last element of $Ry$ in $S_{k_{1}}$. Again we commit
to connecting $S_{k_{1}}$ and $S_{l_{1}}$. Repeat this process until we have
obtained $j+1$ pairs of single rays. That is, when we have $S_{k_{i}}$ and
$S_{l_{i}}$ for an $i<j$, we find $S_{k_{i+1}}$ and $S_{l_{i+1}}$ with
$k_{i+1}$ and $l_{i+1}$ distinct from all previous $k_{m}$ and $l_{m}$ such
that for each tail $R^{\prime}$ of $R$ there is a vertex $x\in S_{k_{i+1}}\cap
R$ and a first $y\in S_{l_{i+1}}\cap R$ after $x$ in $R$ such that no $S_{k}$,
except perhaps $S_{k_{m}}$ or $S_{l_{m}}$ for $m\leq i$, intersects $R$
between these two vertices. This process stops when we define $k_{j}$ and
$l_{j}$.

Finally, we connect these pairs of single rays in the opposite order in which
we defined them: Start by picking some $x^{j}\in S_{k_{j}}\cap R$ and some
$y^{j}\in S_{l_{j}}\cap R$. Then we define a double ray $D_{j}$ by following
$^{\ast}\!S_{k_{j}}$ until $x^{j}$, then following $R$ until $y^{j}$, and
finally following $S_{l_{j}}$, i.e., $D_{j}:=~^{\ast}\!(x^jS_{k_{j}}%
)Ry^{j}S_{l_{j}}$. Having defined $D_{j},D_{j-1},\dots,D_{i+1}$, define
$D_{i}:=~^{\ast}\!(x^{i}S_{k_{i}})Ry^{i}S_{l_{i}}$, where $x^{i}\in S_{k_{i}%
}\cap R$ and $y^{i}\in S_{l_{i}}\cap R$ are chosen as follows: Consider a tail
$R^{\prime}$ of $R$ such that the union of $x^{j}Ry^{j},\dots,x^{i+1}Ry^{i+1}$
is disjoint from (1) $R^{\prime}$; (2) $xS_{k_{i,}}$ for each $x\in S_{k_{i}%
}\cap R^{\prime}$; (3) $yS_{l_{i}}$ for each $y\in S_{l_{i}}\cap R^{\prime}$.
By choice of $k_{i}$ and $l_{i}$, there are vertices $x^{i}\in S_{k_{i}}\cap
R^{\prime}$ and $y^{i}\in S_{l_{i}}\cap R^{\prime}$ such that none of
$S_{k_{j}},\dots,S_{k_{i+1}}$ or $S_{l_{j}},\dots,S_{l_{i+1}}$ intersect
$x^{i}Ry^{i}$.

It is straightforward to check that each of $D_{j},D_{j-1},\dots,D_{i+1}$ is
disjoint from $D_{i}$. This process yields disjoint double rays $D_{j}%
,D_{j-1},\dots,D_{0}$ in $G$, contradicting the maximality of $j$.
\end{proof}

Using some of the ideas in the previous proof, we can prove

\begin{thm}
\label{thm:UYDmc_implies_UVSmc_forests} $\mathsf{IRT}_{\mathrm{UYD}}^{\ast}$
for forests implies $\mathsf{IRT}_{\mathrm{UYS}}^{\ast}$ for forests over
$\mathsf{RCA}_{0}$. Therefore $\mathsf{IRT}_{\mathrm{UYD}}$ for forests
implies $\mathsf{IRT}_{\mathrm{UYS}}$ for forests over $\mathsf{RCA}%
_{0}+\mathsf{I}\Sigma_{1}^{1}$.
\end{thm}

This result will be used in the proofs of Theorems
\ref{thm:XYZmc_implies_finite_choice}, \ref{thm:IRTmc_implies_ABW} and
\ref{thm:ABW_not_imply_IRT}.

\begin{proof}
Let $G$ be a forest. If $G$ happens to have arbitrarily many disjoint double
rays, then by $\mathsf{IRT}^{\ast}_{\mathrm{UYD}}$, $G$ has infinitely many
disjoint double rays. Therefore there is an infinite sequence of disjoint
single rays in $G$. Such a sequence has maximum cardinality, so we are done in
this case.

Suppose $G$ does not have arbitrarily many disjoint double rays. By
$\mathsf{IRT}^{\ast}_{\mathrm{UYD}}$ for forests, there is a sequence
$\seq{R_{i}}_{i<j}$ of disjoint double rays in $G$ of maximum cardinality.
Following the proof of Theorem \ref{thm:UVDmc_implies_UVS}, define the forests
$\mathcal{R}$, $H$, and $H^{\prime}$. There, we proved that no two single rays
in the same connected component $H_{i}$ of $H$ can be disjoint.

By $\mathsf{IRT}^{\ast}_{\mathrm{UYD}}$ for forests, there is a sequence of
disjoint double rays in $H^{\prime}$ of maximum cardinality. If this sequence
is infinite, then there is an infinite sequence of disjoint single rays in $H$
because each double ray in the sequence has a tail which lies in $H$. This is
a sequence of disjoint single rays of maximum cardinality in $G$, so we are
done in this case.

Otherwise, suppose $\seq{S_{k}}_{k<l}$ is a disjoint sequence of double rays in
$H^{\prime}$ of maximum cardinality. Consider the following disjoint sequence
of single rays in $G$. First, for each $k<l$, consider the single ray formed
by intersecting $H$ and the double ray $S_{k}$. Second, for each $i<j$, we can
split the double ray $R_{i}$ into a pair of disjoint single rays in $G$. This
yields a finite sequence $\seq{Q_{m}}_{m<n}$ of disjoint single rays in $G$.

We claim that $\seq{Q_{m}}_{m<n}$ is a sequence of disjoint single rays in $G$ of
maximum cardinality. Suppose there is a larger sequence of disjoint single
rays in $G$. Since $G$ is a forest, any two single rays in $G$ which share
infinitely many edges or vertices must share a tail. Therefore there is a
single ray $Q$ in this larger sequence which only shares finitely many edges
and vertices with each $Q_{m}$. Then some tail of $Q$, say $xQ$, is
vertex-disjoint from each $Q_{m}$. In particular, $xQ$ is vertex-disjoint from
each $R_{i}$, i.e.\ $xQ$ lies in $H$. Extend $xQ$ to a double ray in
$H^{\prime}$ by first connecting $x$ to the $<_{N}$-least vertex in its
connected component $H_{i}$, then following the computable ray which we added.
The resulting double ray is disjoint from every $S_{k}$, because no $S_{k}$
can lie in the same $H_{i}^{\prime}$ as $xQ$ (for $xQ$ is vertex-disjoint from
$S_{k}\cap H$ by construction). This contradicts the maximality of $l$.
\end{proof}

\subsection{Maximal Variants of $\mathsf{IRT}_{}$}

\label{section:maximal}

Instead of sets of disjoint rays of maximum cardinality, we could consider
sets of disjoint rays which are maximal with respect to set inclusion. For
uncountable graphs, Halin \cite{halin65} observed that any uncountable maximal
set of disjoint rays is in fact of maximum cardinality (because rays are
countable). This suggests another variant of $\mathsf{IRT}_{{}}$, which we
call \emph{maximal $\mathsf{IRT}_{{}}$}:

\begin{defn}
\label{defn:MIRT} Let $\mathsf{MIRT}_{\mathrm{XYZ}}$ be the statement that
every $X$-graph $G$ has a (possibly finite) sequence $(R_{i})_{i}$ of $Y$-disjoint
$Z$-rays which is maximal, i.e., for any $Z$-ray $R$ in $G$, there is some $i$
such that $R$ and $R_{i}$ are not $Y$-disjoint.
\end{defn}

$\mathsf{MIRT}_{\mathrm{XYZ}}$ immediately follows from Zorn's Lemma. It is
straightforward to show that $\mathsf{MIRT}_{\mathrm{XYZ}}$ implies $\Pi
_{1}^{1}$-$\mathsf{CA}_{0}$ (see the proof of Theorem
\ref{thm:MIRT_equiv_Pi11-CA} below), hence $\mathsf{MIRT}_{\mathrm{XYZ}}$ is
much stronger than $\mathsf{IRT}_{\mathrm{XYZ}}$ or even $\mathsf{IRT}%
_{\mathrm{XYZ}}^{\ast}$. We show below that $\mathsf{MIRT}_{\mathrm{XYZ}}$ is
equivalent to $\Pi_{1}^{1}$-$\mathsf{CA}_{0}$. This situation is reminiscent
of K\"{o}nig's duality theorem for countable graphs. Aharoni, Magidor, Shore
\cite{ams92} proved that the theorem implies $\mathsf{ATR}_{0}$ and that
$\Pi_{1}^{1}$-$\mathsf{CA}_{0}$ suffices to prove the required existence of a
K\"{o}nig cover. Simpson \cite{sim94} later proved that $\mathsf{ATR}_{0}$
actually suffices. The covers produced in \cite{ams92}, and indeed in all then
known proofs of this duality theorem actually had various maximality
properties. Aharoni, Magidor, Shore proved that the existence of covers with
any of a variety of maximality properties actually implies $\Pi_{1}^{1}%
$-$\mathsf{CA}_{0}$.

\begin{thm}
\label{thm:MIRT_equiv_Pi11-CA} $\Pi_{1}^{1}$-$\mathsf{CA}_{0}$ is equivalent
to $\mathsf{MIRT}_{\mathrm{XYZ}}$.
\end{thm}

\begin{proof}
[Proof that $\mathsf{MIRT}_{\mathrm{XYZ}}$ implies $\Pi_{1}^{1}$%
-$\mathsf{CA}_{0}$]We first prove that $\mathsf{MIRT}_{\mathrm{XYZ}}$ implies
$\mathsf{ACA}_{0}$ by adapting the proof of Theorem
\ref{thm:IRT_variants_ACA_0_and_omega_model_IRT_variants_hyp_closed}: If we
apply $\mathsf{MIRT}_{\mathrm{XYZ}}$ instead of $\mathsf{IRT}_{\mathrm{XYZ}}$
to any of the forests constructed in that proof, we obtain a sequence
containing a $Z$-ray in each tree which constitutes the forest. This is more
than sufficient for carrying out the remainder of the proof of Theorem
\ref{thm:IRT_variants_ACA_0_and_omega_model_IRT_variants_hyp_closed}.

To prove that $\mathsf{MIRT}_{\mathrm{UVS}}$ implies $\Pi_{1}^{1}%
$-$\mathsf{CA}_{0}$, suppose we are given a set $A$. Consider the disjoint
union of all $A$-computable trees (this exists, by $\mathsf{ACA}_{0}$). Any
maximal sequence of $Y$-disjoint rays in this forest must contain a ray in each
ill-founded $A$-computable tree. Hence its jump computes the hyperjump $T^{A}$. This shows
that $\mathsf{MIRT}_{\mathrm{UYS}}$ implies $\Pi_{1}^{1}$-$\mathsf{CA}_{0}$.
To prove that the other $\mathsf{MIRT}_{\mathrm{XYZ}}$ imply $\Pi_{1}^{1}%
$-$\mathsf{CA}_{0}$, it suffices to exhibit a computable procedure which takes
trees $T\subseteq N^{<N}$ to $X$-graphs $T^{\prime}$ such that $T$ is
ill-founded if and only if $T^{\prime}$ contains a $Z$-ray. For $\mathsf{MIRT}%
_{\mathrm{UYD}}$, it suffices to modify each tree by adding a computable
branch which is not already on the tree (as we did in the proof of Theorem
\ref{thm:IRT_variants_ACA_0_and_omega_model_IRT_variants_hyp_closed}). For
$\mathsf{MIRT}_{\mathrm{DYZ}}$, it suffices to orient each of the graphs we
constructed above in the obvious way.
\end{proof}

\begin{proof}
[Proof that $\Pi_{1}^{1}$-$\mathsf{CA}_{0}$ implies $\mathsf{MIRT}%
_{\mathrm{XYZ}}$]First, we give a mathematical proof for $\mathsf{MIRT}%
_{\mathrm{XVZ}}$ that is a direct construction not relying on Zorn's Lemma or
the like. We will then explain how to modify it to apply to the other cases
and then how to get it to work in $\Pi_{1}^{1}$-$\mathsf{CA}_{0}$.

Suppose we are given an $X$-graph $G$ whose vertices are elements of $N$. We
build a sequence of disjoint $Z$-rays in $G$ by recursion. If there are none we
are done. Otherwise start with $R_{0}$ as any $Z$-ray in $G$. Suppose at stage
$n$ we have constructed disjoint $Z$-rays $\seq{R_{i}}_{i<m}$ for some $m\leq n$
such that for each $i$, $R_{i}$ begins at $x_{i}$. If there is a $Z$-ray
beginning at $n$ which is disjoint from the $R_{i}$ for $i<m$, choose one as $R_{m+1}$, if not move on to stage $n+1$. This construction produces a
(possibly finite) sequence $R_{i},R_{i},\dots$ of disjoint $Z$-rays in $G$. We
show that this sequence is maximal. If $R$ is a $Z$-ray which is disjoint from
every $R_{i}$, then go to stage $n$ of the construction, where $n$ is the
first vertex of $R$. If we did insert some $R_{m}$ during stage $n$, then
$R_{m}$ would not be disjoint from $R$. Hence we did not insert any $Z$-ray
during stage $n$. But $R$ is a $Z$-ray that begins at $n$ and is disjoint from
$\seq{R_{i_{j}}}_{i_{j}<n}$, contradiction.

To prove $\mathsf{MIRT}_{\mathrm{XEZ}}$, we modify the above construction as
follows. At stage $n=(u,v)$, we search for a $Z$-ray $R$ which is disjoint from
the previous rays and has $(u,v)$ as its first edge. The rest of the proof
proceeds as above.

The only real obstacle in formalizing the above proofs even in $\mathsf{ACA}_{0}$ is
being able to find out at stage $n$ if there is an $R$ as requested and, if
so, choosing one. The question is $\Sigma_{1}^{1}$ and so in $\Pi_{1}^{1}%
$-$\mathsf{CA}_{0}$ we can answer it and then perhaps use some construction or
choice principle to produce it. As we only get a yes answer some of the time,
$\Sigma_{1}^{1}$-$\mathsf{AC}_{0}$ does not seem sufficient. Also later
choices depend on previous ones. A computability argument using the Gandy
basis theorem and its uniformities works but requires more background
development. A choice principle that returns an element if there is one
satisfying a $\Sigma_{1}^{1}$ property (but may act arbitrarily otherwise) and
that can be iterated in a recursion is \emph{strong $\Sigma_{1}^{1}%
$-$\mathsf{DC}_{0}$} which consists of the scheme
\[
(\exists W)(\forall n)(\forall Y)\left(  \Phi\left(  n,\bigoplus_{i<n}%
W^{[i]},Y\right)  \rightarrow\Phi\left(  n,\bigoplus_{i<n}W^{[i]}%
,W^{[n]}\right)  \right)  ,
\]
for any $\Sigma_{1}^{1}$ formula $\Phi(n,X,Y)$. It is known that strong
$\Sigma_{1}^{1}$-$\mathsf{DC}_{0}$ and $\Pi_{1}^{1}$-$\mathsf{CA}_{0}$ are
equivalent \cite[VII.6.9]{sim_book}. This clearly has the right
flavor and the only issue is defining the required $\Phi(n,X,Y)$ with
parameter $G$. This is slightly fussy but not problematic. We provide the
details: To define $\Phi$, first recursively define a finite sequence
$i_{0},\dots,i_{k}<n$. If $i_{0},\dots,i_{j-1}$ have been defined, define
$i_{j}$ to be the least number (if any) above $i_{j-1}$ and below $n$ such
that $X^{[i_{j}]}$ is a $Z$-ray in $G$ which is disjoint from $X^{[i_{0}]}%
,\dots,X^{[i_{j-1}]}$. It is clear that there is an arithmetic formula with
parameter $G$ which defines $i_{0},\dots,i_{k}$ from $n$ and $X$. Next, we say
that $\Phi(n,X,Y)$ holds if $Y$ is a $Z$-ray in $G$ which begins with $n$, and
$Y$ is disjoint from $X^{[i_{0}]},\dots,X^{[i_{k}]}$.

Apply strong $\Sigma_{1}^{1}$-$\mathsf{DC}_{0}$ for the formula $\Phi$ to
obtain some set $W$. By $\Sigma_{1}^{0}$-comprehension with parameter $G\oplus
W$, we may inductively define a (possibly finite) sequence $i_{0},i_{1},\dots
$, just as we did in the definition of $\Phi$. Clearly $\seq{W^{[i_{j}]}}_{j}$ is
a sequence of disjoint $Z$-rays in $G$. We claim that it is maximal.

Suppose that $R$ is a $Z$-ray in $G$ which is disjoint from every $Z^{[i_{j}]}$.
Suppose that $R$ begins with vertex $n$. Then $R$ is disjoint from
$W^{[i_{0}]},\dots,W^{[i_{k}]}$, where $i_{k}$ is the largest $i_{j}$ below
$n$. It follows that $\Phi(n,\bigoplus_{i<n}W^{[i]},R)$ holds. So
$\Phi(n,\bigoplus_{i<n}W^{[i]},W^{[n]})$ holds, i.e., $W^{[n]}$ is a Z-ray
which begins with $n$ and $W^{[n]}$ is disjoint from $W^{[i_{0}]}%
,\dots,W^{[i_{k}]}$. By definition of $i_{k+1}$, that means that $n=i_{k+1}$.
But then $R$ and $W^{[i_{k+1}]}$ are not disjoint, contradiction.

A slicker proof suggested by the referee requires perhaps more background in
the metamathematics of Reverse Mathematics. It uses a \textquotedblleft
countable coded $\beta$-model\textquotedblright\ with the given graph $G$ as
an element. By $\Pi_{1}^{1}$-$\mathsf{CA}_{0}$, for any set $G$ there is a set $X$ such that $X^{[0]}=G$ and
$B(X)=\seq{N,\{X^{[i]}\mid i\in N\}}$ is a $\beta$-model, i.e.\ any $\Sigma_{1}%
^{1}$ formula $\Phi$ with parameters from among the $X^{[i]}$ is true if and
only if it is true in $B(X)$ \cite[VII.2.10]{sim_book}. Now one carries out the mathematical proof above
but whenever one asks if there is a ray with some property one asks if there
is one such among the $X^{[i]}$. This converts the $\Sigma_{1}^{1}$ questions
to ones of fixed arithmetic complexity in $X$. Then, if the answer is yes,
finding an appropriate $i$ is also arithmetic in $X$ with fixed complexity.
This converts the entire construction to one arithmetic in $X$ (and so
certainly in $\Pi_{1}^{1}$-$\mathsf{CA}_{0}$). Consider now the claim that the
sequence of rays given by this construction over $B(X)$ is actually maximal.
The existence of a counter-example $R$ with first element $n$ gives the same
contradiction as before. The point is that, as $B(X)$ is a $\beta$-model, the
answer to the question asked at stage $n$ of the construction of whether there
is a ray in $B(X)$ with first vertex $n$ disjoint from the finite sequence is the same as the question as to whether there exists one at all. The only
other point to note is that even though the sequence of rays $R_{i}$ is
constructed outside of $B(X)$ each finite initial segment is in $B(X)$ by an
external induction and the fact that $B(X)$ is obviously a model of
$\mathsf{ACA}_{0}$. The argument with the adjustments for $\mathsf{MIRT}%
_{\mathrm{XEZ}}$ is then the same.
\end{proof}

\section{Relationships Between $\mathsf{IRT}$ and Other Theories of
Hyperarithmetic Analysis}

\label{section:IRT_vs_other_theories}

In this section, we establish implications and nonimplications between
variants of $\mathsf{IRT}$ and THAs other than
$\Sigma_{1}^{1}\text{-}\mathsf{AC}_{0}$. One such standard theory is as follows:

\begin{defn}
\label{defn:unique_choice} The theory \emph{unique-$\Sigma_{1}^{1}%
\text{-}\mathsf{AC}_{0}$} consists of $\mathsf{RCA}_{0}$ and the principle
\[
(\forall n)(\exists!X)A(n,X)\rightarrow(\exists Y)(\forall
n)A(n,Y^{[n]})
\]
for each arithmetic formula $A(n,X)$.
\end{defn}

The above theory is typically known as weak-$\Sigma_{1}^{1}\text{-}%
\mathsf{AC}_{0}$ (e.g., \cite[VIII.4.12]{sim_book}). We deviate from this
terminology to introduce a new choice principle where the requirement for
unique solutions is replaced by one for finitely many solutions.

\begin{defn}
\label{defn:finite_choice} The theory \emph{finite-$\Sigma_{1}^{1}%
\text{-}\mathsf{AC}_{0}$} consists of $\mathsf{RCA}_{0}$ and the principle
\[
(\forall n)(\exists\text{ nonzero finitely many }X)A(n,X)\rightarrow
(\exists Y)(\forall n)A(n,Y^{[n]})
\]
for each arithmetic formula $A(n,X)$. Formally, \textquotedblleft$(\exists$
nonzero finitely many $X)A(n,X)$\textquotedblright\ means that there is a
nonempty sequence $\seq{X_{i}}_{i<j}$ such that for each $X$, $A(n,X)$ holds if
and only if $X=X_{i}$ for some $i<j$.
\end{defn}

Similarly to $\Sigma_{1}^{1}$-$\AC$, each of these two choice principles are
equivalent to ones where $A$ is allowed to be of the form $(\exists
!Y)B(n,X,Y)$ or $(\exists$ nonzero finitely many $Y)B(n,X,Y)$, respectively.
However, unlike $\Sigma_{1}^{1}$-$\AC$ neither of these two principle is
equivalent to the version where $A$ is allowed to be $\Sigma_{1}^{1}$. Not
only would those versions fail to capture the idea that we are dealing with
unique or finitely many witnesses and paths through trees but they should be
stronger than the stated principles. It is easy to see, for example, that even
the unique version with $A$ $\Sigma_{1}^{1}$ implies $\Delta_{1}^{1}$-$\CA$ (Definition \ref{defn:Delta11-CA}) which is stronger than unique-$\Sigma
_{1}^{1}$-$\AC$ by Van Wesep \cite{vw_thesis}.

Since the THA $\Sigma_{1}^{1}\text{-}\mathsf{AC}_{0}$ implies finite-$\Sigma
_{1}^{1}\text{-}\mathsf{AC}_{0}$ which in turn implies unique-$\Sigma_{1}%
^{1}\text{-}\mathsf{AC}_{0}$ whose models are closed under hyperarithmetic
reducibility by Proposition \ref{jumpstrees}, it follows that finite-$\Sigma
_{1}^{1}\text{-}\mathsf{AC}_{0}$ is a THA (as is unique-$\Sigma_{1}%
^{1}\text{-}\mathsf{AC}_{0}$). Goh \cite{goh_finite_choice} shows that
finite-$\Sigma_{1}^{1}\text{-}\mathsf{AC}_{0}$ is strictly stronger than
unique-$\Sigma_{1}^{1}\text{-}\mathsf{AC}_{0}$. We were led to study this
version of choice by realizing that a variant of our original proof that
$\mathsf{IRT}_{\mathrm{UVS}}^{\ast}$ implies unique-$\Sigma_{1}^{1}%
\text{-}\mathsf{AC}_{0}$ worked for the finite version.

\begin{thm}
\label{thm:XYZmc_implies_finite_choice} $\mathsf{IRT}^{\ast}_{\mathrm{XYZ}}$
implies finite-$\Sigma^{1}_{1}\text{-}\mathsf{AC}_{0}$ over $\mathsf{RCA}_{0}%
$. (It follows that $\mathsf{IRT}_{\mathrm{XYZ}}$ implies finite-$\Sigma
^{1}_{1}\text{-}\mathsf{AC}_{0}$ over $\mathsf{RCA}_{0} + \mathsf{I}\Sigma^{1}_{1}$,
but this is superseded by Theorem \ref{thm:IRTmc_implies_ABW} below.)
\end{thm}

\begin{proof}
We first prove that $\mathsf{IRT}_{\mathrm{UYS}}^{\ast}$ for forests implies
finite-$\Sigma_{1}^{1}\text{-}\mathsf{AC}_{0}$. By Lemma
\ref{lem:arith_formula_paths_on_tree_seq}, it suffices to prove that for any
sequence $\seq{T_{n}}_{n}$ of subtrees of $N^{<N}$ such that each $T_{n}$ has
finitely many branches, a sequence $\seq{P_{n}}_{n}$ exists with each $P_{n}%
\in\lbrack T_{n}]$. As in the proof of Theorem
\ref{thm:DVD-IRTmc_implies_ISigma11}, we construct a sequence of trees
$\seq{S_{n}}_{n}$ such that for each $n$, the branches on $S_{n}$ are precisely
those of the form $P_{0}\oplus\dots\oplus P_{n}$ where $P_{i}$ is a branch on
$T_{i}$ for $i\leq n$.

By $\mathsf{IRT}_{\mathrm{UYS}}^{\ast}$ (for forests) there is a sequence
$\seq{R_{k}}_{k}$ of $Y$-disjoint rays in $\bigsqcup_{n}S_{n}$ of maximum
cardinality. We claim that $\seq{R_{k}}_{k}$ is infinite. If not, let $m$ be least
such that there is no $R_{k}$ in $S_{m}$. Then we can increase the cardinality
of $\seq{R_{k}}_{k}$ by adding any ray $R$ from $S_{m}$ while maintaining
disjointness by our choice of $n$. The point here is that if $R\cap R_{k}%
\neq\emptyset$ for any $k$ then they are both in $S_{m}$ as the trees $S_{n}$
are disjoint and there are no edges between them. Therefore $\seq{R_{k}}_{k}$ has
a ray in infinitely many $S_{n}$. Thus we may construct the desired sequence $\seq{P_{n}}_{n}$ recursively by searching at stage $n$ for an $R_{k}$ in $S_{m}$
for some $m>n$ and take $P_{n}$ to be the branch in $T_{n}$ which shares a
tail with the $n$th coordinate of $R_{k}$.

Now, by Theorem \ref{thm:UYDmc_implies_UVSmc_forests}, it follows that
$\mathsf{IRT}_{\mathrm{UYD}}^{\ast}$ for forests implies finite-$\Sigma
_{1}^{1}\text{-}\mathsf{AC}_{0}$. By Proposition
\ref{prop:IRTmc_implications_variants}, it follows that $\mathsf{IRT}%
_{\mathrm{DYZ}}^{\ast}$ implies finite-$\Sigma_{1}^{1}\text{-}\mathsf{AC}_{0}$
as well and so we are done.
\end{proof}

Another theory of hyperarithmetic analysis which follows from $\mathsf{IRT}%
_{\mathrm{XYZ}}^{\ast}$ is \emph{arithmetic Bolzano-Weierstrass}
($\mathsf{ABW}_{0}$):

\begin{defn}
\label{defn:ABW} The theory \emph{$\mathsf{ABW}_{0}$} consists of
$\mathsf{RCA}_{0}$ and the following principle: If $A(X)$ is an arithmetic
predicate on $2^{N}$, either there is a finite sequence $\seq{X_{i}}_i$ which
contains every $X$ such that $A(X)$ holds or there is an $X$ such that every
one of its neighborhoods has two $Y$ such that $A(Y)$ holds. Such an $X$ is
called an accumulation point of the class $\{X\mid A(X)\}$.
\end{defn}

Friedman \cite{friedman_icm} introduced $\mathsf{ABW}_{0}$ and asserted that
it follows from $\Sigma_{1}^{1}\text{-}\mathsf{AC}_{0}$ (with unrestricted
induction). Conidis \cite{conidis12} proved Friedman's assertion and
established relationships between $\mathsf{ABW}_{0}$ and most then known
theories of hyperarithmetic analysis. Goh \cite{goh_finite_choice} shows that
$\mathsf{ABW}_{0}+\mathsf{I}\Sigma_{1}^{1}$ implies finite-$\Sigma_{1}^{1}%
\text{-}\mathsf{AC}_{0}$. We do not know if $\mathsf{ABW}_{0}$ is strictly
stronger than finite-$\Sigma_{1}^{1}\text{-}\mathsf{AC}_{0}$.

The following two lemmas will be useful in deriving $\mathsf{ABW}_{0}$ from
$\mathsf{IRT}^{\ast}_{\mathrm{XYZ}}$. The first lemma describes a connection
between sets of solutions of arithmetic predicates and disjoint rays in trees.

\begin{lem}
[$\mathsf{ACA}_{0}$]\label{lem:arith_predicate_disjoint_rays_in_tree} Suppose
$A(X)$ is an arithmetic predicate. Then there is a tree $T \subseteq N^{<N}$
such that if there is a sequence of distinct solutions of $A(X)$, then there
is a sequence of Y-disjoint single rays in $T$ of the same cardinality, and
vice versa.
\end{lem}

\begin{proof}
By Lemma \ref{lem:arith_formula_paths_on_tree}, there is a tree $T\subseteq
N^{<N}$ such that
\begin{align*}
&  \forall X(A(X)\leftrightarrow\exists f(\langle X,f\rangle\in\lbrack T])\\
\text{and}\;  &  \forall X(\exists\text{ at most one }f)(\langle X,f\rangle
\in\lbrack T])\text{.}%
\end{align*}
If $\seq{X_{i}}_{i}$ is a sequence of distinct solutions of $A(X)$, then, as the
required $f_{i}$ are arithmetic uniformly in the $X_{i}$, there is a sequence
of distinct branches $\seq{\seq{X_{i},f_{i}}}_{i}$ on $T$ of the same cardinality.

By taking an appropriate tail of each branch, we obtain a sequence
$\seq{R_{i}}_{i}$ of vertex-disjoint (hence edge-disjoint) single rays in $T$ of
the same cardinality with each one being a tail of $\langle X_{i},f_{i}%
\rangle$: As no two distinct branches in a tree can have infinitely many
vertices in common, simply take $R_{n}$ to be the tail of $\left\langle
X_{n},f_{n}\right\rangle $ starting after all vertices it has in common with
any $\langle X_{i},f_{i}\rangle$, $i<n$.

Conversely, suppose there is a sequence $\seq{R_{i}}_{i}$ of $Y$-disjoint single
rays in $T$. For each $R_{i}$, we define a branch on $T$ which corresponds to
it as follows. Let $x$ be the vertex in $R_{i}$ which is closest to the root
of $T$. Then we can extend $xR_{i}$ to the root to obtain a branch $\langle
X_{i},f_{i}\rangle$ on $T$. We claim that $\seq{X_{i}}_{i}$ is a sequence of
distinct solutions of $A(X)$. For each $i\neq j$, since $R_{i}$ and $R_{j}$
are $Y$-disjoint, they cannot share a tail. So $\langle X_{i},f_{i}\rangle$ and
$\langle X_{j},f_{j}\rangle$ must be distinct. Since for each $X$, there is at
most one $f$ such that $\langle X,f\rangle$ is a branch on $T$, it follows
that $X_{i}\neq X_{j}$ as desired.
\end{proof}

The second lemma is essentially the well-known fact that the
Bolzano-Weierstrass theorem is provable in $\mathsf{ACA}_{0}$:

\begin{lem}
[{\cite[III.2.7]{sim_book}}]\label{lem:seq_Cantor_accumulation} $\mathsf{ACA}%
_{0}$ proves that if $\seq{X_{n}}_{n}$ is a sequence of distinct elements of
$2^{N}$, then there is some $Z$ which is an accumulation point of
$\{X_{n} \mid n\in N\}$.
\end{lem}

\begin{thm}
\label{thm:IRTmc_implies_ABW} $\mathsf{IRT}_{\mathrm{XYZ}}^{\ast}$ implies
$\mathsf{ABW}_{0}$ over $\mathsf{RCA}_{0}$. Therefore $\mathsf{IRT}%
_{\mathrm{XYZ}}$ implies $\mathsf{ABW}_{0}$ over $\mathsf{RCA}_{0}+\mathsf{I}\Sigma
_{1}^{1}$.
\end{thm}

\begin{proof}
By Proposition \ref{prop:IRTmc_implications_variants}, it suffices to show
that the undirected variants of $\mathsf{IRT}_{{}}^{\ast}$ imply
$\mathsf{ABW}_{0}$.

Suppose $A(X)$ is an arithmetic predicate on $2^{N}$ such that no finite
sequence $\seq{X_i}_i$ contains every $X$ such that
$A(X)$ holds. By Lemma \ref{lem:arith_predicate_disjoint_rays_in_tree}, there
is a tree $T\subseteq N^{<N}$ such that for any sequence of distinct solutions
of $A(X)$, there is a sequence of $Y$-disjoint single rays in $T$ of the same
cardinality, and vice versa.

By $\mathsf{IRT}_{\mathrm{UYS}}^{\ast}$, or by $\mathsf{IRT}_{\mathrm{UYD}%
}^{\ast}$ and Theorem \ref{thm:UYDmc_implies_UVSmc_forests}, there is a
sequence of $Y$-disjoint single rays in $T$ of maximum cardinality. This yields
a sequence of distinct solutions of $A(X)$ of the same cardinality.

If this sequence is finite, then there is a solution $Y$ of $A(X)$ not in the
sequence by our assumption. Hence there is a sequence of distinct solutions of
$A(X)$ of larger cardinality, which yields a sequence of $Y$-disjoint single
rays in $T$ of larger cardinality for the desired contradiction.

Thus there is an infinite sequence $\seq{X_{n}}_{n}$ of distinct solutions of $A$.
By Lemma \ref{lem:seq_Cantor_accumulation}, there is an accumulation point of
$\{X_{n} \mid n\in N\}$, which is of course an accumulation point of $\{X\mid A(X)\}$,
as desired.
\end{proof}

We now turn our attention to nonimplications. One prominent theory of
hyperarithmetic analysis is the scheme of \emph{$\Delta_{1}^{1}$%
-comprehension} (studied by Kreisel \cite{kreisel}):

\begin{defn}
\label{defn:Delta11-CA} The theory $\Delta_{1}^{1}$-$\mathsf{CA}_{0}$ consists
of $\mathsf{RCA}_{0}$ and the principle
\[
(\forall n)(\Phi(n)\leftrightarrow\lnot\Psi(n))\rightarrow\exists X(n\in
X\leftrightarrow\Phi(n))
\]
for all $\Sigma_{1}^{1}$ formulas $\Phi$ and $\Psi$.
\end{defn}

\begin{thm}
\label{thm:Delta11-CA_not_imply_IRT} $\Delta^{1}_{1}$-$\mathsf{CA}_{0} \nvdash\mathsf{IRT}_{\mathrm{XYZ}}, \mathsf{IRT}^{\ast}_{\mathrm{XYZ}}$.
\end{thm}

\begin{proof}
Conidis \cite[Theorem 3.1]{conidis12} constructed a standard model which
satisfies $\Delta^{1}_{1}$-$\mathsf{CA}_{0}$ but not $\mathsf{ABW}_{0}$. By
Theorem \ref{thm:IRTmc_implies_ABW}, this model does not satisfy
$\mathsf{IRT}^{\ast}_{\mathrm{XYZ}}$. Since standard models satisfy full
induction, this model does not satisfy $\mathsf{IRT}_{\mathrm{XYZ}}$ either
(by Proposition \ref{prop:IRT_IRTmc_rshp}).
\end{proof}

\begin{thm}
\label{thm:ABW_not_imply_IRT} $\mathsf{ABW}_{0} \nvdash\mathsf{IRT}%
_{\mathrm{XYZ}},\mathsf{IRT}^{\ast}_{\mathrm{XYZ}}$.
\end{thm}

\begin{proof}
By Propositions \ref{prop:IRT_DYZ_implies_IRT_UYZ} and
\ref{prop:IRT_IRTmc_rshp}, it suffices to show that $\mathsf{ABW}_{0}%
\nvdash\mathsf{IRT}_{\mathrm{UYZ}}$. Van Wesep \cite[I.1]{vw_thesis}
constructed a standard model $\mathcal{N}$ which satisfies unique-$\Sigma
_{1}^{1}$-$\mathsf{AC}_{0}$ but not $\Delta_{1}^{1}$-$\mathsf{CA}_{0}$.
Conidis \cite[Theorem 4.1]{conidis12}, using the approach of \cite{neeman08},
showed that $\mathcal{N}$ satisfies $\mathsf{ABW}_{0}$. We show below that
$\mathcal{N}$ does not satisfy $\mathsf{IRT}_{\mathrm{UYZ}}$.

In order to define $\mathcal{N}$, van Wesep constructed a tree $T^{G}$ and
branches $\seq{f_{i}^{G}}_{i\in\mathbb{N}}$ of $T^{G}$ such that (1) $\mathcal{N}$
contains $T^{G}$ and infinitely many (distinct) $f_{i}^{G}$ (see \cite[pg.\ 13
l.\ 1--11]{vw_thesis}); (2) $\mathcal{N}$ does not contain any infinite
sequence of distinct branches of $T^{G}$ (see \cite[pg.\ 12 l.\ 7--9]%
{vw_thesis} and Steel \cite[Lemma 7]{steel78}.) Then $T^{G}$ is an instance of
$\mathsf{IRT}_{\mathrm{UYS}}$ in $\mathcal{N}$ which has no solution in
$\mathcal{N}$. This shows that $\mathcal{N}$ does not satisfy $\mathsf{IRT}%
_{\mathrm{UYS}}$ for trees. (The reader who wants to follow the details of the
proofs in \cite{conidis12} and \cite[I.1]{vw_thesis} should look at the
presentation of the basic methods in \cite{neeman08}.)

Since $\mathcal{N}$ is a standard model, it satisfies full induction. By
Theorem \ref{thm:UYDmc_implies_UVSmc_forests}, it follows that $\mathcal{N}$
does not satisfy $\mathsf{IRT}_{\mathrm{UYD}}$ for forests.
\end{proof}

Figure \ref{fig:hyp_analysis_zoo_IRT} illustrates some of our results. In
order to simplify the diagram, we have omitted all variants of $\mathsf{IRT}%
_{}$ except $\mathsf{IRT}_{\mathrm{UVS}}$.

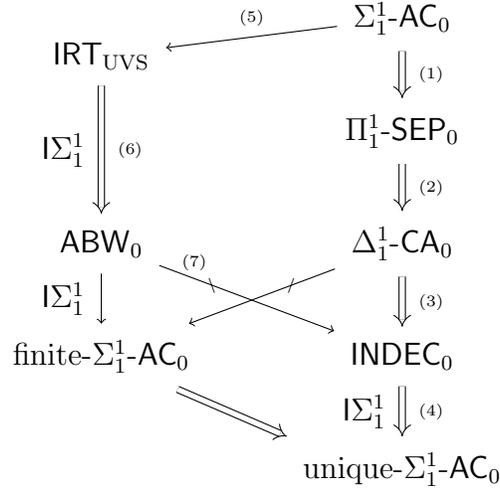
\begin{figure}
\centering
\begin{tikzpicture}[scale=1,every node/.style={inner sep=6pt}]
	\node (Sigma11AC) at (0,6) {$\Sigma^1_1$-$\AC_0$};
	\node (Pi11SEP) at (0,4.5) {$\Pi^1_1$-$\SEP_0$};
	\node (Delta11CA) at (0,3) {$\Delta^1_1$-$\CA_0$};
	\node (INDEC) at (0,1.5) {$\INDEC_0$};
	\node (uniqueAC) at (0,0) {unique-$\Sigma^1_1$-$\AC_0$};
	\node (UVS) at (-4,5.5) {$\IRT{UVS}$};
	\node (ABW) at (-4,3) {$\ABW_0$};
	\node (finiteAC) at (-4,1.5) {finite-$\Sigma^1_1$-$\AC_0$};
	\draw [-implies,double,double distance=2pt] (Sigma11AC) -- node[right] {\tiny (1)} (Pi11SEP);
	\draw [-implies,double,double distance=2pt] (Pi11SEP) -- node[right] {\tiny (2)} (Delta11CA);
	\draw [-implies,double,double distance=2pt] (Delta11CA) -- node[right] {\tiny (3)} (INDEC);
	\draw [-implies,double,double distance=2pt] (INDEC) -- node[left] {$\mathsf{I}\Sigma^1_1$} node[right] {\tiny (4)} (uniqueAC);
	\draw [->] (Sigma11AC) -- node[above=-2pt] {\tiny (5)} (UVS);
	\draw [-implies,double,double distance=2pt] (UVS) -- node[right] {\tiny (6)} node[left] {$\mathsf{I}\Sigma^1_1$} (ABW);
	\draw [->] (ABW) -- node[left] {$\mathsf{I}\Sigma^1_1$} (finiteAC);
	\draw [-implies,double,double distance=2pt] (finiteAC) -- (uniqueAC.north west);
	\draw [->] (Delta11CA) -- (finiteAC) node [anchor=center, pos=0.3, rotate = -20] {\tiny $|$};
	\draw [->] (ABW) -- node[above = -3pt, pos = 0.2] {\tiny (7)} (INDEC) node [anchor=center, pos =0.3, rotate = 20] {\tiny $|$};
\end{tikzpicture}
\caption{Partial zoo of theories of hyperarithmetic analysis. Single arrows
indicate implication while double arrows indicate strict implication. The
references for the above results are as follows: (1, 2) Montalb\'an
\cite[Theorems 2.1, 3.1]{montalban_pi11sep}; (3, 4) Montalb\'an \cite[Theorem
2.2]{montalban_jullien}, Neeman \cite[Theorems 1.2, 1.3, 1.4]{neeman08}, see
also Neeman \cite[Theorem 1.1]{neeman11}; (5) Theorem
\ref{thm:Sigma11-AC_implies_IRT}; (6) Theorems \ref{thm:IRTmc_implies_ABW},
\ref{thm:ABW_not_imply_IRT}; (7) Conidis \cite[Theorem 4.1]{conidis12}. All
results concerning finite-$\Sigma^{1}_{1}\text{-}\mathsf{AC}_{0}$ are in Goh
\cite{goh_finite_choice}.}%
\label{fig:hyp_analysis_zoo_IRT}%
\end{figure}

\section{Isolating the Use of $\Sigma^{1}_{1}$-$\mathsf{AC}_{0}$ in Proving
$\mathsf{IRT}_{}$}

\label{section:WIRT}

We isolate the use of $\Sigma^{1}_{1}$-$\mathsf{AC}_{0}$ in our proofs of
$\mathsf{IRT}_{\mathrm{XYS}}$ and $\mathsf{IRT}_{\mathrm{UVD}}$ (Theorems
\ref{thm:Sigma11-AC_implies_IRT}, \ref{thm:Sigma11-AC_implies_UVD},
\ref{thm:Sigma11-AC_implies_XES}) by identifying the following principles:

\begin{defn}
Let $\mathsf{SCR}_{\mathrm{XYZ}}$ be the assertion that if $G$ is an $X$-graph
with arbitrarily many $Y$-disjoint $Z$-rays, then there is a sequence of sets
$\seq{X_{k}}_{k}$ such that for each $k \in N$, $X_{k}$ is a set of $k$ $Y$-disjoint $Z$-rays in $G$.

Let $\mathsf{WIRT}_{\mathrm{XYZ}}$ be the assertion that if $G$ is an $X$-graph and there is a sequence of sets $\seq{X_{k}}_{k}$ such that for each $k \in N$, $X_{k}$ is a set of $k$ $Y$-disjoint $Z$-rays in $G$, then $G$ has infinitely many $Y$-disjoint $Z$-rays.
\end{defn}

$\mathsf{SCR}_{}$ stands for Strongly Collecting Rays. $\mathsf{WIRT}_{}$
stands for Weak Infinite Ray Theorem.

It is clear that $\Sigma^{1}_{1}$-$\mathsf{AC}_{0}$ implies $\mathsf{SCR}%
_{\mathrm{XYZ}}$ and $\mathsf{SCR}_{\mathrm{XYZ}} + \mathsf{WIRT}%
_{\mathrm{XYZ}}$ implies $\mathsf{IRT}_{\mathrm{XYZ}}$. The only use of
$\Sigma^{1}_{1}$-$\mathsf{AC}_{0}$ in our proofs of $\mathsf{IRT}%
_{\mathrm{XYS}}$ and $\mathsf{IRT}_{\mathrm{UVD}}$ is to prove $\mathsf{SCR}%
_{\mathrm{XYS}}$ and $\mathsf{SCR}_{\mathrm{UVD}}$ respectively:

\begin{thm}
\label{thm:ACA0_implies_WIRT} $\mathsf{ACA}_{0}$ proves $\mathsf{WIRT}%
_{\mathrm{XYS}}$ and $\mathsf{WIRT}_{\mathrm{UVD}}$.
\end{thm}

\begin{proof}
For $\mathsf{WIRT}_{\mathrm{UVS}}$, see the proof of Theorem
\ref{thm:Sigma11-AC_implies_IRT}(ii). In particular, note that the hypothesis
of $\mathsf{WIRT}$ is exactly the instance of $\Sigma_{1}^{1}$-$\AC$ needed in that proof (and the one referenced there). From then on, the argument proceeds in $\ACA_0$ to give the conclusion of $\IRT{}$ and so of $\WIRT{}$ as desired. Similarly,
for $\mathsf{WIRT}_{\mathrm{UVD}}$, see the proof of Theorem
\ref{thm:Sigma11-AC_implies_UVD} and for $\mathsf{WIRT}_{\mathrm{XES}}$, see
the proof of Theorem \ref{thm:Sigma11-AC_implies_XES}. The desired result for
$\mathsf{WIRT}_{\mathrm{DVS}}$ then follows from Lemma
\ref{lem:DEZ_implies_DVZ}.
\end{proof}

Next we will use the above result to show that $\mathsf{SCR}_{\mathrm{XYS}}$
and $\mathsf{SCR}_{\mathrm{UVD}}$ are equivalent over $\mathsf{RCA}_{0}$ to
$\mathsf{IRT}_{\mathrm{XYS}}$ and $\mathsf{IRT}_{\mathrm{UVD}}$ respectively.
First, observe that $\mathsf{IRT}_{\mathrm{XYZ}}$ implies $\mathsf{SCR}%
_{\mathrm{XYZ}}$ for each choice of $XYZ$. Second, Lemmas
\ref{lem:DYZ_implies_UYZ}, \ref{lem:DEZ_implies_DVZ} and
\ref{lem:DYD_implies_DYS} imply

\begin{prop}
\label{prop:SCR_implications_variants} $\mathsf{SCR}_{\mathrm{DYZ}}$ implies
$\mathsf{SCR}_{\mathrm{UYZ}}$, $\mathsf{SCR}_{\mathrm{DEZ}}$ implies
$\mathsf{SCR}_{\mathrm{DVZ}}$ and $\mathsf{SCR}_{\mathrm{DYD}}$ implies
$\mathsf{SCR}_{\mathrm{DYS}}$.
\end{prop}

\begin{prop}
\label{prop:SCR_implies_ACA0} $\mathsf{SCR}_{\mathrm{XYZ}}$ implies
$\mathsf{ACA}_{0}$.
\end{prop}

\begin{proof}
By Proposition \ref{prop:SCR_implications_variants}, it suffices to establish
the desired result for the undirected variants of $\mathsf{SCR}_{}$. The proofs are almost identical to those of Theorems \ref{thm:Sigma11-AC_implies_IRT}(i) and \ref{thm:IRT_variants_ACA_0_and_omega_model_IRT_variants_hyp_closed}. 
There, we applied $\mathsf{IRT}_{\mathrm{UYZ}}$ to forests $G = \bigsqcup_{n} T_{n}%
$, where each $T_{n}$ contains a $Z$-ray, and no two $Z$-rays in $T_{n}$ can be
$Y$-disjoint. Any infinite sequence of $Y$-disjoint $Z$-rays in $G$ must contain a
$Z$-ray in cofinally many graphs $T_{n}$. Therefore from such a sequence we can
uniformly compute $Z$-rays in cofinally many graphs $T_{n}$, which establishes
$\mathsf{ACA}_{0}$ by the construction of $\bigsqcup_{n} T_{n}$. If we assume
$\mathsf{SCR}_{\mathrm{UYZ}}$ instead of $\mathsf{IRT}_{\mathrm{UYZ}}$, we
only have access to a sequence $\seq{X_{k}}_{k \in N}$ such that for each $k$, $X_{k}$
is a set of $k$ $Y$-disjoint $Z$-rays in $G$. From such a sequence we can still
uniformly compute $Z$-rays in cofinally many graphs $T_{n}$, because for any
$k$, $X_{k+1}$ must contain a $Z$-ray in some $T_{n}$, $n \geq k$.
\end{proof}

By Theorem \ref{thm:ACA0_implies_WIRT}, Proposition
\ref{prop:SCR_implies_ACA0}, and the observation that $\mathsf{SCR}%
_{\mathrm{XYZ}} + \WIRT{\mathrm{XYZ}} \vdash\mathsf{IRT}%
_{\mathrm{XYZ}}$, we obtain

\begin{cor}
\label{cor:SCR_IRT_equiv} $\mathsf{SCR}_{\mathrm{XYZ}}$ and $\mathsf{IRT}%
_{\mathrm{XYZ}}$ are equivalent over $\mathsf{RCA}_{0}$ for the following choices of XYZ: XYS and UVD.
\end{cor}

We now turn our attention to $\mathsf{WIRT}_{\mathrm{XYZ}}$. As usual, Lemmas
\ref{lem:DYZ_implies_UYZ}, \ref{lem:DEZ_implies_DVZ} and
\ref{lem:DYD_implies_DYS} imply

\begin{prop}
\label{prop:WIRT_implications_variants} $\mathsf{WIRT}_{\mathrm{DYZ}}$ implies
$\mathsf{WIRT}_{\mathrm{UYZ}}$, $\mathsf{WIRT}_{\mathrm{DEZ}}$ implies
$\mathsf{WIRT}_{\mathrm{DVZ}}$ and $\mathsf{WIRT}_{\mathrm{DYD}}$ implies
$\mathsf{WIRT}_{\mathrm{DYS}}$.
\end{prop}

Recall that $\mathsf{WIRT}_{\mathrm{XYS}}$ and $\mathsf{WIRT}_{\mathrm{UVD}}$
are provable in $\mathsf{ACA}_{0}$ (Theorem \ref{thm:ACA0_implies_WIRT}).
$\mathsf{WIRT}_{\mathrm{DVD}}$ and $\mathsf{WIRT}_{\mathrm{DED}}$ are open,
because $\Sigma^{1}_{1}\text{-}\mathsf{AC}_{0}+\mathsf{WIRT}_{\mathrm{XYZ}}$
implies $\mathsf{IRT}_{\mathrm{XYZ}}$, and $\mathsf{IRT}_{\mathrm{DVD}}$ and
$\mathsf{IRT}_{\mathrm{DED}}$ are open (see comments after Theorem
\ref{thm:variants_IRT_hyp_analysis}). We do not have an upper bound on the
proof-theoretic strength of $\mathsf{WIRT}_{\mathrm{UED}}$ (an upper bound on
$\mathsf{WIRT}_{\mathrm{UED}}$ would yield an upper bound on $\mathsf{IRT}%
_{\mathrm{UED}}$, which we do not currently have).

We do not know if any $\mathsf{WIRT}_{\mathrm{XYZ}}$ is equivalent to
$\mathsf{ACA}_{0}$. In an effort to clarify the situation, we define an
apparent strengthening of $\mathsf{WIRT}_{\mathrm{XYZ}}$ and show that it
implies $\mathsf{ACA}_{0}$:

\begin{defn}
Let nonuniform-$\mathsf{WIRT}_{\mathrm{XYZ}}$ be the assertion that if $G$ is
an $X$-graph and there is a sequence of $Z$-rays $R_{0},R_{1},\dots$ in $G$ such
that for each $k$, there are $i_{0},\dots,i_{k}$ such that $R_{i_{0}}%
,\dots,R_{i_{k}}$ are $Y$-disjoint, then $G$ has infinitely many
$Y$-disjoint $Z$-rays.
\end{defn}

Every instance of $\mathsf{WIRT}_{\mathrm{XYZ}}$ is also an instance of
nonuniform-$\mathsf{WIRT}_{\mathrm{XYZ}}$, so nonuniform-$\mathsf{WIRT}%
_{\mathrm{XYZ}}$ implies $\mathsf{WIRT}_{\mathrm{XYZ}}$. Conversely, we have

\begin{prop}
\label{prop:ACA0_WIRT_implies_nonuniform_WIRT} $\mathsf{ACA}_{0} +
\mathsf{WIRT}_{\mathrm{XYZ}}$ implies nonuniform-$\mathsf{WIRT}_{\mathrm{XYZ}%
}$.
\end{prop}

\begin{proof}
Suppose $G$ is an instance of nonuniform-$\mathsf{WIRT}_{\mathrm{XYZ}}$, i.e.,
$G$ is an $X$-graph and $\seq{R_{n}}_{n\in N}$ is a sequence of $Z$-rays in $G$ such
that for each $k$, there are $i_{0},\dots,i_{k}$ such that $R_{i_{0}}%
,\dots,R_{i_{k}}$ are $Y$-disjoint. Then $\mathsf{ACA}_{0}$ can find
such $i_{0},\dots,i_{k}$ uniformly in $k$. Therefore by $\mathsf{ACA}_{0}$,
$G$ is an instance of $\mathsf{WIRT}_{\mathrm{XYZ}}$. By $\mathsf{WIRT}%
_{\mathrm{XYZ}}$, $G$ has infinitely many $Y$-disjoint $Z$-rays as desired.
\end{proof}

\begin{thm}
\label{thm:nonuniform_WIRT} Nonuniform-$\mathsf{WIRT}_{\mathrm{XYZ}}$ implies
$\mathsf{ACA}_{0}$ over $\mathsf{RCA}_{0}$. It follows that
nonuniform-$\mathsf{WIRT}_{\mathrm{XYS}}$ and nonuniform-$\mathsf{WIRT}%
_{\mathrm{UVD}}$ are both equivalent to $\mathsf{ACA}_{0}$ over $\mathsf{RCA}%
_{0}$.
\end{thm}

\begin{proof}
By Proposition \ref{prop:ACA0_WIRT_implies_nonuniform_WIRT} and Theorem
\ref{thm:ACA0_implies_WIRT}, $\mathsf{ACA}_{0}$ implies
nonuniform-$\mathsf{WIRT}_{\mathrm{XYS}}$ and nonuniform-$\mathsf{WIRT}%
_{\mathrm{UVD}}$.

Next, we show that nonuniform-$\mathsf{WIRT}_{\mathrm{XYZ}}$ implies
$\mathsf{ACA}_{0}$. By Lemma \ref{lem:DYZ_implies_UYZ}, it suffices to
consider the undirected versions of nonuniform-$\mathsf{WIRT}$. First, we
prove that nonuniform-$\mathsf{WIRT}_{\mathrm{UYS}}$ implies $\mathsf{ACA}%
_{0}$ by constructing a computable instance of nonuniform-$\mathsf{WIRT}%
_{\mathrm{UVS}}$ such that every nonuniform-$\mathsf{WIRT}_{\mathrm{UES}}$
solution computes $\emptyset^{\prime}$. (The desired result follows by
relativization.) We use a variation of the graph used in the analogous result
in Theorem \ref{thm:omega_model_IRT_hyp_closed}.

\textit{Construction of $G=(V,E)$}: $V=\{0^{n}\mid n>0\}\cup\{n\concat 
s \concat 0^{t}\mid n>0$ and some number below $n$ is enumerated into
$\emptyset^{\prime}$ at stage $s$, and either $t\leq s$ or $\emptyset
_{s}^{\prime}\upharpoonright n=\emptyset_{t}^{\prime}\upharpoonright n\}$.
$E=\{(0^{n},0^{n+1})\mid n>0\}\cup\{(n\concat s\concat 0^{t}%
,n\concat s\concat 0^{t+1})\mid n\concat s\concat 0^{t}%
,n\concat s\concat 0^{t+1}\in V\}\cup\allowbreak\{(n\concat  s\concat 0^{t},0)\mid n\concat s\concat 0^{t}\in V$ and
$n\concat s\concat 0^{t+1}\notin V\}$. $G$ is clearly computable.

\textit{Verification:} It is clear that there is exactly one ray
$R_{\left\langle n,s\right\rangle }$ in $G$ beginning with $n\concat s$ for
$n\concat s\in V$ and the sequence $\seq{R_{\seq{n,s}}} _{n\concat s\in V}$ is also computable.
Note that if $\emptyset^{\prime}\upharpoonright n=\emptyset_{s}^{\prime}\upharpoonright n$
then this ray is $\left\langle n\concat s\concat 0^{t}\right\rangle
_{t\in N}$. Otherwise, it is $\left\langle n\concat s\concat %
0^{t}\right\rangle _{n\concat s\concat 0^{t}\in V}\concat \left\langle 0^{n}\right\rangle _{n>0}$. Next observe that for each
$k\in N$, $\{\left\langle n,i\right\rangle \mid i<n\leq k$ and $i\in \emptyset^{\prime
}\}$ is $\Sigma_{1}^{0}$ and contained in $k\times k$ and so is a set by
bounded $\Sigma_{1}^{0}$ comprehension \cite[II.3.9]{sim_book}. Thus the
finite function taking $n>l$ (the least number in $\emptyset^{\prime}$) to the last
stage $s_{n}$ at which an $i<n$ is enumerated in $\emptyset^{\prime}$ is also (coded
by) a finite set. So we have, for each $k$, a sequence of V-disjoint rays
$\seq{R_{n,s_{n}}}_{l<n\leq l+k}$ of length $k$ as required for the hypothesis of
nonuniform-$\mathsf{WIRT}_{\mathrm{UVS}}$.

Suppose then that $S_{i}$ is the sequence of rays in a solution for
nonuniform-$\mathsf{WIRT}_{\mathrm{UES}}$. We wish to compute $\emptyset^{\prime}$
from this solution. As the $S_{i}$ are E-disjoint at most one of them contains
the edge $(0,00)$. So by eliminating that one, we can assume none of the
$S_{i}$ contain $(0,00)$. If any of the remaining rays contain some edge of
the form $(0^{j},0^{j+1})$ for $j>0$ then (as it does not contain $(0,00)$) it
must contain $(0^{k},0^{k+1})$ for every $k\geq j$. Thus there can be at most
one such ray among the remaining $S_{i}$ and so we can discard it and assume
there are no such rays in our list. No remaining ray can have $0$ as its first
vertex as if it did its second vertex would have to be of the form
$n\concat s\concat 0^{t}$ with $n\concat s\concat 0^{t+1}\notin
V$. Any continuation of this sequence would have to follow the $n\concat s\concat 0^{r}$ with $r$ descending from $t$ and so would have to terminate
at $n\concat s$ and not be a ray. Thus all the remaining $S_{i}$ are of the
form $\left\langle n_{i}\concat s_{i}\concat 0^{t+j}\right\rangle _{j\in
N}$ for some $t$ with $n_{i}\neq n_{k}$ for $i\neq k$. So the remaining
$S_{i}$ witness the conclusion of nonuniform-$\mathsf{WIRT}_{\mathrm{UVS}}$ as desired.

As we can replace the first vertex of $S_{i}$ by the sequence beginning with
$n\concat s$ and ending with its second vertex, we know that $\emptyset^{\prime
}\upharpoonright n_{i}=\emptyset_{s}^{\prime}\upharpoonright n_{i}$. Since the
sequence $S_{i}$ and so that of the $n_{i}$ is infinite, given any $m$ we can
find an $n_{i}>m$ and so compute $\emptyset^{\prime}\upharpoonright m$ as $\emptyset_{n_{i}%
}^{\prime}\upharpoonright m$ as required.

To show that nonuniform-$\mathsf{WIRT}_{\mathrm{UYD}}$ implies $\mathsf{ACA}%
_{0}$, define $G$ as above. Consider the graph $G^{\prime}$ gotten by adding
on for each $n\concat s\in V$ new vertices $x_{n,s,k}$ for $k>0$ and edges
$(n\concat s,x_{n,s,1})$ and $(x_{n,s,k},x_{n,s,k+1})$ for $k>0$. The
witnesses for the hypothesis of nonuniform-$\mathsf{WIRT}_{\mathrm{UVS}}$ in
$G$ supply ones for nonuniform-$\mathsf{WIRT}_{\mathrm{UVD}}$ by tacking on
the $x_{n,s,k}$ before $n\concat s$ in reverse order. The witnesses for the
conclusion of nonuniform-$\mathsf{WIRT}_{\mathrm{UED}}$ can be converted into
ones for the conclusion of nonuniform-$\mathsf{WIRT}_{\mathrm{UES}}$ in $G$ by
removing the new vertices. So once again we can compute $\emptyset^{\prime}$.
\end{proof}

We are unable to show that $\mathsf{WIRT}_{\mathrm{XYZ}}$ implies
$\mathsf{ACA}_{0}$, but we can prove the following:

\begin{thm}
\label{thm:WIRT_RCA0} $\mathsf{WIRT}_{\mathrm{XYZ}}$ is not provable in
$\mathsf{RCA}_{0}$.
\end{thm}

\begin{proof}
By Proposition \ref{prop:WIRT_implications_variants}, it suffices to consider
the undirected variants of $\mathsf{WIRT}_{{}}$. For all these variants it
suffices to construct a computable graph $G$ on $\mathbb{N}$ and a computable
sequence $\left\langle \left\langle X_{i}^{k}\right\rangle _{i<k}\right\rangle
_{k\in\mathbb{N}}$ such that (1) for each $k\in\mathbb{N}$, the $X_{i}^{k}$
for $i<k$ are pairwise vertex-disjoint double rays in $G$ and (2) there is no
computable sequence $\left\langle R_{j}\right\rangle _{j\in\mathbb{N}}$ of
edge-disjoint single rays in $G$. It is clear that the $G$ constructed for
this $C$ is a counterexample to each $\mathsf{WIRT}_{\mathrm{UYZ}}$ in the
standard model of $\mathsf{RCA}_{0}$ with second order part the recursive
sets. Of course, as this model is standard, $\mathsf{WIRT}_{\mathrm{XYZ}}$ is
not provable in $\mathsf{RCA}$, $\mathsf{RCA}_{0}$ plus induction for all formulas.

The computable construction will be a finite injury priority argument. At the
end of stage $s$ of our construction, for each $i<k\leq s$, we will have
defined a path $P_{i,s+1}^{k}$ with lengths strictly increasing with $s$ which
is intended to be a segment of the double ray $X_{i}^{k}$. We think of these
paths $P_{s}$ as having domain a segment $[u,v]$ of $\mathbb{Z}$ containing
$[-s,s]$. Its \emph{endpoints} are $P_{s}(u)$ and $P_{s}(v)$. The intention is
that the $X_{i}^{k}=\bigcup_{s}P_{i,s}^{k}$ will be the desired double rays such
that, for each $k$, the $X_{i}^{k}$ for $i<k$ will be vertex-disjoint. We will
also have put all numbers less than $s$ in as vertices in at least one of
these $P_{i,s}^{k}$. In future stages, we will not add any edges between
vertices which are currently in any $P_{i,s}^{k}\ $for $i<k$. Thus $G$ will be
a computable graph given by the union of the double rays $X_{i}^{k}=\bigcup
_{s}P_{i,s}^{k}$. We let $G_{s}$ be the graph defined so far, i.e.\
$\bigcup\{P_{i,s}^{k} \mid i<k<s\}.$ We let $G_{s}^{>t}$ be its subgraph defined as
the union of the $P_{0,s}^{k},\dots,P_{k-1,s}^{k}$ for $k>t$ and similarly for
$k<t$ and other interval notations.

We say that the \emph{disjointness condition, d.c., holds at stage }$s$ if for
every $m<s$ and distinct $n$ and $n^{\prime}<m$, $P_{n,s}^{m}$ and
$P_{n^{\prime},s}^{m}$ are vertex-disjoint. Otherwise we say we have
\emph{violated the d.c.} Clearly, if we never violate the d.c.\ the $X_{n}^{m}$ (for fixed $m$) are pairwise vertex-disjoint. We arrange the construction so that we obviously
never violate the d.c.

So it suffices to also meet the following requirements:
\begin{align*}
Q_{e}:  &  \text{ If }R_{0},R_{1},\dots\text{ is a computable sequence of
single rays in }G\text{ defined by $\Phi_{e}$, }\\
\text{i.e. }\Phi_{e}(i,n)  &  =R_{i}(n)\text{ then the }R_{i}\text{ are not
edge-disjoint.}%
\end{align*}
The requirements $Q_{e}$ are listed in order of priority. During our
construction, if all else fails, we will attempt to satisfy each $Q_{e}$ at
some stage $s$ by \emph{merging} certain rays $X_{i}^{k}$ and $X_{j}^{l}$
\emph{using vertices} $x$ and $y\neq x$ which are endpoints of $P_{i,s}^{k}$
and $P_{j,s}^{l}$, respectively. We do this by adding the least new number $r$
(\emph{the merge point of this merger}) as a vertex of $G$ as well as edges
$(x,r)$ and $(y,r)$ which are appended to each $P_{r,s}^{q}$ with $x$ or $y$,
respectively, as an endpoint. We also ensure that $P_{i,t}^{k}$ and
$P_{j,t}^{l}$ henceforth agree after the vertex $r$ as they grow in the
corresponding directions.

Without loss of generality, and to simplify notation later, we make the
assumption that if $\Phi_{e,s}(i,u)$ is convergent for any $e$, $i$ and $u$
then so is $\Phi_{e,s}(i,u^{\prime})$ for every $u^{\prime}<u$.

\smallskip

\textit{Construction.} At stage $s$ of the construction, we are given a finite
graph $G_{s}$ consisting of, for each $k<s$, finite vertex-disjoint paths
$P_{0,s}^{k},\dots,P_{k-1,s}^{k}$ as described above. We let $f_{s}(e)$ be the
final stage before $s$ at which $Q_{e}$ was initialized. For notational
convenience when $s$ and $e$ are specified we simply write $f$ for $f_{s}(e)$.

First, we act for the requirement $Q_{e}$ of highest priority with $e<s$ which
requires attention as described below. All requirements $Q_{e}$ are
initialized and unsatisfied at stage 0 and are initialized and declared to be
unsatisfied whenever we act for a $Q_{e^{\prime}}$ with $e^{\prime}<e$.

We say that \emph{$Q_{e}$ requires attention at stage $s$} if $Q_{e}$ is not
satisfied and there are $a,u,x,b,v,y<s$ such that $u>0$, $\Phi_{e,s}%
(a,u)\downarrow$, $\left\langle \Phi_{e,s}(a,n)\right\rangle _{n\leq u}$ is
a path in $G_{s}$ disjoint from $G_{s}^{<f}$ which can be extended in only one
way to a maximal path in $G_{s}$ and this extension eventually reaches an
$x\notin G_{s}^{<f}$ which is an endpoint of some $P_{i,s}^{k}$ for $k\geq f$
and similarly for $b,v$ and $y$ for some $P_{j,s}^{l}$ such that $x$ and $y$
are not both endpoints of the same $P_{r}^{q}$. We also require that the
merger using $x$ and $y$ would not violate the d.c. We then let $a(e,s)$ etc.\ be the associated witnesses for the least such computation. In this case, the
actions for $Q_{e}$ is to perform the merger using $x(e,s)$ and $y(e,s)$ as
defined above and declare $Q_{e}$ to be satisfied.

Finally, for each $w$, in turn, which is an endpoint of any the paths
$P_{i}^{k}$ as now defined we extend those paths by taking the least new
number $z$ which we append after $w$ in each of these paths (and so add
$(w,z)$ as a new edge). For each $i<s$, in turn, we also take the next $2s+1$
least new numbers and let the $P_{i}^{s}(n)\ $be these numbers in order for
$n\in\lbrack-s,s]$. This defines the $P_{i,s+1}^{k}$ for $i<k<s+1$ and
completes stage $s$ of the construction. As promised we let $X_{i}^{k}%
=\bigcup_{s}P_{i,s}^{k}$.

\textit{Verification.} It is clear that, for each $i<k\in\mathbb{N}$,
$X_{i}^{k}$ is a double ray and that $G$ is a computable graph consisting of
the union of these rays. It is also clear that by construction we never
violate the d.c.\ and so for each $k$ the $X_{i}^{k}$ for $i<k$ are pairwise
vertex disjoint as required. Thus we only need to prove that we meet each
$Q_{e}$.

We now state a series of facts about $G$ each of which follows immediately (or
by simple inductions) from the construction and previous facts on the list.

\begin{lem}
\label{obs}Every vertex $r$ which is a merge point has exactly three neighbors
and they are the $x$ and $y$ used in the merger and the $z$ added on after $r$
in the final part of the action at the merger stage. In addition, no endpoint
of any $P_{i,s}^{k}$ is a merge point and every vertex which is not a merge
point has exactly two neighbors.

If $u\in G_{s}$ and so $u\in P_{i,s}^{k}$ for some $i<k<s$, then for any
$j<l$, $u\in X_{j}^{l}\Leftrightarrow u\in P_{j,s}^{l}$ and if so $l<s$.

If $(u,v)\in G_{s}$ then not both $u$ and $v$ are merge points. If neither are
merge points then $\forall k\forall i<k(v\in P_{i,s}^{k}\Leftrightarrow u\in
P_{i,s}^{k})$. If one, say $u$, is a merge point $r$ for a merger at some
stage $t$ necessarily less than $s$ using some $x$ and $y$ with $(r,z)$ the
edge added on at the end of stage $t$, then the other ($v$) is $x$, $y$ or
$z$; $\forall k\forall i<k(r\in P_{i,t+1}^{k}\Leftrightarrow x\in P_{i,t}%
^{k}\vee y\in P_{i,t}^{k})$; $\forall k\forall i<k(r\in P_{i,t+1}%
^{k}\Leftrightarrow r\in P_{i,s}^{k}\Leftrightarrow z\in P_{i,t+1}%
^{k}\Leftrightarrow z\in P_{i,s}^{k})$.

Any ray in $G$ which begins in $G_{s}$ remains in $G_{s}$ until it reaches an
endpoint of some $P_{i,s}^{k}$.

Each requirement acts at most once after the last time it is initialized. So,
by induction, each requirement acts and is initialized only finitely often.
Thus there is an infinite sequence $w_{i}$ such that at each $w_{i}$ we act
for some $Q_{e}$ and we never act for any $Q_{e^{\prime}}$ with $e^{\prime
}\leq e$ afterwards. The observation to make here is that if there were a last
stage $w$ at which we act for any $Q_{e}$ then there would be no mergers of
any $X_{i}^{k}$ and $X_{j}^{l}$ for $k,l>w$. In this case all the $X_{i}^{k}$, $i < k$ for $k > w$ would be disjoint. It is then easy to see from the observations
above that at some stage after $w$ we would act for a $Q_{e}$ with $e>w$ such
that $\Phi_{e}(0,n)=X_{i}^{k}(n)$ and $\Phi_{e}(1,n)=X_{j}^{l}(n)$ for
$k,l>w$. 

If $w$ is one of the $w_{i}$ just defined, $z\in G^{<w}$ and $(z,z^{\prime
})\in G$ then $z^{\prime}\notin G^{\geq w}$. The point to notice here is that
by construction no merger at a stage $s>w$ can use any $x\in G^{<w}$ as every
$Q_{e}$ which can act after $w$ is initialized at $w$ by the choice of the
$w_{i}$. Thus any path in $G$ which starts with a $z\in G^{<w}$ never
enters $G^{\geq w}$. $\square$
\end{lem}

Suppose now that $\Phi_{e}$ is total and defines a sequence of edge-disjoint single rays $\left\langle R_{i}\right\rangle _{i\in\mathbb{N}}$ in $G$. By Lemma
\ref{obs}, we may choose a $w=w_{c}$ for some $c$ after which no
$Q_{e^{\prime}}$ for $e^{\prime}\leq e$ ever acts again.

If $Q_{e}$ is not satisfied at the end of stage $w$, we argue that we act for
it later to get a contradiction. Once a ray begins in $G^{<w}$, it must stay
there by Lemma \ref{obs} and once beyond all the (finitely many) merge points
in $G^{<w}$ (none can be put in after stage $w$ and no ray has repeated
vertices) it remains in some $X_{i}^{k}$ for $k<w$ with which it shares a tail
(as each vertex which is not a merge point has exactly two neighbors). So by the edge-disjointness of the $R_{i}$ there is an $a$ such that $\Phi_{e}(a,0)=R_{a}(0)$
is not in $G^{<w}$. It starts in $G^{<w^{\prime}}$ with $w^{\prime
}=w_{c^{\prime}}$ for some $c^{\prime}>c$ and so shares a tail with some
$P_{j}^{k}$ with $w_{c}\leq k<w_{c^{\prime}}$. Similarly there is $b$ such
that $R_{b}$ begins in some $G^{\geq w_{d}}$ with $d>c^{\prime}$ and shares a
tail with some $X_{j}^{l}$ with $w_{d}\leq l$. So eventually we have a stage
$s$ such that $\Phi_{e,s}(a,u)\downarrow$ and $\Phi_{e,s}(b,v)\downarrow$
define paths in $G_{s}$ which go beyond the points by which $R_{a}$ and
$R_{b}$ share tails with $X_{i}^{k}$ and $X_{j}^{l}$, respectively, and after
which neither tail contains a merge point and so these paths have unique
extensions to paths in $G_{s}$ (determined by the appropriate tails of
$X_{i}^{k}$ and $X_{j}^{l}$) ending with the endpoints of $P_{i,s}^{k}$ and
$P_{j,s}^{l}$, respectively. Finally, note that the merger of $X_{i}^{k}$ and
$X_{j}^{l}$ at $s$ would not violate d.c.\ as by Lemma \ref{obs} $P_{i,s}^{k}$
can share vertices only with $P_{n,s}^{m}$ with $w_{c}\leq m<w_{c^{\prime}}$
and $P_{j,s}^{l}$ can share vertices only with $P_{n^{\prime},s}^{m}$ with
$m\geq w_{d}>w_{c^{\prime}}$. Thus at stage $s$ we would act to satisfy
$Q_{e}$ for the desired contradiction.

So $Q_{e}$ is satisfied at $w$ and was satisfied at some $s\leq w$. With the
notations as at $s$, $\Phi_{e}(a,u)$ and $\Phi_{e}(b,v)$ define initial
segments of $R_{a}$ and $R_{b}$ which can each be extended in only one way to
maximal paths in $G_{s}$ eventually reaching the vertices $x$ and $y$,
respectively, as described at $s$. The merger performed at $s$ puts the merge
point $r$ in both $P_{i,s+1}^{k}$ and $P_{l,s+1}^{j}$ and so in $R_{a}$ and
$R_{b}$ (as the successor of $x$ and $y$, respectively, as by Lemma \ref{obs}
$r$ is the only neighbor in $G$ of $x$ other than its predecessor in $R_{a}$
and similarly for $y$ and $R_{b}$). We then add in $(r,z)$ to $G$ at stage $s$. By Lemma \ref{obs}, the only neighbors of $r$ are $x$, $y$ and $z$ so
any continuation of $R_{a}$ and $R_{b}$ after $r$ must produce a shared edge, as $R_{a}$ can continue only with either $(y,r)$ which is an edge of $R_{b}$
or with $(r,z)$ and $R_{b}$ can continue only with $(x,r)$ which is an edge of
$R_{a}$ or $(r,z)$. This yields the final contradiction.
\end{proof}

\section{Open Questions}

\label{section:open}

In addition to the variations of the Halin type theorems investigated here
that remain open problems of graph theory ($\mathsf{IRT}_{\mathrm{DVD}}$ and
$\mathsf{IRT}_{\mathrm{DED}}$) the most intriguing computational and reverse
mathematical questions are about either separating the variants or providing
additional reductions or equivalences among the $\mathsf{IRT}_{\mathrm{XYZ}}$
and $\Sigma_{1}^{1}$-$\AC_{0}$. Clearly the most important issue is deciding if
any (or even all) the $\mathsf{IRT}_{\mathrm{XYZ}}$ which are known to be
provable in $\Sigma_{1}^{1}$-$\AC_{0}$ are actually equivalent to it. We extend
this problem to include the $\mathsf{IRT}_{\mathrm{XYZ}}^{\ast}$ and
$\mathsf{I}\Sigma_{1}^{1}$.

\begin{question}
\label{q1}Can one show that any of the $\mathsf{IRT}_{\mathrm{XYZ}}$ which are
provable in $\Sigma_{1}^{1}\text{-}\mathsf{AC}_{0}$ ($\mathsf{IRT}%
_{\mathrm{XYS}}$ and $\mathsf{IRT}_{\mathrm{UVD}}$) do not imply $\Sigma
_{1}^{1}\text{-}\mathsf{AC}_{0}$ over $\mathsf{RCA}_{0}$ or even over
$\mathsf{RCA}_{0}+\mathsf{I}\Sigma_{1}^{1}$? An intermediate result might be that
$\mathsf{IRT}_{\mathrm{XYZ}}^{\ast}$ (for one of these versions) does not
imply $\Sigma_{1}^{1}\text{-}\mathsf{AC}_{0}$ over $\mathsf{RCA}_{0}$.
\end{question}

Should any of these $\mathsf{IRT}_{\mathrm{XYZ}}$ be strictly weaker than
$\Sigma_{1}^{1}$-$\mathsf{AC}_{0}$, the question would then be to determine the
relations among the $\mathsf{IRT}_{\mathrm{XYZ}}$ and analogously the
$\mathsf{IRT}_{\mathrm{XYZ}}^{\ast}$.

\begin{question}
Can any additional arrows be added to Figure \ref{fig:variants_IRT} over
$\mathsf{RCA}_{0}$ or $\mathsf{RCA}_{0}+\mathsf{I}\Sigma_{1}^{1}$? (This includes the
question of whether $\mathsf{RCA}_{0}\vdash\mathsf{IRT}_{\mathrm{UVD}%
}\rightarrow\mathsf{IRT}_{\mathrm{UVS}}$.)
\end{question}

As we noted in Remark \ref{rmk:BCP} there is an apparent additional reduction
in Bowler, Carmesin, Pott \cite[pg.\ 2 l.\ 3--7]{bcp15}. They use an
intermediate reduction to locally finite graphs in the sense of relying on the
fact that if a graph has arbitrarily many disjoint rays it has a locally
finite subgraph with arbitrarily many disjoint rays. This is the principle to
which that Remark refers. It plus $\mathsf{ACA}_{0}$ is a THA but over
$\mathsf{RCA}_{0}$ it does not imply $\mathsf{ACA}_{0}$ and is provably very
weak (in the sense of being highly conservative over $\mathsf{RCA}_{0}$).
Shore \cite{shore_atha} proves these results and further analyzes this and
many similar principles some related to the $\mathsf{IRT}_{\mathrm{XYZ}}$ and
others to an array of classical logical principles.

Any reductions in $\mathsf{RCA}_{0}$ as requested in the Question above would,
of course, provide the analogous ones for the $\mathsf{IRT}^{\ast
}_{\mathrm{XYZ}}$. However, it is possible that other implications can be
proven for the $\mathsf{IRT}^{\ast}_{\mathrm{XYZ}}$:

\begin{question}
Can any implications of the form $\mathsf{IRT}^{\ast}_{\mathrm{XYZ}%
}\rightarrow~\mathsf{IRT}^{\ast}_{\mathrm{X^{\prime}Y^{\prime}Z^{\prime}}}$ be
proven in $\mathsf{RCA}_{0}$ other than the ones known to hold for the
$\mathsf{IRT}_{}$ versions?
\end{question}

Probably more challenging is the problem of separating the principles.

\begin{question}
Can one prove any nonimplication over $\mathsf{RCA}_{0}$ or over
$\mathsf{RCA}_{0} + \mathsf{I}\Sigma^{1}_{1}$ for any pair of the $\mathsf{IRT}%
_{\mathrm{XYZ}}$?
\end{question}

Of course, any such separation for the $\mathsf{IRT}_{\mathrm{XYZ}}$ of
Question \ref{q1} would answer a case of that question by proving that at
least one of these principles is strictly weaker than $\Sigma_{1}^{1}$%
-$\mathsf{AC}_{0}$. In addition, a separation by standard models or even ones over
$\mathsf{I}\Sigma_{1}^{1}$ for the $\mathsf{IRT}_{\mathrm{XYZ}}$ would give
nonimplication for the corresponding $\mathsf{IRT}_{\mathrm{XYZ}}^{\ast}$ but
it might be that nonstandard models could be used to separate one pair of
versions but not the other.

The next natural question looks below $\mathsf{ABW}_{0}$ in Figure
\ref{fig:hyp_analysis_zoo_IRT}.

\begin{question}
Can one prove that finite-$\Sigma^{1}_{1}\text{-}\mathsf{AC}_{0}$ does not
imply $\mathsf{ABW}_{0}$ over $\mathsf{RCA}_{0}$ or $\mathsf{RCA}_{0} + \mathsf{I}\Sigma^{1}_{1}$?
\end{question}

The weaker versions, $\mathsf{WIRT}_{\mathrm{XYZ}}$, of the $\mathsf{IRT}%
_{\mathrm{XYZ}}$, prompt a question about $\mathsf{ACA}_{0}$.

\begin{question}
Do any of the $\mathsf{WIRT}_{\mathrm{XYZ}}$ (especially the ones provable
from $\mathsf{ACA}_{0}$) imply $\mathsf{ACA}_{0}$? An easier question might be
whether they imply $\mathsf{WKL}_{0}$?
\end{question}

\bibliographystyle{plain}
\bibliography{irt_references}

\begin{thebibliography}{10}

\bibitem{a84}
Ron Aharoni.
\newblock K\"{o}nig's duality theorem for infinite bipartite graphs.
\newblock {\em J. London Math. Soc. (2)}, 29(1):1--12, 1984.

\bibitem{ams92}
Ron Aharoni, Menachem Magidor, and Richard~A. Shore.
\newblock On the strength of {K}\"onig's duality theorem for infinite bipartite
  graphs.
\newblock {\em J. Combin. Theory Ser. B}, 54(2):257--290, 1992.

\bibitem{bcp15}
Nathan Bowler, Johannes Carmesin, and Julian Pott.
\newblock Edge-disjoint double rays in infinite graphs: a {H}alin type result.
\newblock {\em J. Combin. Theory Ser. B}, 111:1--16, 2015.

\bibitem{conidis12}
Chris~J. Conidis.
\newblock Comparing theorems of hyperarithmetic analysis with the arithmetic
  {B}olzano-{W}eierstrass theorem.
\newblock {\em Trans. Amer. Math. Soc.}, 364(9):4465--4494, 2012.

\bibitem{diestel_book}
Reinhard Diestel.
\newblock {\em Graph theory}, volume 173 of {\em Graduate Texts in
  Mathematics}.
\newblock Springer, Berlin, fifth edition, 2017.

\bibitem{handbook}
Yu.~L. Ershov, S.~S. Goncharov, A.~Nerode, J.~B. Remmel, and V.~W. Marek,
  editors.
\newblock {\em Handbook of recursive mathematics. {V}ol. 2}, volume 139 of {\em
  Studies in Logic and the Foundations of Mathematics}.
\newblock North-Holland, Amsterdam, 1998.
\newblock Recursive algebra, analysis and combinatorics.

\bibitem{friedman_LC}
Harvey Friedman.
\newblock Algorithmic procedures, generalized {T}uring algorithms, and
  elementary recursion theory.
\newblock In {\em Logic {C}olloquium '69 ({P}roc. {S}ummer {S}chool and
  {C}olloq., {M}anchester, 1969)}, pages 361--389, 1971.

\bibitem{friedman_icm}
Harvey Friedman.
\newblock Some systems of second order arithmetic and their use.
\newblock {\em Proceedings of the {I}nternational {C}ongress of
  {M}athematicians ({V}ancouver, {B}. {C}., 1974), {V}ol. 1}, pages 235--242,
  1975.

\bibitem{friedman_thesis}
Harvey~M. Friedman.
\newblock {\em Subsystems of set theory and analysis}.
\newblock 1967.
\newblock Thesis (Ph. D.)--Massachusetts Institute of Technology, Dept. of
  Mathematics.

\bibitem{goh_finite_choice}
Jun~Le Goh.
\newblock The strength of an axiom of finite choice for branches in trees.
\newblock To appear in the Journal of Symbolic Logic.

\bibitem{halin65}
Rudolf Halin.
\newblock \"{U}ber die {M}aximalzahl fremder unendlicher {W}ege in {G}raphen.
\newblock {\em Math. Nachr.}, 30:63--85, 1965.

\bibitem{halin70}
Rudolf Halin.
\newblock Die {M}aximalzahl fremder zweiseitig unendlicher {W}ege in {G}raphen.
\newblock {\em Mathematische Nachrichten}, 44(1‐6):119--127, 1970.

\bibitem{hirschfeldt}
Denis~R. Hirschfeldt.
\newblock {\em Slicing the truth}, volume~28 of {\em Lecture Notes Series.
  Institute for Mathematical Sciences. National University of Singapore}.
\newblock World Scientific Publishing Co. Pte. Ltd., Hackensack, NJ, 2015.
\newblock On the computable and reverse mathematics of combinatorial
  principles, Edited and with a foreword by Chitat Chong, Qi Feng, Theodore A.
  Slaman, W. Hugh Woodin and Yue Yang.

\bibitem{jullien_thesis}
Pierre Jullien.
\newblock {\em Contribution \`a l'\'etude des types d'ordres dispers\'es}.
\newblock 1968.
\newblock Thesis (Ph.D.)--Marseille.

\bibitem{kreisel}
Georg Kreisel.
\newblock The axiom of choice and the class of hyperarithmetic functions.
\newblock {\em Nederl. Akad. Wetensch. Proc. Ser. A 65 = Indag. Math.},
  24:307--319, 1962.

\bibitem{montalban_jullien}
Antonio Montalb\'{a}n.
\newblock Indecomposable linear orderings and hyperarithmetic analysis.
\newblock {\em J. Math. Log.}, 6(1):89--120, 2006.

\bibitem{montalban_pi11sep}
Antonio Montalb\'{a}n.
\newblock On the {$\Pi_1^1$}-separation principle.
\newblock {\em MLQ Math. Log. Q.}, 54(6):563--578, 2008.

\bibitem{montalban_open}
Antonio Montalb\'{a}n.
\newblock Open questions in reverse mathematics.
\newblock {\em Bull. Symbolic Logic}, 17(3):431--454, 2011.

\bibitem{neeman08}
Itay Neeman.
\newblock The strength of {J}ullien's indecomposability theorem.
\newblock {\em J. Math. Log.}, 8(1):93--119, 2008.

\bibitem{neeman11}
Itay Neeman.
\newblock Necessary use of {$\Sigma^1_1$} induction in a reversal.
\newblock {\em J. Symbolic Logic}, 76(2):561--574, 2011.

\bibitem{ps76}
Klaus-Peter Podewski and Karsten Steffens.
\newblock Injective choice functions for countable families.
\newblock {\em J. Combinatorial Theory Ser. B}, 21(1):40--46, 1976.

\bibitem{rogers}
Hartley Rogers, Jr.
\newblock {\em Theory of recursive functions and effective computability}.
\newblock MIT Press, Cambridge, MA, second edition, 1987.

\bibitem{rosenstein}
Joseph~G. Rosenstein.
\newblock {\em Linear orderings}, volume~98 of {\em Pure and Applied
  Mathematics}.
\newblock Academic Press, Inc. [Harcourt Brace Jovanovich, Publishers], New
  York-London, 1982.

\bibitem{sacks}
Gerald~E. Sacks.
\newblock {\em Higher recursion theory}.
\newblock Perspectives in Mathematical Logic. Springer-Verlag, Berlin, 1990.

\bibitem{shore93}
Richard~A. Shore.
\newblock On the strength of {F}ra\"\i ss\'e's conjecture.
\newblock In {\em Logical methods ({I}thaca, {NY}, 1992)}, volume~12 of {\em
  Progr. Comput. Sci. Appl. Logic}, pages 782--813. Birkh\"auser Boston,
  Boston, MA, 1993.

\bibitem{shore_atha}
Richard~A. Shore.
\newblock Almost theorems of hyperarithmetic analysis.
\newblock {\em J. Symbolic Logic}, 88(2):664--696, 2023.

\bibitem{sim94}
Stephen~G. Simpson.
\newblock On the strength of {K}\"{o}nig's duality theorem for countable
  bipartite graphs.
\newblock {\em J. Symbolic Logic}, 59(1):113--123, 1994.

\bibitem{sim_book}
Stephen~G. Simpson.
\newblock {\em Subsystems of second order arithmetic}.
\newblock Perspectives in Logic. Cambridge University Press, Cambridge;
  Association for Symbolic Logic, Poughkeepsie, NY, second edition, 2009.

\bibitem{steel78}
John~R. Steel.
\newblock Forcing with tagged trees.
\newblock {\em Ann. Math. Logic}, 15(1):55--74, 1978.

\bibitem{vw_thesis}
Robert~Alan Van~Wesep.
\newblock {\em Subsystems of Second-order Arithmetic, and Descriptive Set
  Theory under the Axiom of Determinateness}.
\newblock ProQuest LLC, Ann Arbor, MI, 1977.
\newblock Thesis (Ph.D.)--University of California, Berkeley.

\end{thebibliography}

\end{document}